\documentclass{m2an}
\usepackage{amsmath,amsfonts,amssymb,mathtools}
\usepackage{graphicx,subcaption}
\usepackage{braket,dsfont,url,enumerate}
\usepackage{tikz}
\usepackage{pdfpages} % to import PDF pages
\usepackage{cmap}
\usepackage{hyperref}
\usepackage{color,braket}
\usepackage{bbm}
\usepackage{fancyhdr}
\usepackage{times}
\usepackage{cmap}
\usepackage{pgf}
\usepackage{comment}
\usepackage{algorithm}
\usepackage{algorithmic}
\usepackage[textsize=small]{todonotes}
\usepackage[noabbrev]{cleveref}
\makeatletter
\@namedef{subjclassname@1991}{2010 Mathematics Subject Classification}
\makeatother
\definecolor{darkred}{rgb}{.7,0,0}
\definecolor{RED}{rgb}{1,0,0}

\newcommand\blue[1]{\textcolor{black}{#1}}
\definecolor{green}{rgb}{0,0.7,0}
\newtheorem{theorem}{Theorem}[section]
\newtheorem{lemma}[theorem]{Lemma}
\newtheorem{corollary}[theorem]{Corollary}
\newtheorem{proposition}[theorem]{Proposition}
\newtheorem{remark}[theorem]{Remark}
\newtheorem{assumption}[theorem]{Assumption}

\newcommand{\throwout}[1]{}

\usepackage{accents}

\newcommand{\R}{\mathbb{R}}
\newcommand{\N}{\mathbb{N}}
\newcommand{\IOprod}[1]{\left( #1 \right)_{I\times\Omega}}

\newcommand{\IOpair}[1]{\left\langle #1 \right\rangle_{I\times\Omega}}
\newcommand{\ImOprod}[1]{\left( #1 \right)_{I_m\times\Omega}}
\newcommand{\IjOprod}[2]{\left( #1 \right)_{I_{#2}\times\Omega}}

\newcommand{\Oprod}[1]{\left( #1 \right)_{\Omega}}
\newcommand{\Opair}[1]{\left\langle #1 \right\rangle_{\Omega}}

\newcommand{\chatO}[1]{\hat c\left( #1 \right)}
\newcommand{\chatIO}[1]{\hat c \left(\!\left( #1 \right)\!\right)}
\newcommand{\cO}[1]{c\left( #1 \right)}
\newcommand{\cIO}[1]{c \left(\!\left( #1 \right)\!\right)}

\newcommand{\uukh}{\overline{uu_{kh}}}
\newcommand{\Ukh}{U_{kh}}
\newcommand{\Vkh}{V_{kh}}
\newcommand{\Mkh}{M_{kh}}
\newcommand{\Xkh}{Y_{kh}}
\newcommand{\ukh}{u_{kh}}
\newcommand{\zkh}{z_{kh}}
\newcommand{\ukht}{\tilde u_{kh}}
\newcommand{\pkht}{\tilde p_{kh}}
\newcommand{\vkh}{v_{kh}}
\newcommand{\vk}{v_{k}}
\newcommand{\wkh}{w_{kh}}

\newcommand{\ekh}{e_{kh}}
\newcommand{\pkh}{p_{kh}}
\newcommand{\qkh}{q_{kh}}
\newcommand{\rhokh}{\varrho_{kh}}
\newcommand{\phikh}{\phi_{kh}}
\newcommand{\psikh}{\psi_{kh}}

\newcommand{\Linftwo}{L^\infty(I;L^2(\Omega)^2)}
\newcommand{\LinfLtwo}{L^\infty(I;L^2(\Omega))}
\newcommand{\LinfHone}{L^\infty(I;H^1(\Omega))}

\newcommand{\HmOd}{H^{-1}(\Omega)^2}
\newcommand{\Hone}{H^{1}(\Omega)}
\newcommand{\LtwoHone}{L^2(I;H^1(\Omega))}
\newcommand{\LtwoHoned}{L^2(I;H^1(\Omega)^2)}
\newcommand{\LtwoHtwo}{L^2(I;H^2(\Omega))}
\newcommand{\LtwoLtwo}{L^2(I \times \Omega)}
\newcommand{\LoneLtwo}{L^1(I ;L^2(\Omega))}

\newcommand{\Hozs}{H^1_0(\Omega)^2}
\newcommand{\LtwomLtwo}{L^2(I_m \times \Omega)}
\newcommand{\LtwoV}{L^2(I;V)}

\newcommand{\dualrhs}{g}

\newcommand{\OseenConstant}{K\left(\|\uukh\|_{L^2(I;H^1(\Omega))\cap L^\infty(I;L^2(\Omega))}\right)}
\newcommand{\MathSet}[1]{\left\lbrace \, #1 \, \right\rbrace}

\newcommand{\Ph}{\mathbb P_h}

\DeclareMathOperator{\supp}{supp}

\begin{document}
\author{Boris Vexler}\address{Chair of Optimal Control, Technical University of Munich,
  School of Computation Information and Technology,
Department of Mathematics , Boltzmannstra{\ss}e 3, 85748 Garching b. Munich, Germany
(vexler@tum.de).}
\author{Jakob Wagner}\address{Chair of Optimal Control, Technical University of Munich,
  School of Computation Information and Technology,
Department of Mathematics , Boltzmannstra{\ss}e 3, 85748 Garching b. Munich, Germany
(wagnerja@cit.tum.de) ORCID: \url{https://orcid.org/0000-0001-8510-9790}.}

\begin{abstract}
  % This is the abstract for the paper
In this work we consider the two dimensional instationary Navier-Stokes equations
with homogeneous Dirichlet/no-slip boundary conditions.
We show error estimates for the fully discrete problem, where a discontinuous Galerkin method in time
and inf-sup stable finite elements in space are used.
% Recently, a best approximation type estimate for the error in the $L^\infty(I;L^2(\Omega))$ norm
% for the Stokes problem has been shown.
% The present work extends this error estimate to the Navier-Stokes equations, by
% pursuing an error splitting approach.
% The necessary regularity of solutions to the Navier-Stokes equations is analyzed 
% to estimate the first part of the error.
% The second part, facilitated by the variational formulation of the dG time discretization,
% is treated by an appropriate duality argument.
% A specialized Gronwall lemma is presented, which allows the analysis of the discrete
% dual equation.
% The techniques developed in this work allow us to show also best approximation type error estimates 
% in the $L^2(I;H^1(\Omega))$ and $L^2(I;L^2(\Omega))$ norms, which we present.
% 
% 
% 
Recently, best approximation type error estimates for the Stokes problem in the $L^\infty(I;L^2(\Omega))$, 
$L^2(I;H^1(\Omega))$ and $L^2(I;L^2(\Omega))$ norms have been shown.
The main result of the present work extends the error estimate in the $L^\infty(I;L^2(\Omega))$ norm 
to the Navier-Stokes equations, by pursuing an error splitting approach and an appropriate duality argument.
In order to discuss the stability of solutions to the discrete primal and dual equations,
a \blue{specially tailored} discrete Gronwall lemma is presented.
The techniques developed towards showing the $L^\infty(I;L^2(\Omega))$
error estimate, also allow us to show best approximation 
type error estimates in the $L^2(I;H^1(\Omega))$ and $L^2(I;L^2(\Omega))$ norms, which complement this work.
% which we present.
% and 
% A specialized Gronwall lemma is presented, which allows the analysis of the discrete
% dual equation.
% The techniques developed in this work allow us to show also best approximation type error estimates 
% in the $L^2(I;H^1(\Omega))$ and $L^2(I;L^2(\Omega))$ norms, which we present.

\end{abstract}

\title{Error estimates for finite element discretizations of the
  instationary Navier-Stokes equations}


\maketitle

\section{Introduction}
In this paper, we consider the instationary Navier-Stokes equations in two space dimensions
with homogeneous boundary conditions, i.e.,
\begin{equation}\label{eq:nav_stokes_classical}
  \left\lbrace
\begin{aligned}
  \partial_t u - \nu \Delta u + (u \cdot \nabla)u + \nabla p & = f && \text{in } I \times \Omega,\\
  \nabla \cdot u & = 0 && \text{in } I \times \Omega,\\
  u(0) & = u_0 && \text{in } \Omega,\\
  u & = 0 && \text{on }I \times \partial \Omega.
\end{aligned}
\right.
\end{equation}
Here $\nu >0$ denotes the viscosity, $I = (0,T] \subset \R$ a bounded, half-open interval for some fixed
finite endtime $T>0$, and $\Omega \subset \R^2$ a bounded convex polygonal domain.
The equations are discretized in time by a
discontinuous Galerkin (dG) method, i.e., the solution is approximated
by piecewise polynomials in time, defined on subintervals of $I$, without any
requirement of continuity at the time nodes, see, e.g.,
\cite{eriksson_time_1985,thomee_galerkin_2006}.
The parameter indicating the time discretization will be denoted by $k$ and corresponds
to the length of the largest subinterval in the partition of $I$.
To discretize in space, we use
inf-sup stable pairs of finite element spaces for the velocity and pressure components.
The parameter indicating the spatial discretization will
be denoted by $h$ and corresponds to the largest diameter of cells in the mesh.
\\
Due to the variational formulation of the dG time discretization, this 
discretization scheme is particularly suited for treating optimal control problems.
See, e.g., \cite{casas_discontinuous_2012,casas_analysis_2016}
for optimal control problems governed by the Navier-Stokes equations, 
also \cite{meidner_priori_2011,leykekhman_numerical_2020}
for optimal control of general parabolic problems.
While in \cite{casas_discontinuous_2012,casas_analysis_2016} the focus was put on low order schemes,
recently the authors of \cite{ahmed_higherorder_2021} analyzed dG schemes of arbitrary order for the 
Navier-Stokes equations.
Another advantage of the dG time discretization is the fact, that the maximal parabolic regularity,
exhibited by parabolic problems,
is preserved on the discrete level, and moreover can be extended to the limiting 
cases $L^1$ and $L^\infty$ in time, at the expense of a logarithmic factor,
see \cite[Theorems 11 \& 12]{leykekhman_discrete_2017}.
The natural energy norm for the Navier-Stokes equations is the norm of the space
$\Linftwo \, \cap \, \LtwoHoned$.
Indeed by formally testing \eqref{eq:nav_stokes_classical} with the solution $u$, one obtains
\begin{equation*}
  \|u\|_{L^\infty(I;L^2(\Omega))} + \|u\|_{L^2(I;H^1(\Omega))} 
  \le C\left(\|u_0\|_{L^2(\Omega)} + \|f\|_{L^1(I;L^2(\Omega))} \right).
\end{equation*}
This bound is also preserved on the discrete level, i.e., holds for the fully discrete solution $\ukh$,
see \Cref{thm:stability_discrete_navier_stokes}.
Our main goal in writing this paper, was the investigation of 
the discretization error in terms of the $L^\infty(I;L^2(\Omega))$-norm.
In particular, such estimates are required for the
analysis of optimal control problems, subject to state constraints pointwise in time,
see \cite{meidner_priori_2011}, as the corresponding Lagrange multiplier in this case is a measure in time.
There have been numerous approaches in the literature,
deriving error estimates for the Navier-Stokes equations.
In \cite[Theorem 4.7]{casas_discontinuous_2012},
%\todo{Not Chrysafinos, Walkington? -> No, there only nodewise}
the same combination of dG-cG discretization schemes was used,
and for $f \in L^2(I;L^2(\Omega)^2)$ an estimate
\begin{equation}\label{eq:errest_casas}
  \|u - \ukh\|_{L^\infty(I;L^2(\Omega))} \le C ( \sqrt{k} + h)
\end{equation}
was shown. 
Additional terms arise, when changes in the spatial mesh on different time intervals are permitted.
Under the much stricter assumption $f \in W^{1,\infty}(I;L^2(\Omega)^2)$ and for
the implicit Euler time discretization, in 
\cite{heywood_finite_1986} the estimate
\begin{equation}\label{eq:errest_heywood}
  \|u - \ukh\|_{L^\infty(I;L^2(\Omega))} \le C ( k + h^2)
\end{equation}
was shown. It was extended to the Crank-Nicholson scheme in \cite{Heywood1990} under the 
assumption of $f \in W^{2,\infty}(I;L^2(\Omega)^2)$,
yielding an error estimate at the time nodes of order $\mathcal O(k^2+h^2)$.
% \todo{Include this?}
% \textcolor{red}{
%   In the recent contribution of \cite{bajpai_priori_2023}, the authors show the same orders of convergence
%   stated in \eqref{eq:errest_heywood} for dG discretizations in space and the implicit Euler timestepping 
%   scheme.
%   Their results are derived under the assumptions $u \in L^\infty(I;H^2(\Omega))$ and 
%   $\partial_t u \in L^2(I;H^2(\Omega)^2)\cap L^\infty(I;H^1(\Omega)^2)$.
%   A discussion of the corresponding required regularity assumptions on the data is omitted.
% }
Beyond the references mentioned above, there are many works analyzing the stationary equations,
  e.g., \cite{girault_finite_1986,schieweck_optimal_1996}, 
  semidiscrete equations in space, e.g., \cite{bernardi_conforming_1985, deckelnick_semidiscretization_2004},
  or semidiscrete equations in time, e.g., \cite{emmrich_error_2004, shen_error_1992,arndt_local_2015}.
  Fully discrete error estimates for stabilized discretization schemes can be found e.g., in
  \cite{ahmed_higherorder_2021,kirk_analysis_2022,de_frutos_error_2019,arndt_quasi_optimal_2016}.
  The error estimates are most often derived assuming the necessary regularity of $u$, such that the 
  discretization schemes can exhibit their full approximative power.
The main drawback of the results in the literature is the fact, that usually the errors in 
$L^\infty(I;L^2(\Omega))$ and $L^2(I;H^1(\Omega))$ are estimated 
in a combined fashion. Thus the estimate always has to account for the spatial error in the $H^1$ norm,
yielding an order reduction for the estimate of the error in the spatial $L^2$ norm.
The goal of this paper is to prove an error estimate which can be formulated as a best approximation
type error estimate, which thus estimates the error in the $L^\infty(I;L^2(\Omega))$ norm in an isolated manner.
Such an estimate for the instationary Stokes equations, has been derived 
recently in \cite{behringer_fully_2022}. More specifically, 
for velocity and pressure fields $(w,r)$, solving the instationary Stokes equations on the continuous level,
and their fully discrete approximations $(w_{kh},r_{kh})$, it holds
\begin{equation*}
  \|w - w_{kh}\|_{L^\infty(I;L^2(\Omega))} \le C 
  \ln \left(\frac{T}{k}\right) \left( \inf_{\blue{\chi_{kh}} \in \Vkh} 
    \|w - \blue{\chi_{kh}}\|_{L^\infty(I;L^2(\Omega))}
  + \|w- R^S_h(w,r)\|_{L^\infty(I;L^2(\Omega))}\right),
\end{equation*}
where $\Vkh$ is the space of discretely divergence free space-time finite element functions.
The operator $R^S_h$ denotes a stationary Stokes projection, see \cite{behringer_fully_2022} or 
\Cref{sec:ErrEst} for a formal definition.
Using this Stokes result, we will show that a best approximation type
result also holds for the nonlinear Navier-Stokes equations.
This main result is stated in \Cref{thm:nav_stokes_bestapprox}.
In \Cref{corr:err_est_nav_stokes_spec_orders}, for $f \in L^\infty(I;L^2(\Omega)^2)$,
we then obtain an estimate in terms of $\mathcal O(l_{k} (k + h^2))$, where
$l_{k}$ denotes a logarithmic term depending on $k$.
This result provides a better order of convergence compared to the estimate \eqref{eq:errest_casas},
shown in \cite{casas_discontinuous_2012}.
The order of convergence in \blue{the estimate} \eqref{eq:errest_heywood}, \blue{that was presented in}
\cite{heywood_finite_1986} is comparable,
but requires a much stronger regularity assumption.
% Comparing with the specific orders from the literature,
% for the case of a convex domain $\Omega$ and $f\in L^\infty(I;L^2(\Omega)^2)$,
% in \Cref{corr:err_est_nav_stokes_spec_orders}
% we then obtain an estimate in terms of $\mathcal O(l_{kh} (k + h^2))$ , where
% $l_{kh}$ denotes a logarithmic term depending on $k$ and $h$.\\
The main tools used in this paper for proving the proposed
$\LinfLtwo$ error estimate are an error splitting approach
and a bootstrapping argument, to apply the corresponding error estimate for the Stokes equations
to the first part of the error. In order to apply such an argument, understanding the precise 
regularity of the occuring nonlinear term $(u \cdot \nabla) u$ is crucial. 
The second part of the error will be estimated by a duality argument. This is possible due to the 
variational nature of the dG time discretization. We will derive a stability 
result for a discrete dual equation. For this result, a specially adapted version of a 
discrete Gronwall lemma, \Cref{thm:discrete_gronwall_quadlinear}, will be presented.
For the analysis of the discrete dual problem, we require an estimate in the $L^2(I;H^1(\Omega))$ norm,
which is why we formulate this result first. We shall then show the error estimate in 
\blue{$L^\infty(I;L^2(\Omega))$}
with the tools presented above. The proof of the error estimate in the $L^2(I;L^2(\Omega))$ norm 
is then straightforward and concludes our work. Summarizing, our main results read
\begin{align*}
  \blue{\|u-\ukh\|_{L^2(I;H^1(\Omega))}} &\le C \left(\sqrt{k}+h\right) && 
  \text{if } f \in L^2(I;L^2(\Omega)^2) &&\hspace{-12mm}\text{ and } u_0 \in V,\\
  \|u-\ukh\|_{L^\infty(I;L^2(\Omega))} &\le C \ln(T/k)^2\left(k+h^2\right) && 
  \text{if } f \in L^\infty(I;L^2(\Omega)^2) &&\hspace{-12mm}\text{ and } u_0 \in V\cap H^2(\Omega)^2,\\
  \|u-\ukh\|_{L^2(I;L^2(\Omega))} &\le C \left(k+h^2\right) && 
  \text{if } f \in L^2(I;L^2(\Omega)^2) &&\hspace{-12mm}\text{ and } u_0 \in V,
\end{align*}
and can be found in \Cref{thm:nav_stokes_l2h1,thm:nav_stokes_bestapprox,thm:nav_stokes_l2l2} and
\Cref{corr:err_est_nav_stokes_spec_orders,corr:l2h1_orders}.
All three results are, up to logarithmic terms, optimal in terms of order of convergence and in terms 
of required regularity.
The structure of this paper will be as follows.
First, we fix some notation and function spaces in \Cref{sec:preliminary}.
We proceed in \Cref{sec:NSE} by stating the appropriate weak formulations of
\eqref{eq:nav_stokes_classical}
with and without pressure and recall some known regularity results.
We conclude the section with an analysis of the regularity properties of 
the nonlinear term $(u \cdot \nabla) u$.
\Cref{sec:Discretization} will be devoted to the spatial and temporal discretizations.
\blue{We present a version of a discrete Gronwall lemma, which allows us to treat problems with
right hand sides that are only $L^1$ in time, and show stability results for fully discrete primal equations.
With this result, we derive an error estimate for the Navier-Stokes equations in $L^2(I;H^1(\Omega))$.}
Lastly, in \Cref{sec:ErrEst},
\blue{we show stability results for a discrete dual problem,}
%With this, we can apply improved stability results for the discrete Navier-Stokes 
%equations and a discrete dual problem, 
which in the end allows us to show the error estimates in 
the $L^\infty(I;L^2(\Omega))$ and $L^2(I;L^2(\Omega))$ norms.
% In the end, we recall all results that were formulated in terms of bestapproximation errors, and state
% explicit orders of convergence. For the three considered norms, we obtain the following results
% \begin{equation*}
%   \|u-\ukh\|_{L^2(I;H^1(\Omega))} \le \ln \left(\frac{T}{k}\right)^2 C(k^{1/2}+h)
%   \|u-\ukh\|_{L^\infty(I;L^2(\Omega))} \le C (k+h^2)
%   \|u-\ukh\|_{L^2(I;L^2(\Omega))} \le C (k+h^2)
% \end{equation*}

\section{Preliminary}\label{sec:preliminary}
For a \blue{convex,} polygonal domain $\Omega \subset \R^2$, $1 \le p\le \infty$ and $k \in \mathbb N$,
we denote by $L^p(\Omega)$, $W^{k,p}(\Omega)$, $H^k(\Omega)$ and $H^1_0(\Omega)$
the usual Lebesgue and Sobolev spaces.
The inner product on $L^2(\Omega)$ will be denoted by $\Oprod{\cdot , \cdot}$.
The space $L^2_0(\Omega)$ is the subspace of $L^2(\Omega)$, consisting of all functions,
that have zero mean.
% By $W^{k,p}(\Omega)$ for $k\in \N$ we denote the Sobolev space of functions with weak derivatives 
% up to order $k$, that are $p-$integrable. 
% For $p=2$ we employ the usual notation $H^k(\Omega)$, and by $H^1_0(\Omega)$ we denote the space
%of those functions in $H^1(\Omega)$ which have zero trace.
For $s \in \R\backslash \N$, $s>0$ the fractional order Sobolev(-Slobodeckij)
space $W^{s,p}(\Omega)$ is defined, see, e.g., \cite{demengel_functional_2012}, as
\begin{equation*}
  W^{s,p}(\Omega) := \MathSet{ v \in W^{\lfloor s\rfloor,p}(\Omega) : 
  \sum_{|\alpha|=\lfloor s \rfloor} \iint_{\Omega \times \Omega} 
  \dfrac{|D^\alpha v(x) - D^\alpha v(y)|^p}{|x-y|^{(s-\lfloor s\rfloor)p+2}}
  dx dy < + \infty}.
\end{equation*}
In case $p=2$ we again use the notation $H^s(\Omega)$. Note that in this case $H^s(\Omega)$ 
can equivalently be obtained via real or complex
interpolation of the integer degree spaces $H^k(\Omega)$.
This is due to the fact, that in the Hilbert space setting, all resulting Bessel potential spaces
$H^s_2(\Omega)$, Besov spaces $B^s_{2,2}(\Omega)$ and Sobolev-(Slobodeckij) spaces $H^s(\Omega)$ coincide,
see \cite[pp. 12,39]{triebel_theory_1992}.
For $X$ being any function space over $\Omega$, we denote by $X^*$ its topological dual space,
and abbreviate the duality pairing by $\Opair{\cdot , \cdot}$. We will also use the notation
$H^1_0(\Omega)^* = H^{-1}(\Omega)$.
The structure of the Stokes and Navier-Stokes equations requires also the definition of 
some vector valued spaces, consisting of divergence free vector fields.
We denote by $\nabla \cdot$ the divergence operator and introduce the spaces
\begin{equation*}
  V:= \overline{\MathSet{ v \in C^\infty_0(\Omega)^2:
  \ \nabla \cdot v = 0 }}^{H^1(\Omega)}
  \quad \text{and} \quad
  H:= \overline{ \MathSet{ v \in C^\infty_0(\Omega)^2 :
  \ \nabla \cdot v = 0 }}^{L^2(\Omega)}.
\end{equation*}
Note that instead of the definition via closures, in the case of $\Omega$ being bounded and Lipschitz,
these spaces are alternatively characterized in the following way,
see \cite[Chapter 1, Theorems 1.4 \& 1.6]{Temam1977}:
\begin{equation*}
  V = \MathSet{ v \in H^1_0(\Omega)^2 : \ \nabla \cdot v = 0 }
  \quad \text{and} \quad
  H = \MathSet{ v \in L^2(\Omega)^2: \ \nabla \cdot v= 0, \ u\cdot n = 0 \text{ on } \partial \Omega },
\end{equation*}
where by $u \cdot n$ we denote the normal trace of the vector field $u$.
To improve readability, whenever vectorial spaces like $H^1(\Omega)^2$ would arise in the subscript of some
norm, we shall drop the outer superscript $(\cdot)^2$.
For a Banach space $X$ and $I = (0,T]$
we denote by $L^p(I;X)$ the Bochner space of $X$ valued functions, for which the following norm is finite
\begin{equation*}
  \|v\|_{L^p(I;X)} = \left( \int_I \|v(t)\|_X^p \ dt  \right)^{1/p},
\end{equation*}
with the usual convention when $p= + \infty$.
It holds $L^p(I;L^p(\Omega)) \cong L^p(I\times \Omega)$, and for $p=2$,
we denote the inner product by $\IOprod{\cdot , \cdot}$.
Whenever $X$ is separable and $1 \le p < \infty$, it holds $(L^p(I;X))^* \cong L^{p^*}(I;X^*)$,
where $1/p + 1/p^* = 1$. The duality pairing for such spaces will be denoted by 
$\IOpair{\cdot , \cdot}$.
By $W^{k,p}(I;X)$ and $H^k(I;X)$ for $k \in \N$ we denote the spaces of functions $v$ 
satisfying $\partial_t^j v \in L^p(I;X)$, $j=0,...,k$.

\section{Navier-Stokes equations}\label{sec:NSE}
We start by recalling some regularity results for the Navier-Stokes equations,
and \blue{we} are going to prove some additional results, especially adapted to the situation considered in
this paper.
Throughout this paper, we shall always assume the convexity of $\Omega$. We will state explicitly, whenever
results also hold in a more general setting.
It is a well known result, that for $f \in L^2(I;V^*)+L^1(I;L^2(\Omega)^2)$
and $u_0 \in H$, there exists a unique weak solution,
i.e., the following Proposition holds, see \cite[Chapter 3, Theorem 3.1 \& Remark 3.1]{Temam1977}:
\begin{proposition}\label{prop:navier_stokes_solvability}
  Let
  $f \in L^2(I;V^*) + L^1(I;L^2(\Omega)^2)$ and $u_0 \in H$. Then there exists a unique weak solution 
  $u \in L^2(I;V)\cap C(\bar I;L^2(\Omega)^2)$ of \eqref{eq:nav_stokes_classical}, satisfying
  \begin{equation}\label{eq:nav_stokes_weak}
    \Opair{ \partial_t u,v} + \nu \Oprod{\nabla u, \nabla v} + \Oprod{(u\cdot \nabla)u,v} 
    = \Oprod{f,v} 
    \qquad \text{for all } v \in V
  \end{equation}
  in the sense of distributions on $I$,
  and $u(0) = u_0$. Moreover, there holds an estimate
  \begin{equation*}
    \|u\|_{\LinfLtwo} + \blue{\sqrt{\nu}} \|u\|_{\LtwoV} 
    \le C \left(\|u_0\|_H + \|f\|_{L^2(I;V^*) + L^1(I;L^2(\Omega))}\right).
  \end{equation*}
\end{proposition}
Note that the constant $C$ above only depends on $\nu$ and $\Omega$, but is independent of $T$,
see \cite[Theorems V.1.4.2, V.1.5.3, V.3.1.1]{Sohr2001}.
It is well known, that under the assumptions of \Cref{prop:navier_stokes_solvability}, the 
nonlinearity satisfies for $1 \le s,q < 2$:
\begin{equation}\label{eq:nonlinearity_regularity}
  (u\cdot \nabla) u \in L^s(I;L^q(\Omega)^2) \quad \text{whenever} \quad
  \frac{1}{s} + \frac{1}{q} \ge \frac{3}{2},
\end{equation}
see \cite[Lemma V.1.2.1]{Sohr2001}.
\Cref{eq:nav_stokes_weak} is the weak formulation of \eqref{eq:nav_stokes_classical} in divergence free spaces.
The proof of the above proposition relies heavily on the fact, that the trilinear form $c(\cdot,\cdot,\cdot)$
defined by
\begin{equation*}
  %c: \left[\Hozs\right]^3 \rightarrow \R, \quad \cO{u,v,w} = \Oprod{(u\cdot \nabla)v,w},
  c: \blue{\Hozs \times \Hozs \times \Hozs}  \rightarrow \R, \quad \cO{u,v,w} = \Oprod{(u\cdot \nabla)v,w},
\end{equation*}
posesses the properties summarized in the following lemma.
\begin{lemma}\label{lemm:c_swap}
  Let $\Omega \subset \R^2$ be an open Lipschitz domain, then there holds the estimate
  \begin{align}
    % \|v\|_{L^4(\Omega)} &\le 2^{-\frac{1}{4}}\|v\|_{L^2(\Omega)}^{\frac{1}{2}}\|\nabla v\|_{L^2(\Omega)}^{\frac{1}{2}}
    % && \text{for all } v \in H^1_0(\Omega),
    % \label{eq:l4_interpol_h10}
    % \\
    \|v\|_{L^4(\Omega)} &\le C\|v\|_{L^2(\Omega)}^{\frac{1}{2}}\|v\|_{H^1(\Omega)}^{\frac{1}{2}}
    \qquad \text{for all } v \in H^1(\Omega),
    \label{eq:l4_interpol_general}
  \end{align}
  due to which, the trilinear form $\cO{\cdot,\cdot,\cdot}$ satisfies for all $u,v,w \in H^1_0(\Omega)^2$:
  \begin{align*}
    \cO{u,v,w} & \le \|u\|_{L^4(\Omega)} \|\nabla v\|_{L^2(\Omega)} \|w\|_{L^4(\Omega)},\\
    \cO{u,v,u} & \le C\|u\|_{L^2(\Omega)} \|\nabla u\|_{L^2(\Omega)}\|\nabla v\|_{L^2(\Omega)}.
  \end{align*}
  Let further $\nabla \cdot u = 0$. Then it holds
  \begin{equation*}
    \cO{u,v,w} = - \cO{u,w,v} \quad \text{and} \quad 
    \cO{u,v,v} = 0.
  \end{equation*}
\end{lemma}
\begin{proof}
  The estimate for the $L^4(\Omega)$ norm can be found in \cite[Lemma II.3.2]{galdi_introduction_2011}
  and \cite[Theorem 3]{adams_cone_1977}.
%   Note that the specific constant for the $H^1_0(\Omega)$ case is strictly smaller,
%   than $2^{1/2}$, often used in literature, which is
%   the original constant in the Ladyzhenskaya inequality, see \cite[Equation 6]{ladyzhenskaia_solution_1959},
%   \cite[Chapter 3, Lemmas 3.3 \& 3.4]{Temam1977}.
  The properties of the trilinear forms are then consequences of \blue{Hölder's} inequality and
  integration by parts, and are shown, e.g., in \cite[Lemma IX.2.1]{galdi_introduction_2011}. 
\end{proof}
In what follows, we often consider the trilinear form $c$ integrated in time, which we denote by
\begin{equation*}
    \cIO{u,v,w} := \int_I \cO{u,v,w} \ dt.
\end{equation*}
% \begin{remark}
%   \todo{Move this to somewhere later?}
%   Note that due to the previous lemma, we can equivalently replace $c(\cdot,\cdot,\cdot)$
%   in equation \eqref{eq:nav_stokes_weak} by $\hat c(\cdot,\cdot,\cdot)$, defined by
%   \begin{equation}\label{eq:anti_symmetrized_c}
%     \hat c: \left[\Hozs \right]^3 \to \R, \quad \chatO{u,v,w} = \frac{1}{2} \cO{u,v,w} - \frac{1}{2} \cO{u,w,v}.
%   \end{equation}
%   On the space $V$ it holds $\cO{\cdot,\cdot,\cdot} = \chatO{\cdot,\cdot,\cdot}$,
%   thus in the continuous setting, we can formulate the 
%   Navier-Stokes equations with either choice of the trilinear form.
%   On the discrete spaces, the statements of \Cref{lemm:c_swap} do not hold true anymore. Hence
%   as in \cite{heywood_finite_1986,Chrysafinos2010},
%   we shall use this modified form $\chatO{\cdot,\cdot,\cdot}$, in order to define the fully discrete 
%   equations, see \eqref{eq:nav_stokes_discrete}.
% \end{remark}
% Analogously to the trilinear form $c$, we will use the notation
% \begin{equation*}
%     \chatIO{u,v,w} := \int_I \chatO{u,v,w} \ dt.
% \end{equation*}
In analyzing the Navier-Stokes equations, 
the instationary Stokes equations frequently arise as an auxiliary problem.
For initial data $u_0 \in H$ and right hand side $f\in L^1(I;L^2(\Omega)^2)\blue{+L^2(I;V^*)}$,
there exists a unique solution $w \in L^2(I;V) \cap L^\infty(I;L^2(\Omega)^2)$ to the Stokes equations
\begin{equation}\label{eq:stokes_weak}
  \left\lbrace
  \begin{aligned}
    \Opair{\partial_t w,v} + \nu \Oprod{\nabla w, \nabla v} &= \Oprod{f,v}
  \qquad \text{for all } v \in V,\\
  w(0) &= u_0,
  \end{aligned}
  \right.
\end{equation}
where the first line of \eqref{eq:stokes_weak} is understood in the sense of distributions on $I$.
We introduce the Stokes operator $A\colon D(A) \to H$ defined by 
\begin{equation*}
  \Oprod{A w,v} = \nu \Oprod{\nabla w,\nabla v} \qquad \text{for all } v \in V,
\end{equation*}
with domain $D(A) := \MathSet{v \in V : \Delta v \in L^2(\Omega)^2}$ and the
projection operator \mbox{$\mathbb P:L^2(\Omega)^2 \to H$}, defined by
\begin{equation*}
  \Oprod{\mathbb P w,v} = \Oprod{w,v} \qquad \text{for all } v \in H,
\end{equation*}
which is called the Helmholtz or Leray projection. 
% \cref{eq:stokes_weak} in operator notation then reads
% \begin{equation*}
%   \left\lbrace
%   \begin{aligned}
%     \partial_t w + A w &= \mathbb P f,\\
%     w(0) &= u_0.
%   \end{aligned}
%   \right.
% \end{equation*}
Note that with the vector-valued Laplacian
\begin{equation*}
  - \Delta\colon D(\Delta) \to L^2(\Omega)^2,
\end{equation*}
with domain $D(\Delta) := \MathSet{v \in H^1_0(\Omega)^2 : \Delta v \in L^2(\Omega)^2}$, the Stokes operator
also satisfies the representation
\begin{equation*}
  A = - \mathbb P \Delta.
\end{equation*}
For convex $\Omega$, the domains of the operators introduced above satisfy the representations
\begin{equation*}
  D(\Delta) = H^1_0(\Omega)^2 \cap H^2(\Omega)^2 \qquad \text{and}\qquad
  D(A) = V \cap H^2(\Omega)^2,
\end{equation*}
see \cite{Dauge1989} for the $H^2$ regularity of the Stokes operator.
The Stokes operator $A$ generates an analytic semigroup in $H$,
see \cite{behringer_fully_2022,lunardi_analytic_2012}, also
\cite{ashyralyev_well_posedness_1994} for a detailled general analysis.
One important feature of the Stokes problem is the
maximal parabolic regularity, which
indicates, that both the time derivative $\partial_t w$ and the Stokes operator
$A w$ individually inherit certain regularity properties of $f$, see \Cref{prop:max_par_reg} below. 
For homogeneous initial data, this consequence of the analyticity of the semigroup has been shown in
\cite{de_simon_applicazione_1964}, see also \cite[Chapter IV, Theorem 1.6.3]{Sohr2001}. Since our analysis
should also treat inhomogeneous initial data $u_0$, we need to define the proper spaces for the initial
data: 
\begin{equation*}
  V_{1-1/s} := \MathSet{ v \in H : \ \|v\|_{V_{1-1/s}} < \infty },
\end{equation*}
see also \cite{behringer_fully_2022} and \cite[Chapter 1, Section 3.3]{ashyralyev_well_posedness_1994},
where
\begin{equation*}
  \|v\|_{V_{1-1/s}} := \left( \int_I \|A \exp(-t A) v\|_H^s \ dt \right)^{1/s} + \|v\|_H.
\end{equation*}
\begin{remark}\label{rem:initial_data_spaces}
  Instead of using the spaces $V_{1-1/s}$ explicitly as a requirement for the initial data, we can
  make use of the following imbedding results:
  For $1 < s \le 2$ it holds $V \hookrightarrow V_{1-1/s}$ and for $1 < s < \infty$,
  it holds $V \cap D(\Delta) \hookrightarrow V_{1-1/s}$, see \cite[Remarks 2.8, 2.9]{behringer_fully_2022}.
\end{remark}
The Stokes problem exhibits the following regularity properties, 
see \cite[Proposition 2.6]{behringer_fully_2022}.
\begin{proposition}[Maximal parabolic regularity]\label{prop:max_par_reg}
  Let $1 < s < \infty$, $f \in L^s(I;L^2(\Omega)^2)$ and $u_0 \in V_{1-1/s}$.
  Then the solution $w$ of the Stokes equations \eqref{eq:stokes_weak} satisfies
  \begin{equation*}
    \|\partial_t w\|_{L^s(I;L^2(\Omega))} + \|A w\|_{L^s(I;L^2(\Omega))} \le 
    C\left(\|f\|_{L^s(I;L^2(\Omega))} + \|u_0\|_{V_{1-1/s}}\right).
  \end{equation*}
\end{proposition}
\blue{In the setting of \Cref{prop:max_par_reg}, the Stokes problem \eqref{eq:stokes_weak} formulated
  in divergence free spaces is equivalent to the following velocity-pressure formulation:
  %the following problem, introducing the pressure field $r$ associated to the velocity field $w$:
Find $(w,r) \in \left[L^2(I;H^1_0(\Omega)^2)\cap C(\bar I;L^2(\Omega)^2)\right] \times L^s(I;L^2_0(\Omega))$
satisfying
\begin{equation}\label{eq:stokes_weak_pressure}
  \IOpair{\partial_t w,v} + \nu \IOprod{\nabla w,\nabla v} - \IOprod{r,\nabla \cdot v}
  + \IOprod{q,\nabla \cdot w} = \IOprod{f,v}
\end{equation}
for all 
$(v,q) \in \left[L^2(I;H^1_0(\Omega)^2)\cap L^\infty(I;L^2(\Omega)^2)\right] \times L^2(I;L^2_0(\Omega))$ 
and $w(0) = u_0$, see \cite[Theorem 2.10]{behringer_fully_2022}.
Furthermore, the maximal parabolic regularity results of \Cref{prop:max_par_reg} imply the following 
estimate for the pressure:
\begin{equation*}
  \|r\|_{L^s(I;L^2(\Omega))} \le C \left( \|f\|_{L^s(I;L^2(\Omega))} + \|u_0\|_{V_{1-1/s}} \right).
\end{equation*}
}
\blue{Without additional smoothness assumptions, the maximal parabolic regularity for the Stokes problem
  does not immediately extend to the Navier-Stokes equations. A partial result can be obtained by 
  considering that the maximal parabolic regularity of the 
  Stokes operator was recently extended to the $L^p(\Omega)$ setting, where $p$ in general depends
  on the smoothness of $\Omega$. For a general Lipschitz domain $\Omega$, \cite[Theorem 1.6]{gabel_stokes_2022}
  shows that for some $\varepsilon > 0$ and any $p$, such that $|1/p - 1/2| < 1/4 + \varepsilon$,
  the maximal parabolic regularity holds. With this, we obtain the following result for the Navier-Stokes
  equations.
  \begin{theorem}\label{thm:exist_pressure}
%    \todo{\textcolor{black}{@Boris: Macht es Sinn dieses Theorem zu behalten?
%        Für die Existenz eines Drucks braucht man keine H2 Regularität}}
    Let $f \in L^s(I;L^2(\Omega)^2)$ and $u_0 \in V_{1-1/s}$ for some $s>1$ and 
    $u \in L^2(I;V)\cap L^\infty(I;L^2(\Omega)^2)$ the unique solution of \eqref{eq:nav_stokes_weak}.
    Then for $\gamma := \min \{s,4/3\}$ it holds
    \begin{equation*}
      \partial_t u \ , \ A u \in L^\gamma(I;L^{4/3}(\Omega)^2).
    \end{equation*}
    Further, there exists a unique $p \in L^\gamma(I;L^2_0(\Omega))$,  such that
  \begin{equation*}%\label{eq:nav_stokes_weak_pressure}
    \Opair{ \partial_t u,v} + \nu\Oprod{\nabla u, \nabla v} + \Oprod{(u\cdot \nabla)u,v} 
    - \Oprod{p,\nabla \cdot v} + \Oprod{\nabla \cdot u,q}
    = \Oprod{f,v} 
  \end{equation*}
  for all $v \in H^1_0(\Omega)^2, q \in L^2_0(\Omega)$
  in the sense of distributions on $I$, and $u(0) = u_0$.
\end{theorem}
\begin{proof}
Due to \eqref{eq:nonlinearity_regularity} it holds $(u\cdot \nabla)u \in L^{4/3}(I;L^{4/3}(\Omega)^2)$.
Since $f \in L^s(I;L^2(\Omega))$, this implies 
$\tilde f := f - (u\cdot \nabla) u \in L^\gamma(I;L^{4/3}(\Omega)^2)$.
%\begin{equation*}
%  \tilde f := f - (u\cdot \nabla) u \in L^\gamma(I;L^{4/3}(\Omega)) 
%  \quad \text{where} \quad \gamma := \min \{s, 4/3\}.
%\end{equation*}
Since $|3/4 - 1/2| = 1/4$, \cite[Theorem 1.6]{gabel_stokes_2022} shows that the Stokes problem posesses 
maximal parabolic regularity in $L^{4/3}(\Omega)$,
which applied to $u_0$ and $\tilde f$ then yields
$%\begin{equation*}
  \partial_t u, A u \in L^\gamma(I;L^{4/3}(\Omega)^2).
$%\end{equation*}
The existence of a pressure can then be proven as in \cite[Theorem 2.10]{behringer_fully_2022}.
%we define for almost every $t \in I$
%\begin{equation*}
%  \Opair{g(t),w} := \Oprod{f(t),w} - \Oprod{(u(t)\cdot \nabla)u(t),w} - \Oprod{\partial_t u(t),w}
%  - \nu\Oprod{\nabla u(t), \nabla w} \quad \text{for all } w \in H^1_0(\Omega)^2.
%\end{equation*}
%This construction yields a well defined $g(t) \in H^{-1}(\Omega)^2$, due to $f(t) \in L^2(\Omega)^2$, 
%$\partial_t u, (u(t) \cdot \nabla) u(t) \in L^{4/3}(\Omega)^2$ and $\Delta u \in H^{-1}(\Omega)^2$.
%By \cite[Chapter 1, Proposition 1.1]{Temam1977} there thus exists for every $t$ a unique
%$p(t) \in L^2_0(\Omega)$ satisfying $g(t) = \nabla p(t)$.
%Further, it holds
%\begin{align*}
%  \|p\|_{L^\gamma(I;L^2(\Omega))} &= \|g\|_{L^\gamma(I;H^{-1}(\Omega)}\\
%  & \le \|f\|_{L^\gamma(I;L^2(\Omega)^2)} 
%  + \|(u\cdot \nabla)u\|_{L^\gamma(I;L^{4/3}(\Omega)^2)}
%  + \|\partial_t u\|_{L^\gamma(I;L^{4/3}(\Omega)^2)}
%  + \nu \|\Delta u\|_{L^\gamma(I;H^{-1}(\Omega)^2)}\\
%  & \le \|f\|_{L^\gamma(I;L^2(\Omega)^2)} 
%  + \|(u\cdot \nabla)u\|_{L^\gamma(I;L^{4/3}(\Omega)^2)}
%  + \|\partial_t u\|_{L^\gamma(I;L^{4/3}(\Omega)^2)}
%  + \nu \|\nabla u\|_{L^\gamma(I;L^2(\Omega)^2)}.
%\end{align*}
%Since all terms on the right are bounded by the assumptions on $u_0,f$ this concludes the proof.
\end{proof}
}
\blue{The above result only recovers parts of the regularity available for the Stokes problem.
Using $H^2$ regularity, we will now show improved regularity results for the Navier-Stokes equations.}
% For the nonlinear Navier-Stokes equations,
% the maximal parabolic regularity does not immediately transfer.
% Using $H^2$ regularity and bootstrapping arguments however, we can recover some of the regularity 
% results.
Note that from now on we explicitly require the convexity of $\Omega$, whereas 
\blue{\Cref{prop:navier_stokes_solvability,prop:max_par_reg} and \Cref{thm:exist_pressure}} also hold for general
Lipschitz domains.
\begin{theorem}[$H^2$ regularity]\label{thm:h2_regularity}
  Let $u_0 \in V$ and $f \in L^2(I;L^2(\Omega)^2)$.
  Then the weak solution $u$ to the Navier-Stokes equations \eqref{eq:nav_stokes_weak}
  satisfies the improved regularity
  \begin{equation*}
    u \in L^2(I;H^2(\Omega)^2) \cap H^1(I;L^2(\Omega)^2) \hookrightarrow C(\blue{\bar I};H^1(\Omega)^2),
  \end{equation*}
  and there exist constants $C_1, C_2 > 0$, \blue{depending on $\nu,\Omega$ but independent of $T$,}
  such that there hold the bounds
  \begin{align*}
    \|u\|_{\LinfHone} & \le C_1 \left(\|u_0\|_{V} + \|f\|_{\LtwoLtwo}\right) 
      \exp\left(C_2\left(\|u_0\|_H^4 + \|f\|_{\LoneLtwo}^4\right)\right),\\
    \|u\|_{\LtwoHtwo} & \le C_1 \left(\|u_0\|_{V} + \|f\|_{\LtwoLtwo}\right)
    \left(1 + \|u_0\|_{H}^2 + \|f\|_{\LoneLtwo}^2\right)
                      \exp\left(C_2\left(\|u_0\|_H^4 + \|f\|_{\LoneLtwo}^4\right)\right).
  \end{align*}
\end{theorem}
\begin{proof}
  The proof of this result for $C^2$ domains can be found in \cite[Chapter 3, Theorem 3.10]{Temam1977}.
  Instead of a $C^2$ boundary, we can also use the $H^2$ regularity for the Stokes operator on 
  convex, polygonal domains, see, e.g., \cite[Theorem 5.5]{Dauge1989} or \cite[Theorem 2]{Kellogg1976},
  to obtain the claimed regularity. The norm bounds are obtained by the Gronwall lemma.
\end{proof}
\blue{Note that contrary to the $H^2$ regularity for the instationary Stokes problem, 
for the Navier-Stokes problem, even the proofs of the regularities $\partial_t u, Au \in L^2(I;L^2(\Omega)^2)$,
that are contained in the above result, require convexity of the domain.}
We shall state a corresponding estimate for $\|\partial_t u\|_{\LtwoLtwo}$ after discussing 
the regularity of the nonlinearity.
\begin{remark}\label{rem:LtwoLtwo_nonlinearity}
  By \blue{Hölder's} inequality, with this $H^2$ regularity result, and the imbedding
  $H^2(\Omega) \hookrightarrow L^\infty(\Omega)$, we immediately obtain
  \begin{equation*}
    (u \cdot \nabla) u \in L^2(I;L^2(\Omega)^2).
  \end{equation*}
  The regularity $u \in C(\blue{\bar I};H^1(\Omega)^2)$  almost yields that $u$ is uniformly
  bounded over the whole space time cylinder. However, since in two dimensions
  $H^1(\Omega) \not\hookrightarrow L^\infty(\Omega)$, the boundedness in space
  has to be shown via an additional argument.
\end{remark}
We first show, that the nonlinearity actually posesses more regularity, than what was claimed in 
\Cref{rem:LtwoLtwo_nonlinearity}.
\begin{theorem}\label{thm:nonlinearity_higher_regularity}
  Let the assertions of \Cref{thm:h2_regularity} be satisfied,
  i.e., $u_0 \in V$ and $f \in L^2(I;L^2(\Omega)^2)$, and let $u$ be the unique solution to the
  Navier-Stokes equations
  \eqref{eq:nav_stokes_weak}. Then the nonlinearity satisfies
  \begin{equation*}
    (u \cdot \nabla) u \in L^s(I;L^2(\Omega)^2) \qquad \text{for any } 1 \le s < \infty,
  \end{equation*}
  and for any $0 < \delta \le \min\{2,s\}$ there holds the estimate
  \begin{equation*}
    \|(u\cdot \nabla) u \|_{L^s(I;L^2(\Omega))} 
    \le C \|u\|_{\LtwoHtwo}^{ \frac{\delta}{s}} \|u\|_{\LinfHone}^{2 - \frac{\delta}{s} },
  \end{equation*}
  where the norms of $u$ on the right hand side can be estimated by \Cref{thm:h2_regularity}.
\end{theorem}
\begin{proof}
  The proof follows the ideas of \cite[Theorem V.1.8.2]{Sohr2001}, however we will use  
  interpolation spaces instead of fractional powers of the Stokes operator.
  With \blue{Hölder's} inequality  and the Sobolev imbedding 
  $H^{1+ \frac{\delta}{s}}(\Omega) \hookrightarrow C(\overline{\Omega})$,
  it holds for any $s < \infty$ and $\delta > 0$:
  \begin{equation*}
    \|(u\cdot \nabla)u\|_{L^2(\Omega)}^s 
    \le \|u\|_{L^\infty(\Omega)}^s \|\nabla u\|_{L^2(\Omega)}^s
    \le C\|u\|_{H^{1 + \frac{\delta}{s}}(\Omega)}^s \|\nabla u\|_{L^2(\Omega)}^s.
  \end{equation*}
  Since $\delta \le s$, we can express the space $H^{1+ \frac{\delta}{s}}(\Omega)$ as interpolation space
  $[H^1(\Omega),H^2(\Omega)]_{\delta/s}$, and
  obtain from \cite[Theorem 1]{Brezis2018} the estimate
  \begin{equation*}
    \|u\|_{H^{1 + \frac{\delta}{s}}(\Omega)} \le 
    C \|u\|_{H^{1}(\Omega)}^{1- \frac{\delta}{s}} 
    \|u\|_{H^{2}(\Omega)}^{ \frac{\delta}{s}}.
  \end{equation*}
  All in all, we see that
  \begin{equation*}
    \|(u\cdot \nabla)u\|_{L^2(\Omega)}^s \le C \|u\|_{H^2(\Omega)}^\delta \|u\|_{H^1(\Omega)}^{2s - \delta},
  \end{equation*}
  which is integrable in time, since $u \in L^\infty(I;H^1(\Omega)^2) \cap L^2(I;H^2(\Omega)^2)$,
  by Theorem \ref{thm:h2_regularity} and $\delta \le 2$. 
  With \blue{Hölder's} inequality, we obtain the proposed estimate, which concludes the proof.
\end{proof}
\begin{remark}
  The previous result $(u\cdot \nabla)u \in L^s(I;L^2(\Omega)^2)$ shows, that 
  the Navier-Stokes equations inherit the maximal parabolic regularity of the 
  Stokes problem, in cases where $f \in L^s(I;L^2(\Omega)^2)$ and $u_0 \in V_{1- 1/s}$ for some
  $2 \le s < \infty$. Especially it also holds $Au, \partial_t u \in L^s(I;L^2(\Omega)^2)$ with
  \begin{equation*}
    \|\partial_t u\|_{L^s(I;L^2(\Omega))} + \|Au\|_{L^s(I;L^2(\Omega))} 
    \le C \left(\|u_0\|_{V_{1-1/s}} 
    + \|f\|_{L^s(I;L^2(\Omega))} + \|(u \cdot \nabla) u\|_{L^s(I;L^2(\Omega))}\right).
  \end{equation*}
\end{remark}
With \Cref{thm:nonlinearity_higher_regularity}, and slightly higher regularity
of the data, we obtain the boundedness of $u$ in the space-time cylinder:
\begin{theorem}\label{thm:navier_stokes_higher_regularity}
  Let $\varepsilon > 0$, $f \in L^{2+\varepsilon}(I;L^2(\Omega)^2)$ and 
  $u_0 \in V_{1-1/(2+\varepsilon)}$.
  Then the unique weak solution $u$ to 
  \eqref{eq:nav_stokes_classical} satisfies additionally
  \begin{equation*}
    u \in C(\bar I \times \bar \Omega)^2,
  \end{equation*}
  and for $\varepsilon$ sufficiently small, there holds the estimate
  \begin{equation*}
    \|u\|_{L^\infty(I \times \Omega)}
    \le C \left(\|u_0\|_{V_{1-1/(2+\varepsilon)}} + \|f\|_{L^{2+\varepsilon}(I;L^2(\Omega))} 
    + \|u\|_{\LtwoHtwo}^{\varepsilon} \|u\|_{\LinfHone}^{2 -\varepsilon}\right),
  \end{equation*}
  where the norms of $u$ on the right hand side can be estimated by \Cref{thm:h2_regularity}.
\end{theorem}
\begin{proof}
  Applying Theorem \ref{thm:nonlinearity_higher_regularity}, we observe, that it holds especially 
  $(u\cdot \nabla) u \in L^{2+\varepsilon}(I;L^2(\Omega)^2)$.
  For $\varepsilon$ small enough, such that $\delta := \varepsilon (2 + \varepsilon) \le 2$, 
  there moreover holds
  \begin{equation*}
    \|(u\cdot \nabla) u \|_{L^{2+\varepsilon}(I;L^2(\Omega))} 
    \le C \|u\|_{\LtwoHtwo}^{\varepsilon} \|u\|_{\LinfHone}^{2 - \varepsilon}.
  \end{equation*}
  By a bootstrapping argument, since $f \in L^{2+\varepsilon}(I;\blue{L^2(\Omega)^2})$ and 
  $u_0 \in V_{1-1/\varepsilon}$, we can apply the maximal parabolic regularity 
  of \Cref{prop:max_par_reg}, and obtain together with $H^2$ regularity
  \begin{equation*}
    \partial_t u \in L^{2+\varepsilon}(I;\blue{L^2(\Omega)^2}), \qquad u \in L^{2+\varepsilon}(I;\blue{H^2(\Omega)^2}).
  \end{equation*}
  Hence
  $u \in W^{1,2+\varepsilon}(I;L^2(\Omega)^2) \cap L^{2+\varepsilon}(I;H^2(\Omega)^2)$.
  With the definitions $p:= 2+ \varepsilon$, $s:= \frac{1}{2 + \varepsilon/2}$,
  $\theta := 1 - \frac{1}{2+\varepsilon/4} = \frac{1}{2} + \frac{2\varepsilon}{8 + \varepsilon}$ it holds
  $1/p < s < 1/2$ and $1/2 < \theta < 1 - s$.
  We can thus apply \cite[Theorem 3]{Amann2001}, see also \cite[Lemma 2.11 b)]{disser_holder_2017},
  and obtain
  \begin{equation*}
    u \in C^{s- 1/p}(\blue{\bar I};(L^2(\Omega),H^2(\Omega))_{\theta,1}^2).
  \end{equation*}
  Applying 
  \cite[Theorems 3.4.1 \& 6.4.5]{bergh_interpolation_1976}, and 
  \cite[Theorem 4.57]{demengel_functional_2012}, % Sobolev imbedding for fractional space
  we moreover have
    \begin{align*}
      u \in C^{s - 1/p}(\blue{\bar I;(L^2(\Omega),H^2(\Omega))_{\theta,2}^2})
      = C^{ \frac{\varepsilon}{8 + 6 \varepsilon+ \varepsilon^2}}
      (\blue{\bar I};H^{1+ \frac{4\varepsilon}{8 + \varepsilon}}(\Omega)^2)
            \hookrightarrow C(\bar I \times \bar \Omega)^2.
    \end{align*}
  Note that here we have used, that for the $L^2(\Omega)$ case, all corresponding
  Sobolev-Slobodeckij, Besov, Bessel-potential and interpolation spaces coincide, see, e.g.,
  \cite{magenes_interpolational_1966}. For an overview over the topic of function spaces, see also
  \cite{triebel_theory_1992}. From the used embeddings, we moreover have the proposed estimate.
  This concludes the proof.
\end{proof}
\begin{corollary}\label{corr:nonlinearity_linftyl2}
  Let the assertions of Theorem \ref{thm:navier_stokes_higher_regularity} be satisfied, i.e.,
  $f \in L^{2+\varepsilon}(I;L^2(\Omega)^2)$ and $u_0 \in V_{1-1/(2+\varepsilon)}$
  for some $\varepsilon > 0$.
  Then the nonlinear term in the Navier-Stokes equations satisfies
  \begin{equation*}
    (u \cdot \nabla) u \in L^\infty(I;L^2(\Omega)^2),
  \end{equation*}
  and there holds the estimate
  \begin{equation*}
    \|(u\cdot \nabla) u\|_{L^\infty(I;L^2(\Omega))} \le 
    C \left(\|u_0\|_{V_{1-1/(2+\varepsilon)}} + \|f\|_{L^{2+\varepsilon}(I;L^2(\Omega))} 
    + \|u\|_{\LtwoHtwo}^{ \frac{2}{2+\varepsilon}} \|u\|_{\LinfHone}^{2 - \frac{2}{2+\varepsilon}}\right)
    \|u\|_{\LinfHone}.
  \end{equation*}
  where the norms of $u$ on the right hand side can be estimated by \Cref{thm:h2_regularity}.
\end{corollary}
\begin{proof}
  This is a direct consequence of \Cref{thm:h2_regularity,thm:navier_stokes_higher_regularity}
  and application of \blue{Hölder's} inequality.
\end{proof}
In the formulation of \cref{eq:nav_stokes_weak}, we have used divergence free test functions.
Whenever we want to test the equation with functions, that are not divergence free, we have to 
consider an alternative, equivalent formulation, that includes the pressure.
The following theorem guarantees, that we can freely switch between the two formulations, 
\blue{when $u_0 \in V$ and $f \in L^2(I;L^2(\Omega)^2)$}.
% It is based on the equivalence of the Stokes problem \eqref{eq:stokes_weak} in divergence free spaces,
% and the problem of finding $(w,r) \in L^2(I;H^1_0(\Omega)^2) \times L^2(I;L^2_0(\Omega))$ satisfying
% \begin{equation}\label{eq:stokes_weak_pressure}
%   \IOpair{\partial_t w,v} + \nu \IOprod{\nabla w,\nabla v} - \IOprod{r,\nabla \cdot v}
%   + \IOprod{q,\nabla \cdot w} = \IOprod{f,v}
% \end{equation}
% for all $(v,q) \in L^2(I;H^1_0(\Omega)^2) \times L^2(I;L^2_0(\Omega))$ and $w(0) = u_0$.
%see \cite[Theorem 2.10, Corollary 2.11]{behringer_fully_2022}.
\begin{theorem}
  Let the assertions of Theorem \ref{thm:h2_regularity} be satisfied,
  i.e., $u_0 \in V$ and $f \in L^2(I;L^2(\Omega)^2)$.
  Then there exists a unique solution $(u,p)$ with
  \begin{equation*}
    u \in L^2(I;H^2(\Omega)^2) \cap C(\blue{\bar I};V) \blue{, \ \partial_t u \in L^2(I;L^2(\Omega)^2)}
    \quad \text{and} \quad
    p \in L^2(I;H^1(\Omega) \cap L^2_0(\Omega)),
  \end{equation*}
  such that $u(0) = u_0$ and
  \begin{equation}\label{eq:nav_stokes_weak_with_pressure}
    \blue{\IOprod{\partial_t u,v} }
    + \nu \IOprod{\nabla u, \nabla v} 
    + \IOprod{(u\cdot \nabla)u,v}
    - \IOprod{p,\nabla \cdot v} + \IOprod{q, \nabla \cdot u} 
    = \IOprod{f,v},
  \end{equation}
  for all $(v,q) \in L^2(I;H^1_0(\Omega)^2)\times L^2(I;L^2_0(\Omega))$.
  Further it holds
  \begin{equation*}
    \|p\|_{L^2(I;H^1(\Omega))} \le C \left(\|u_0\|_V + \|f\|_{L^2(I;L^2(\Omega))} 
    + \|u\|_{L^2(I;H^2(\Omega))}\|u\|_{L^\infty(I;H^1(\Omega))}\right),
  \end{equation*}
  where the norms of $u$ on the right hand side can be estimated by \Cref{thm:h2_regularity}.
\end{theorem}
\begin{proof}
  This result can be shown using \Cref{thm:nonlinearity_higher_regularity} with the choice $s=2$,
  \blue{together } with a bootstrapping
  argument. Moving $(u\cdot \nabla)u \in L^2(I;L^2(\Omega)^2)$
  to the right hand side \blue{and applying \Cref{prop:max_par_reg} yields 
    $\partial_t u, Au \in L^2(I;L^2(\Omega)^2)$. The pressure is then obtained by following the 
    same steps as 
  %applying \Cref{rem:LtwoLtwo_nonlinearity} and using the result for Stokes, e.g.,
  \cite[Theorem 2.10, Corollary 2.11]{behringer_fully_2022}, where its uniqueness is given by its construction.}
\end{proof}

\section{Discretization}\label{sec:Discretization}
\blue{After discussing the continuous formulation of the Navier-Stokes equations, we now turn towards their
discretization.}
\subsection{Spatial discretization}\label{sect:space_discretization}
Let $\{\mathcal T_h\}$ denote a family of quasi-uniform triangulations of $\overline{\Omega}$
consisting of closed 
simplices. The index $h$ denotes the maximum meshsize.
We discretize the velocity $u$ by a discrete function space $U_h \subset H^1_0(\Omega)^2$
and the pressure $p$ by the discrete space $M_h \subset L^2_0(\Omega)$,
where $(U_h,M_h)$ satisfy the discrete, uniform LBB-condition
\begin{equation*}
  \sup_{v_h \in U_h} \dfrac{(\nabla \cdot v_h, q_h)_\Omega}{\|\nabla v_h\|_{L^2(\Omega)}}
  \ge \beta \|q_h\|_{L^2(\Omega)} \qquad \text{for all } q_h \in M_h,
\end{equation*}
with a constant $\beta > 0$ independent of $h$.
Throughout this work, we will assume the following approximation properties of the spaces $U_h$ and $M_h$.
This assumption is valid, e.g., for Taylor-Hood and \blue{MINI} finite
elements, even on shape regular meshes, see \cite[Assumption 7.2]{behringer_fully_2022}.
\begin{assumption}\label{ass:interpolation_operators}
  There exist interpolation operators $i_h\colon H^2(\Omega)^2 \cap H^1_0(\Omega)^2 \to U_h$
  and $r_h\colon L^2(\Omega) \to M_h$, such that
  \begin{align*}
    \|\nabla(v - i_h v)\|_{L^2(\Omega)} &\le c h \|\nabla^2 v\|_{L^2(\Omega)}
                                        && \hspace{-30mm} \text{for all } v \in H^2(\Omega)^2 \cap H^1_0(\Omega)^2,\\
    \|q - r_h q\|_{L^2(\Omega)} &\le c h \|\nabla q\|_{L^2(\Omega)} 
                                && \hspace{-30mm} \text{for all } q \in H^1(\Omega).
  \end{align*}
\end{assumption}
The (vector valued) discrete Laplacian $\Delta_h: U_h \to U_h$ is defined by
\begin{equation*}
  \Oprod{\nabla u_h,\nabla v_h} = -\Oprod{\Delta_h u_h, v_h} \quad \text{for all } v_h \in U_h.
\end{equation*}
We introduce the space $V_h$ of discretely divergence free functions as
\begin{equation*}
  V_h := \left\lbrace v_h \in U_h : \ \Oprod{\nabla \cdot v_h, q_h} = 0 \ \text{ for all } q_h \in M_h
  \right\rbrace.
\end{equation*}
The $L^2$ projection onto this space will be denoted by $\mathbb P_h\colon  \blue{L^2(\Omega)^2}
\to V_h$, satisfying
\begin{equation*}
  \Oprod{\Ph v , \phi_h} = \Oprod{v,\phi_h} \quad \text{for all }\phi_h \in V_h,
\end{equation*}
and allows us to introduce the discrete Stokes operator $A_h \colon  V_h \to V_h, A_h v_h = - \Ph \Delta_h v_h$.
Having defined these discrete spaces and operators,
we can now consider the Ritz projection for the stationary 
Stokes problem.
For any $(w,r) \in H^1_0(\Omega)^2 \times L^2(\Omega)$, 
the projection $\left(R^S_h(w,r),R^{S,p}_h(w,r)\right) \in U_h \times M_h$ is defined by
\begin{equation}\label{eq:stokes_ritzprojection}
  \begin{aligned}
    \Oprod{\nabla \left(w - R^S_h(w,r)\right),\nabla \phi_h} 
    - \Oprod{r - R^{S,p}_h(w,r),\nabla \cdot \phi_h} &= 0
    \quad \text{for all } \phi_h \in U_h\\
    \Oprod{\nabla \cdot \left(w - R^S_h(w,r)\right),\psi_h} &= 0
    \quad \text{for all } \psi_h \in M_h.
  \end{aligned}
\end{equation}
\blue{In case that} $w$ is discretely divergence free, i.e., $\Oprod{\nabla \cdot w,\psi_h}=0$ 
for all $\psi_h \in M_h$, it holds $R_h^S(w,r) \in V_h$.
Note that the space $V_h$ is in general
not a subspace of the space $V$ of pointwise divergence free 
functions.
% \throwout{
% For a Taylor-Hood pairing, this is due to the fact, that, while $\nabla \cdot v_h$ and $q_h$ have the same 
% polynomial degree on each element, $\nabla \cdot v_h$ can be discontinuous across elements. Hence the 
% degrees of freedom of $M_h$ are not sufficient to enforce $\nabla \cdot v_h = 0$ pointwise.
% }
This means, that on the discrete space $V_h$, the form $\cO{\cdot,\cdot,\cdot}$ does not posess
the anti symmetry properties shown in \Cref{lemm:c_swap}.
Hence we define, as in \cite{heywood_finite_1986,Chrysafinos2010}, an anti symmetric variant:
\begin{equation}\label{eq:anti_symmetrized_c}
  %\hat c: \left[\Hozs \right]^3 \to \R, \quad \chatO{u,v,w} = \frac{1}{2} \cO{u,v,w} - \frac{1}{2} \cO{u,w,v}.
  \hat c: \blue{\Hozs \times \Hozs \times \Hozs} \to \R,
  \quad \chatO{u,v,w} = \frac{1}{2} \cO{u,v,w} - \frac{1}{2} \cO{u,w,v}.
\end{equation}
Analogously to the trilinear form $\cIO{\cdot,\cdot,\cdot}$, we will use the notation
\begin{equation*}
    \chatIO{u,v,w} := \int_I \chatO{u,v,w} \ dt.
\end{equation*}
\begin{remark}
  Note that, due to \Cref{lemm:c_swap}, we can equivalently replace $c(\cdot,\cdot,\cdot)$
  in the continuous formulation of the Navier-Stokes equations \eqref{eq:nav_stokes_weak} by
  $\hat c(\cdot,\cdot,\cdot)$, as the two forms coincide on $V$.
%   On the space $V$ it holds $\cO{\cdot,\cdot,\cdot} = \chatO{\cdot,\cdot,\cdot}$,
%   thus in the continuous setting, we can formulate the 
%   Navier-Stokes equations with either choice of the trilinear form.
%   On the discrete spaces, the statements of \Cref{lemm:c_swap} do not hold true anymore. Hence
%   as in \cite{heywood_finite_1986,Chrysafinos2010},
%   we shall use this modified form $\chatO{\cdot,\cdot,\cdot}$, in order to define the fully discrete 
%   equations, see \eqref{eq:nav_stokes_discrete}.
\end{remark}
By its definition, $\chatO{\cdot,\cdot,\cdot}$ now has the following antisymmetric
properties on the space $V_h$, which will later allow us to show the stability of the 
fully discrete solutions.
\begin{lemma}\label{lemm:chat}
  The trilinear form $\chatO{\cdot,\cdot,\cdot}$ satisfies
  \begin{align*}
    \chatO{u_h,v_h,w_h} & \le C\|\nabla u_h\|_{L^2(\Omega)} \|\nabla v_h\|_{L^2(\Omega)} 
    \|\nabla w_h\|_{L^2(\Omega)} && \text{for all } u_h, v_h, w_h \in U_h,\\
    \chatO{u_h, v_h, w_h} &= - \chatO{u_h,w_h,v_h} && \text{for all } u_h, v_h, w_h \in U_h,\\
    \chatO{u_h, v_h, v_h} &= 0 && \text{for all } u_h, v_h \in U_h.
  \end{align*}
\end{lemma}
\begin{proof}
  The last two identities are a direct consequence of the definition of $\chatO{\cdot,\cdot,\cdot}$.
  The first estimate follows from \Cref{lemm:c_swap} and the imbedding 
  $H^1_0(\Omega) \hookrightarrow L^4(\Omega)$.
\end{proof}
% To obtain the same regularities for the spacially discretized solution, we thus use the trilinear
% form $\chatO{\cdot,\cdot,\cdot}$ introduced in \eqref{eq:anti_symmetrized_c}, instead of 
% $\cO{\cdot,\cdot,\cdot}$. 
Note that due to the above lemma,
formally we are still allowed to switch the second and third argument of $\chatO{\cdot,\cdot,\cdot}$.
The original form $\cO{\cdot,\cdot,\cdot}$ however had a strict disctinction between the two arguments, as it contains
the gradient of the second argument, but only the function values of the third argument.
This is of importance when estimating the form in terms of its arguments.
In $\chatO{\cdot,\cdot,\cdot}$ gradients occur in the second and third argument, thus switching the
arguments does not allow us to obtain improved estimates.
For this reason, we state the following lemma, which allows us to switch the second and third
arguments of $\cO{\cdot,\cdot,\cdot}$ by introducing an additional term,
even if the first argument is not (pointwise) divergence free.
\begin{lemma}\label{lemm:c_swap_no_div}
  Let $u,v,w \in \blue{H^1_0(\Omega)^2}$. Then it holds
  \begin{equation*}
    \cO{u,v,w} 
    = - \cO{u,w,v} - \Oprod{\nabla \cdot u, v \cdot w}.
  \end{equation*}
\end{lemma}
\begin{proof}
  The proof is simply an application of integration by parts, and can also be seen, e.g., from
  \cite[Equation 2.9]{ahmed_higherorder_2021}.
\end{proof}

We conclude this subsection on the space discretization by recalling some important interpolation 
estimates. On the continuous level, applying \eqref{eq:l4_interpol_general} to the first order derivatives,
and using $H^2$ regularity, yields
\begin{equation}\label{eq:w14_interpolation}
  \|\nabla w \|_{L^4(\Omega)} \le C \|\nabla w\|_{L^2(\Omega)}^{\frac{1}{2}}\|A w\|_{L^2(\Omega)}^{\frac{1}{2}}
  \qquad \text{for all } w \in H^2(\Omega)^2.
\end{equation}
Using the Stokes operator on the right hand side, instead of second order derivatives, allows us
to translate this result to the discrete setting. This is facilitated by the following result,
showing that for discretely divergence free functions, the
discrete Laplacian $\Delta_h$ can be bounded in terms of the discrete Stokes operator $A_h$:
\begin{equation}\label{eq:guermond_pasciak}
   \|\Delta_h w_h\|_{L^2(\Omega)} \le C \|A_h w_h \|_{L^2(\Omega)} \qquad \text{for all } w_h \in V_h,
\end{equation}
see \cite[Corollary 4.4]{heywood_finite_1982} or \cite[Lemma 4.1]{guermond_stability_2008}.
With this, we can translate \eqref{eq:w14_interpolation} to the discrete setting,
by considering for some fixed $w_h \in V_h$
the solution $w \in H^1_0(\Omega)^2$ to the continuous problem
\begin{equation*}
\Oprod{\nabla w,\nabla v} = \Oprod{-\Delta_h w_h,v} \qquad \text{for all } v \in H^1_0(\Omega)^2.
\end{equation*}
By the stability of the Poisson Ritz projection in $W^{1,4}(\Omega)$, \eqref{eq:w14_interpolation} and 
\eqref{eq:guermond_pasciak}, we then obtain the discrete version of \eqref{eq:w14_interpolation}.
\begin{equation}\label{eq:discrete_w14_interpolation}
  \|\nabla w_h\|_{L^4(\Omega)} \le 
  C \|\nabla w_h\|_{L^2(\Omega)}^{\frac{1}{2}}\|A_h w_h\|_{L^2(\Omega)}^{\frac{1}{2}}
  \qquad \text{for all } w_h \in V_h.
\end{equation}
Analogously, the Gagliardo-Nirenberg inequality,
\begin{equation}\label{eq:gagliardo_nirenberg}
  \|w\|_{L^\infty(\Omega)} \le C \|w\|_{L^2(\Omega)}^{\frac{1}{2}}\|A w\|_{L^2(\Omega)}^{\frac{1}{2}},
\end{equation}
which is a consequence of \cite[Theorem 3]{adams_cone_1977} together with $H^2$ regularity, has a discrete 
analogon. It can be shown using the standard discrete Gagliardo-Nirenberg inequality
\begin{equation*}
   \|w_h\|_{L^\infty(\Omega)} 
   \le C \|w_h\|_{L^2(\Omega)}^{ \frac{1}{2}} \|\Delta_h w_h\|_{L^2(\Omega)}^{ \frac{1}{2}}
   \qquad \text{for all } w_h \in U_h,
\end{equation*}
which was proven in \cite[Lemma 3.3]{hansbo_strong_2002}. The proof stated there for smooth domains 
remains the same for convex domains.
The discrete version of \eqref{eq:gagliardo_nirenberg} is then again obtained by applying
\eqref{eq:guermond_pasciak} and reads
\begin{equation}\label{eq:discrete_stokes_gagliardo_nirenberg}
   \|w_h\|_{L^\infty(\Omega)} \le C \|w_h\|_{L^2(\Omega)}^{\frac{1}{2}}\|A_h w_h\|_{L^2(\Omega)}^{\frac{1}{2}}
   \qquad \text{for all } w_h \in V_h.
\end{equation}
Straightforward calculations, using the definition of $A_h$, also  give
\begin{equation}\label{eq:discrete_h1_interpolation}
   \|\nabla w_h\|_{L^2(\Omega)} \le C \|w_h\|_{L^2(\Omega)}^{\frac{1}{2}}\|A_h w_h\|_{L^2(\Omega)}^{\frac{1}{2}}
   \qquad \text{for all } w_h \in V_h.
\end{equation}
With these considerations regarding
the spatial discretization, we can now consider the fully discrete
Navier-Stokes equations by also discretizing in time.

\subsection{Temporal discretization}
For discretization in time, we employ the discontinuous Galerkin method of order q (dG(q)),
which is also used, e.g., in \cite{behringer_fully_2022,Chrysafinos2010,eriksson_time_1985}.
The time interval $I = (0,T]$ is partitioned into $M$ half-open sub-intervals $I_m = (t_{m-1},t_m]$ with
$0=t_0 < t_1 < t_2 < ... < t_M = T$. We denote each timestep by $k_m = t_m - t_{m-1}$ and for fixed
$M$ the maximal timestep by $k:= \max_{1\le m \le M} k_m$, as well as the minimal one by 
$k_{\min} := \min_{1\le m \le M} k_m$.
If we want to emphasize that $I_m$ belongs to a discretization level $k$, we denote it by $I_{m,k}$.
We make some standard assumptions on the properties on the time discretization:
\begin{enumerate}
  \item There are constants $C,\beta > 0$ independent of $k$, such that 
    \begin{equation*}
      k_{\min} \ge C k^\beta.
    \end{equation*}
  \item There is a constant $\kappa > 0$ independent of $k$, such that for all $m=1,2,...,M-1$
    \begin{equation*}
      \kappa^{-1} \le \frac{k_m}{k_{m-1}} \le \kappa.
    \end{equation*}
  \item It holds $k \le \frac{T}{4}$.
\end{enumerate}
A dG(q) function with values in a given Banach space $\mathcal B$ is then given as a function
in the space
\begin{equation*}
  X^q_k(\mathcal B):= 
  \MathSet{ v \in L^2(I;\mathcal B): \  
  v|_{I_m} \in \mathcal P_q(I_m;\mathcal B) \quad \text{for all } 1 \le m \le M },
\end{equation*}
where on each $I_m$ the space $\mathcal P_q(I_m;\mathcal B)$ is given as the space of polynomials 
in time up to degree $q$ with values in $\mathcal B$:
\begin{equation*}
  \mathcal P_q(I_m;\mathcal B) = \MathSet{ v \in L^2(I_m;\mathcal B) : \
  \exists \ v_0,...,v_q \in \mathcal B \ \text{s.t.} \ v = \sum_{j=0}^q v_j t^j }.
\end{equation*}
Note that no continuity is required at the time nodes $t_m$, which is why
we use the following standard notations for one sided limits and jump terms:
\begin{equation*}
  v_m^+ : = \lim_{\varepsilon \to 0^+} v(t_m + \varepsilon), \qquad
  v_m^- : = \lim_{\varepsilon \to 0^+} v(t_m - \varepsilon), \qquad
  [v]_m : = v_m^+ - v_m^-.
\end{equation*}
We introduce the compact notations
\begin{equation*}
  \Vkh := X^q_k(V_h), \qquad
  \Ukh := X^q_k(U_h), \qquad
  \Mkh := X^q_k(M_h), \qquad
  \Xkh := \Ukh \times \Mkh.
\end{equation*}
Having defined these dG spaces, we introduce the following projection operator in time:
$\pi_\tau\colon C(I;L^2(\Omega)) \to X^q_k(L^2(\Omega))$ defined by
\begin{equation}\label{eq:dg_time_projection}
  \left\{
\begin{aligned}
  \ImOprod{\pi_\tau v - v,\varphi} &= 0 \qquad \quad \text{for all } \varphi \in \mathcal P_{q-1}(I_m;L^2(\Omega)), 
  \text{ if } q>0,\\
  \pi_\tau v(t_m^-) &= v(t_m^-),
\end{aligned}
 \right.
 \end{equation}
for all $m=1,2,...,M$. In case $q=0$, the projection operator is defined solely by the second condition.

\begin{remark}
  In this paper we will restrict ourselves to the two lowest cases $q=0$ and $q=1$.
  Since we work in a setting of low regularity of the right hand side $f$, 
  the error estimates would not benefit from higher order schemes.
  Also, since the Navier-Stokes equations already pose a challenging large system to solve,
  higher order schemes in many applications are not feasible from the standpoint of computational cost.
%   Lastly, since the argument in this paper employs a discrete dual equation,
%   we need to show the solvability of this dual problem.
%   Testing the discrete problem with its solution yields evaluations of left and rightsided limits at the
%   time
%   nodes, due to the $\partial_t \cdot$ term, and $\|\cdot \|_{L^\infty(I_m)}$ terms 
%   due to the remaining terms in the equation.
%   For higher order schemes, it is still an open question how to treat these terms, such that
%   that, e.g., a discrete Gronwall Lemma shows the overall boundedness in $L^\infty(I)$.
%   Lastly, for $q=0,1$, the $L^\infty(I_m)$-norms on each subinterval can be estimated by 
%   the evaluations in the two endpoints $t_{m-1}, t_{m}$, ensuring well posedness of the occuring
%   discretized equations via a Gronwall argument.
\end{remark}
We can now introduce the time-discretized formulation of 
the Navier-Stokes equations.
We define the time-discrete bilinear form for the transient Stokes equations as in 
\cite{behringer_fully_2022} by
\begin{equation*}
  \mathfrak B (u,v) := \sum_{m=1}^M \ImOprod{\partial_t u, v} + \nu\IOprod{\nabla u, \nabla v}
  + \sum_{m=2}^M \Oprod{[u]_{m-1},v^+_{m-1}} + \Oprod{u_0^+,v_0^+}.
\end{equation*}
Since we frequently will test some discrete equations with their respective solutions,
let us recall the following lemma.
\begin{lemma}\label{lemm:jump_reorder}
  For any $\vk\in X_k^q(L^2(\Omega))$ it holds
  \begin{align*}
  \ImOprod{\partial_t \vk,\vk} + \Oprod{[\vk]_{m-1},v_{kh,m-1}^+}
  &= \frac{1}{2} \left(\|v_{k,m}^-\|_{L^2(\Omega)}^2  + \|[\vk]_{m-1}\|_{L^2(\Omega)}^2 - \|v_{k,m-1}^-\|_{L^2(\Omega)}^2 \right),\\
  - \ImOprod{\vk, \partial_t \vk} - \Oprod{v_{k,m}^-, [\vk]_{m}}
  &= \frac{1}{2} \left(\|v_{k,m-1}^+\|_{L^2(\Omega)}^2  + \|[\vk]_{m}\|_{L^2(\Omega)}^2 - \|v_{k,m}^+\|_{L^2(\Omega)}^2 \right).
  \end{align*}
\end{lemma}
\begin{proof}
  For the first equality, we can express the integral over the time derivatives via
  \begin{equation*}
   \ImOprod{\partial_t \vk,\vk} =  \frac{1}{2} \|v_{k,m}^-\|^2 - \frac{1}{2} \|v_{k,m-1}^+\|^2.
  \end{equation*}
  Writing $v_{k,m-1}^+ = [v]_{k,m-1} + v_{k,m-1}^-$ and recombining terms gives the first identity.
  The proof of the second equality works completely anologous.
\end{proof}
The fully discrete formulation of the transient Navier-Stokes equations, using the anti symmetrized trilinear 
form $\chatO{\cdot,\cdot,\cdot}$, introduced in \eqref{eq:anti_symmetrized_c}, is \blue{now} given as:
Find $\ukh \in \Vkh$, such that 
\begin{equation}\label{eq:nav_stokes_discrete}
  \mathfrak B (\ukh,\vkh) + \chatIO{\ukh,\ukh,\vkh} = \IOprod{f,\vkh} + \Oprod{u_0,v_{kh,0}^+} 
  \quad \text{for all } \vkh \in \Vkh.
\end{equation}
Since in general $V_h \not \subset V$, the discrete solution $\ukh \in \Vkh$ is not divergence free,
and thus, we are not allowed to 
use it as a test function for the divergence-free continuous formulation \eqref{eq:nav_stokes_weak}.
Because of this, we introduce an equivalent formulation with pressure:
Find $(\ukh,\pkh) \in \Xkh$, such that for all $(\vkh,\qkh) \in \Xkh$
it holds
\begin{equation}\label{eq:nav_stokes_discrete_with_pressure}
  B((\ukh,\pkh),(\vkh,\qkh)) + \chatIO{\ukh,\ukh,\vkh} = \IOprod{f,\vkh} + \Oprod{u_0,v_{kh,0}^+},
\end{equation}
where the mixed bilinear form $B$ is defined by
\begin{align*}
  B ((u,p),(v,q)) := & \sum_{m=1}^M \ImOprod{\partial_t u, v} + \nu\IOprod{\nabla u, \nabla v}
  + \sum_{m=2}^M \Oprod{[u]_{m-1},v^+_{m-1}} + \Oprod{u_0^+,v_0^+} \\
                    & - \IOprod{\nabla \cdot v, p} + \IOprod{\nabla \cdot u, q}.
\end{align*}
\begin{theorem}
  The two formulations \eqref{eq:nav_stokes_discrete} and \eqref{eq:nav_stokes_discrete_with_pressure}
  are equivalent in the sense that, if $\ukh \in \Vkh$ satisfies 
  \eqref{eq:nav_stokes_discrete}, then there exists $\pkh \in \Mkh$ such that 
  $(\ukh,\pkh)$ solves \eqref{eq:nav_stokes_discrete_with_pressure}.
  Conversely, if $(\ukh,\pkh)$ satisfies \eqref{eq:nav_stokes_discrete_with_pressure},
  then $\ukh$ is an element of $\Vkh$ and satisfies \eqref{eq:nav_stokes_discrete}.
\end{theorem}
\begin{proof}
  This can be shown by using the same arguments as \cite[Proposition 4.3]{behringer_fully_2022}.
\end{proof}

\subsection{Stokes error estimates}
\blue{Before further analyzing the fully discrete Navier-Stokes equations, let us recall the 
  discrete formulation of the Stokes equations \eqref{eq:stokes_weak_pressure},
  and the available error estimates that we wish to extend to the
  Navier-Stokes equations in this work.}
%Before showing the fully discrete error estimates for the Navier-Stokes equations,
%we recapitulate the corresponding error estimates for the Stokes equations \eqref{eq:stokes_weak_pressure}.
  The fully discrete formulation \blue{of the instationary Stokes equation} reads:
  Find $(\wkh,r_{kh}) \in \Xkh$ satisfying
\begin{equation}\label{eq:stokes_discrete}
  B((\wkh,r_{kh}),(\phikh,\psikh)) = \Oprod{u_0,\phi_{kh,0}^+} + \IOprod{f,\phikh} 
  \qquad \text{for all } (\phikh,\psikh) \in \Xkh.
\end{equation}
  In the recent contributions \cite{behringer_fully_2022,vexler_l2i_2023}, 
  best approximation type error estimates for the Stokes problem 
  were shown in the norms of $L^\infty(I;L^2(\Omega))$, $L^2(I;L^2(\Omega))$ and $L^2(I;H^1(\Omega))$.
  There hold the following results, formulated in terms of best approximation error terms and the
  errors of the projection in time $\pi_\tau$, defined in \eqref{eq:dg_time_projection},
  and the Stokes Ritz projection $R_h^S$, defined in \eqref{eq:stokes_ritzprojection}.
\begin{proposition}[{\cite[Theorems 6.1\& 6.3]{vexler_l2i_2023}}]\label{thm:stokesl2l2_l2h1}
  Let $f \in L^2(I;L^2(\Omega)^2)$ and $u_0 \in V$, let $(w,r)$ and $(\wkh,r_{kh})$ be the continuous
   and fully discrete solutions to the Stokes problems \eqref{eq:stokes_weak_pressure}
   and \eqref{eq:stokes_discrete}. Then for any $\chi_{kh} \in \Vkh$, there holds
   \begin{align*}
     \|w-\wkh\|_{L^2(I\times \Omega))} & \le C \left(\| w - \chi_{kh}\|_{L^2(I\times\Omega)} 
     + \|w - \pi_\tau w\|_{L^2(I\times \Omega)} + \|w - R_h^S(w,r)\|_{L^2(I\times \Omega)}\right)\\
     \|\nabla(w-\wkh)\|_{L^2(I\times \Omega))} & \le C\left( \| \nabla(w - \chi_{kh})\|_{L^2(I\times\Omega)} 
     + \|\nabla(w - \pi_\tau w)\|_{L^2(I\times \Omega)} + \|\nabla(w - R_h^S(w,r))\|_{L^2(I\times \Omega)}
   \right).
   \end{align*}
%    where $\chi$ is any function from the chosen discretization space. In terms of convergence orders, 
%    one obtains
%    \begin{align*}
%       \|w-\wkh\|_{L^2(I;L^2(\Omega))} & \le C(k+h^2) (\|f\|_{\LtwoLtwo} + \|u_0\|_V)\\
%       \|\nabla(w-\wkh)\|_{L^2(I;L^2(\Omega))} & \le C(k^{\frac{1}{2}} + h) (\|f\|_{\LtwoLtwo} + \|u_0\|_V).
%    \end{align*}
\end{proposition}
\begin{proposition}[{\cite[Corollary 6.4]{behringer_fully_2022}}]\label{prop:stokes_linfty}
  Let $f \in L^s(I;L^2(\Omega)^2)$ and $u_0 \in V_{1-1/s}$ for some $s>1$, and let $(w,r)$ and $(\wkh,r_{kh})$
  be the continuous and fully discrete solutions to the Stokes equations \eqref{eq:stokes_weak_pressure} and
  \eqref{eq:stokes_discrete}. Then for any $\chi_{kh} \in \Vkh$, there holds
  \begin{equation*}
    \|w - \wkh\|_{L^\infty(I;L^2(\Omega))}
    \le C \left(\ln \frac{T}{k}\right)
    \left( \|w - \chi_{kh}\|_{\LinfLtwo} + \|w - R_h^S(w,r)\|_{\LinfLtwo}\right).
  \end{equation*}
\end{proposition}
Having shown these error estimates for the Stokes equations, the natural question arises,
whether these results can be extended to the Navier-Stokes equations.
We first give a positive answer for the $L^2(I;H^1(\Omega))$ error in \Cref{thm:nav_stokes_l2h1}.
The main result of this work is the error estimate in the $L^\infty(I;L^2(\Omega))$ norm, presented in 
\Cref{thm:nav_stokes_bestapprox}.
With the same techniques, the proof of the $L^2(I;L^2(\Omega))$ error estimate is then straightforward,
and we state the result in \Cref{thm:nav_stokes_l2l2}.

\subsection{Existence and stability results}
\blue{We begin our analysis of the fully discrete Navier-Stokes equations by presenting some stability
results, followed by an existence result of fully discrete solutions. To show these,}
let us \blue{first} recall the following
version of a discrete Gronwall lemma,
which is stated in \cite[Lemma 5.1]{Heywood1990} for a constant timestep $k$, but its proof
can easily be adapted to the setting of variable timesteps $k_m$.
\begin{lemma}\label{lemm:discrete_gronwall}
  Let $\{k_n\}, \{a_n\}, \{b_n\}, \{c_n\}, \{\gamma_n\}$ be sequences of nonnegative numbers
  and $B \ge 0$ a constant, such that for each $n \in \N_0$ it holds
  \begin{equation*}
    a_n + \sum_{m=0}^n k_m b_m \le \sum_{m=0}^n k_m \gamma_m a_m + \sum_{m=0}^n k_m c_m + B,
  \end{equation*}
  and $k_m \gamma_m < 1$ for all $m \in \{0,...,n\}$.
  Then with $\sigma_m := (1 - k_m\gamma_m)^{-1}$ it holds
  \begin{equation*}
    a_n + \sum_{m=0}^n k_m b_m \le \exp \left( \sum_{m=0}^n k_m \sigma_m \gamma_m \right) \cdot
    \left(\sum_{m=0}^n k_m c_m + B\right).
  \end{equation*}
\end{lemma}
\begin{proof}
  By choosing $\widetilde k = 1$, $\widetilde \gamma_m = k_m \gamma_m$ and $\widetilde c_m = k_m c_m$,
  the quantities with $\widetilde \cdot$ correspond to the notation of \cite[Lemma 5.1]{Heywood1990}.
  The result then directly follows from this redefinition.
\end{proof}
Continuous and discrete Gronwall lemmas are stated in many different forms in the literature, see, e.g.,
\cite{li_generalized_2021} and the references therein for an overview over different generalizations
of the original lemma.
Note that, since the sum on the right hand side of the assumed bound goes up to $n$, 
the additional assumption $k_m \gamma_m < 1$ is needed. This is not the case in explicit forms 
of discrete Gronwall lemmas, where the sum on the right goes only up to $n-1$, which is the form
most often considered in the literature, e.g., \cite{pachpatte_integral_2006}.
In the context of dG timestepping schemes, the sequences $\{a_n\}$ and $\{b_n\}$
often correspond to squared norms at time nodes or over subintervals.
It therefore seems natural to also state the following version of a Gronwall lemma,
where we also include a sum over non-squared contributions.
To the best of the authors knowledge, a result like this has not been explicitly used in the literature.
In later sections of this work, this lemma will facilitate the analysis of discrete problems
with right hand sides that are $L^1$ in time. Here the norms of the solution occur in a
non squared contribution.
The following lemma shows, that the weights of the squared sum enter exponentially into the 
estimate, whereas the weights of the linear sum enter linearly in the estimate for $x_n$.
It can be understood as an adaptation of Bihari's inequality to the discrete setting, see
\cite{bihari_generalization_1956,lipovan_retarded_2000}.
To improve readability, we drop the explicit mentioning of the timesteps $k_m$, as we can always apply
a transformation as done in the proof of \Cref{lemm:discrete_gronwall}.
\begin{lemma}\label{thm:discrete_gronwall_quadlinear}
  Let $\{x_n\}, \{b_n\}, \{c_n\}, \{d_n\}, \{\gamma_n\}$ be sequences of nonnegative numbers
  and $B \ge 0$ a constant, such that for each $n \in \N_0$ it holds
  \begin{equation*}
    x_n^2 + \sum_{m=0}^n b_m \le \sum_{m=0}^n \gamma_m x_m^2 
    + \sum_{m=0}^n d_m x_m
    + \sum_{m=0}^n c_m + B,
  \end{equation*}
  and $\gamma_m < 1$ for all $m \in \{0,...,n\}$. Then with $\sigma_m := (1 - \gamma_m)^{-1}$ it holds
  \begin{equation*}
    x_n^2 + \sum_{m=0}^n b_m \le 2\exp \left(2 \sum_{m=0}^n \sigma_m \gamma_m \right)\cdot
    \left[\left(\sum_{m=0}^n d_m \right)^2 + \sum_{m=0}^n c_m + B\right].
  \end{equation*}
\end{lemma}
\begin{proof}
  For $n \in \N_0$ define $\delta_n \ge 0$ such that 
  \begin{equation}\label{eq:Xn}
    x_n^2 + \sum_{m=0}^n b_m + \delta_n 
    =\sum_{m=0}^n \gamma_m x_m^2 
    + \sum_{m=0}^n d_m x_m
    + \sum_{m=0}^n c_m + B,
  \end{equation}
  and set $X_n := \sqrt{x_n^2 + \sum_{m=0}^n b_m + \delta_n}$. We will show
  \begin{equation*}
    X_n^2
    \le 2\exp \left(2 \sum_{m=0}^n \sigma_m \gamma_m \right) \cdot
    \left[\left(\sum_{m=0}^n d_m \right)^2 + \sum_{m=0}^n c_m + B\right]
    \qquad \text{for all } n \in \N_0,
  \end{equation*}
  from which the assertion follows by the definition of $X_n$.
  Note that by assumption it holds
  \begin{equation*}
    X_n^2 - X_{n-1}^2 = \gamma_n x_n^2 + d_n x_n + c_n \ge 0
    \qquad \text{for all } n \in \N,
  \end{equation*}
  and thus the sequences $\lbrace X_n^2\rbrace$ and $\lbrace X_n\rbrace$
  are monotonically increasing.
  By $X_n \ge x_n$ and the monotonicity of $X_n$, from \eqref{eq:Xn} we obtain
  \begin{equation}\label{eq:discrete_gronwall_squared_form}
    \begin{aligned}
      X_n^2 &\le \sum_{m=0}^n \gamma_m X_m^2 + \sum_{m=0}^n d_m X_m + \sum_{m=0}^n c_m + B
            &\le \left(\sum_{m=0}^n \gamma_m X_m + \sum_{m=0}^n d_m\right) X_n + \sum_{m=0}^n c_m + B.
    \end{aligned}
  \end{equation}
  Let us now recall, that for $a,b > 0$, an estimate
  $x^2 \le a x + b$ implies $x \le a + \sqrt{b}$.
  To see this, we start by computing the roots of the quadratic polynomial, yielding
  $x \le \frac{a}{2} + \sqrt{ \frac{a^2}{4} +b}$.
  The root of the polynomial can then be estimated by
    $\frac{a}{2} + \sqrt{ \frac{a^2}{4} +b} \le a + \sqrt{b},$
  which can be shown by subtracting $a/2$ and squaring both sides.
  Applied to \cref{eq:discrete_gronwall_squared_form}, we have thus for every $n \in \N_0$ the estimate
  \begin{equation*}
    X_n \le \sum_{m=0}^n \gamma_m X_m + \sum_{m=0}^n d_m + \sqrt{\sum_{m=0}^n c_m + B}
    = \sum_{m=0}^n \gamma_m X_m + \sum_{m=0}^n \tilde c_m + \tilde B,
  \end{equation*}
  where we have defined $\tilde B :=0$ as well as 
  \begin{equation*}
    \tilde c_0 := \left(d_0 + \sqrt{c_0 + B}\right) \blue{\ge 0}, \quad \text{and} \quad
    \tilde c_m := \left(d_m + \sqrt{\sum_{j=0}^m c_j + B} 
    - \sqrt{\sum_{j=0}^{m-1} c_j + B}\right) \blue{\ge 0} \text{ for } m>0.
  \end{equation*}
  We can thus apply \Cref{lemm:discrete_gronwall} to $X_n$
  and obtain after resubstituting the tilded quantities:
  \begin{equation*}
    X_n \le \exp\left(\sum_{m=0}^n \gamma_m \sigma_m\right) \cdot
    \left( \sum_{m=0}^n \tilde c_m + \tilde B \right)
    \le \exp\left(\sum_{m=0}^n \gamma_m \sigma_m\right) \cdot
    \left( \sum_{m=0}^n d_m + \sqrt{\sum_{m=0}^n c_m + B} \right).
  \end{equation*}
  Squaring and estimating the square of the sum yields the result.
 %  with $\sigma_m = (1 - \gamma_m)^{-1}$:
 %  \begin{equation*}
 %    X_n^2 \le 2 \exp\left(2\sum_{m=0}^n \gamma_m \sigma_m\right) \cdot
 %    \left[ \left(\sum_{m=0}^n d_m\right)^2 + \sum_{m=0}^n c_m + B \right].
 %  \end{equation*}
\end{proof}
\blue{With these technical lemmas, we can show}
%Before showing unique solvability of the discrete Navier-Stokes equations \eqref{eq:nav_stokes_discrete},
%we show 
stability \blue{of the discrete solutions} in different norms under different assumptions on $f$.
Solutions to the discrete problem also satisfy the same energy bounds as the weak solutions,
i.e., there holds the following proposition, see \cite[Lemma 5.1, Theorem A.1]{Chrysafinos2010}.

\begin{proposition}[Stability of discrete Navier-Stokes]
  Let $f \in L^2(I;\blue{H^{-1}(\Omega)^2})$, $u_0 \in H$ and $u_{kh} \in \Vkh$ satisfy
  \eqref{eq:nav_stokes_discrete}.
  Then there hold the bounds
  \blue{
  \begin{align*}
    \|\ukh\|_{L^\infty(I;L^2(\Omega))} 
    + \|\ukh\|_{\LtwoHone} \le C \left(\|u_0\|_H + \|f\|_{L^2(I;H^{-1}(\Omega))}\right),
    \quad \text{ if } q = 0,\\
    \left(\|\ukh\|_{L^\infty(I;L^2(\Omega))}\right)^{\frac{1}{2}} 
    + \|\ukh\|_{\LtwoHone} \le C \left(\|u_0\|_H + \|f\|_{L^2(I;H^{-1}(\Omega))}\right), \quad \text{ if } q \ge 1,
  \end{align*}
}
  with a constant $C$ depending on the domain $\Omega$ and the viscosity $\nu$.
\end{proposition}
Note that \blue{for the case $q \ge 1$}, contrary to the continuous setting
\blue{of} \Cref{prop:navier_stokes_solvability},
due to the exponent $\frac{1}{2}$ on the left hand side,
the above proposition states a bound for
$\|\ukh\|_{L^\infty(I;L^2(\Omega))}$, which depends on the squared norms of the data $u_0$ and $f$.
% \begin{proof}
%   The $L^2(I;H^1(\Omega))$ and $L^2(\Omega)$ estimates at the time nodes can be obtained from
%   \cite[Lemma 5.1]{Chrysafinos2010}. 
%   Moreover, since we only consider $\Omega \subset \R^2$, 
%   \cite[Theorem A.1]{Chrysafinos2010} holds, which yields the $L^\infty(I;L^2(\Omega))$
%   bound on the whole time interval. 
% \end{proof}
For the two low order cases $q=0$ and $q=1$, and $f\in L^2(I;L^2(\Omega)^2)$,
\cite[Lemma 5.1]{Chrysafinos2010} also shows an estimate of the form
  \begin{equation*}
    \max_{1\le m \le M} \|u_{kh,m}^-\|_{L^2(\Omega)} + \|\ukh\|_{\blue{\LtwoHone}}
    \le C_1 e^{C_2 T} \left(\|u_0\|_H + \|f\|_{L^2(I;L^2(\Omega))}\right).
  \end{equation*}
  With our discrete Gronwall estimate \Cref{thm:discrete_gronwall_quadlinear},
  we can generalize this result to $f \in L^1(I;L^2(\Omega)^2)$.
  Furthermore, by using the version of Gronwall's lemma presented here,
  contrary to \cite{Chrysafinos2010}, the bound does not grow exponentially in $T$. This is in
  agreement with the result of \Cref{prop:navier_stokes_solvability} in the continuous setting.
  We obtain the following result, which now yields an estimate for the norms of $\ukh$, that is linear
  in terms of the data.
\begin{theorem}\label{thm:stability_discrete_navier_stokes}
  Let $f \in L^1(I;L^2(\Omega)^2) + L^2(I;\blue{H^{-1}(\Omega)^2})$, $u_0 \in H$ and $u_{kh} \in \Vkh$ satisfy
  \cref{eq:nav_stokes_discrete} for either $q=0$ or $q=1$. Then there holds the bound
  \begin{equation*}
    \|\ukh\|_{L^\infty(I;L^2(\Omega))} + \blue{\sqrt{\nu}} \|\ukh\|_{\blue{\LtwoHone}} 
    \le C \left(\|u_0\|_H + \|f\|_{L^1(I;L^2(\Omega)) + L^2(I;\blue{H^{-1}(\Omega)})}\right),
  \end{equation*}
  \blue{with a constant $C$ only depending on $\Omega,\nu$. If $f \in L^1(I;L^2(\Omega)^2)$, then 
  $C$ only depends on $\Omega$.}
\end{theorem}
\begin{proof}
  We first prove the result for $f \in L^1(I;L^2(\Omega)^2)$ and remark at the end, 
  which modifications are needed to also cover the case, \blue{where $f$ is partly in $L^2(I;\HmOd)$.}
  \blue{Note that in the special case of $f \in L^2(I;\HmOd)$ and $q=0$, this theorem is precisely the 
  first statement of \cite[Lemma 5.1]{Chrysafinos2010}.}
  For notational simplicity, we use the convention $u_{kh,0}^- := u_0$ and accordingly
  $[\ukh]_0 = u_{kh,0}^+ - u_0$.
  By testing \eqref{eq:nav_stokes_discrete} with $\ukh|_{I_m}$ and using \Cref{lemm:chat}
  we arrive on each time interval at:
  \begin{equation*}
    \blue{\ImOprod{\partial_t \ukh,\ukh}} + \nu\ImOprod{\nabla \ukh,\nabla \ukh}
    + \Oprod{[\ukh]_{m-1},u_{kh,m-1}^+} = \ImOprod{f,\ukh}.
  \end{equation*}
 %  Integrating the term containing the time derivative yields
 %  \begin{equation*}
 %    \frac{1}{2} \|u_{kh,m}^-\|_{L^2(\Omega)}^2 - \frac{1}{2} \|u_{kh,m-1}^+\|_{L^2(\Omega)}^2
 %    + \nu\|\nabla \ukh\|_{\LtwomLtwo}^2
 %    + \Oprod{[\ukh]_{m-1},u_{kh,m-1}^+} = \ImOprod{f,\ukh}.
 %  \end{equation*}
  Applying \Cref{lemm:jump_reorder} gives
  \begin{equation*}
    \frac{1}{2} \|u_{kh,m}^-\|_{L^2(\Omega)}^2
    + \frac{1}{2} \|[\ukh]_{m-1}\|_{L^2(\Omega)}^2
    + \nu\|\nabla \ukh \|_{\LtwomLtwo}^2
    = \frac{1}{2} \|u_{kh,m-1}^-\|_{L^2(\Omega)}^2
    + \ImOprod{f,\ukh}.
  \end{equation*}
  Multiplying by two and summing up the identity over the intervals $1,...,n$, we obtain
  \begin{equation}\label{eq:nav_stokes_discrete_tested}
    \|u_{kh,n}^-\|_{L^2(\Omega)}^2 
    + \sum_{m=1}^n \|[\ukh]_{m-1}\|_{L^2(\Omega)}^2 
    + 2\nu \|\nabla \ukh\|^2_{L^2(I_m\times \Omega)}
    = \|u_{kh,0}^-\|_{L^2(\Omega)}^2 
    + 2\sum_{m=1}^n \IjOprod{f,\ukh}{m}.
  \end{equation}
  From this point on, we have to treat the two cases $q=0$ and $q=1$ separately.\\
  \textit{Case 1:} If $q=0$, then it holds 
  $\|\ukh\|_{L^\infty(I_m;L^2(\Omega))} = \|u_{kh,m}^-\|_{L^2(\Omega)}$.
  Thus with \blue{Hölder's} inequality, the terms in the last sum of \eqref{eq:nav_stokes_discrete_tested} 
  can be estimated as
  \begin{equation*}
    \IjOprod{f,\ukh}{m}
    \le \|f\|_{L^1(I_m;L^2(\Omega))}\|\ukh\|_{L^\infty(I_m;L^2(\Omega))}
    \le \|f\|_{L^1(I_m;L^2(\Omega))}\|u_{kh,m}^-\|_{L^2(\Omega)}.
  \end{equation*}
  Hence, from \eqref{eq:nav_stokes_discrete_tested} we obtain the following
  \begin{equation*}
    \|u_{kh,n}^-\|_{L^2(\Omega)}^2 
    + \sum_{m=1}^n \|[\ukh]_{m-1}\|_{L^2(\Omega)}^2 
    + 2 \nu \|\nabla \ukh\|^2_{L^2(I_m\times \Omega)}
    \le \|u_{0}\|_{L^2(\Omega)}^2 
    + 2\sum_{m=1}^n \|f\|_{L^1(I_m;L^2(\Omega))}\|u_{kh,m}^-\|_{L^2(\Omega)}.
  \end{equation*}
  An application of \Cref{thm:discrete_gronwall_quadlinear} proves the assertion.\\
  \textit{Case 2:} If $q=1$, then the $L^\infty(I_m)$ norm can be estimated by the evaluation at 
  the two endpoints of the interval:
  \begin{equation*}
    \|\ukh\|_{L^\infty(I_m;L^2(\Omega))} 
    \le \max \lbrace \|u_{kh,m}^-\|_{L^2(\Omega)},\|u_{kh,m-1}^+\|_{L^2(\Omega)}\rbrace.
  \end{equation*}
  With triangle inequality we can estimate the right sided limit in terms of a left sided limit and a 
  jump:
  \begin{equation*}
    \|\ukh\|_{L^\infty(I_m;L^2(\Omega))} 
    \le \max \lbrace \|u_{kh,m}^-\|_{L^2(\Omega)},
    \|u_{kh,m-1}^-\|_{L^2(\Omega)}+\|[\ukh]_{m-1}\|_{L^2(\Omega)}\rbrace.
  \end{equation*}
  Defining $x_n^2 := \|u_{kh,n}^-\|_{L^2(\Omega)}^2 + \sum_{m=1}^n \|[\ukh]_{m-1}\|_{L^2(\Omega)}^2$
  for $n=1,...,M$
  yields from \eqref{eq:nav_stokes_discrete_tested} the estimate
  \begin{equation*}
    x_n^2 + 2\nu \sum_{m=1}^n \|\nabla \ukh \|_{L^2(I_m;L^2(\Omega))}^2
    \le \|u_{0}\|_{L^2(\Omega)}^2 + 2\sum_{m=1}^n \|f\|_{L^1(I_m;L^2(\Omega))} (x_m + x_{m-1}),
  \end{equation*}
  or after an index shift
  \begin{equation*}
    x_n^2 + 2 \nu \sum_{m=1}^n \|\nabla \ukh \|_{L^2(I_m;L^2(\Omega))}^2
    \le \|u_{0}\|_{L^2(\Omega)}^2 
    + 2\|f\|_{L^1(I_1;L^2(\Omega))} \|u_0\|
    + 2\sum_{m=1}^n \|f\|_{L^1(I_m \cup I_{m+1};L^2(\Omega))} \cdot x_m.
  \end{equation*}
  In order to shorten the notation, we have added here a term $\|f\|_{L^1(I_{n+1};L^2(\Omega))}x_n$,
  where in case $n=M$ we use the convention $I_{M+1} = \emptyset$.
  Again we can apply \Cref{thm:discrete_gronwall_quadlinear}, which yields the estimate
  \begin{equation*}
    x_n^2 + 2 \nu \sum_{m=1}^n \|\nabla \ukh \|_{L^2(I_m;L^2(\Omega))}^2
    \le 2\left[\left(2 \sum_{m=1}^n \|f\|_{L^1(I_m \cup I_{m+1};L^2(\Omega))} \right)^2 
    + \|u_0\|_{L^2(\Omega)}^2 
    + 2\|f\|_{L^1(I_1;L^2(\Omega))} \|u_0\|\right].
  \end{equation*}
 Using Young's inequality, to estimate the term $2\|f\|_{L^1(I_1;L^2(\Omega))} \|u_0\|$,
 concludes the proof.
 In case $f = f_1 + f_2$, with $f_1 \in \blue{L^1(I;L^2(\Omega)^2)}$, $f_2 \in L^2(I;\blue{\HmOd})$, 
 the contribution of $f_1$ can be treated as above. For the $f_2$ contribution, in 
 \eqref{eq:nav_stokes_discrete_tested} one applies \blue{Hölder's and Young's} inequalities, and absorbs
 the $\|\nabla \ukh\|_{L^2(I_m;L^2(\Omega))}$ term to the left. 
 Note that in both cases $q=0,1$, in the estimates before the application of 
 \Cref{thm:discrete_gronwall_quadlinear},
 no $x_m^2$ terms are summed on the right hand side, thus
 no exponential dependency is introduced in the final estimate.
\end{proof}
The techniques presented above now allow us to show the \blue{existence of solutions to} 
the discrete equations \eqref{eq:nav_stokes_discrete}.
\begin{theorem}
  Let $u_0 \in H$ and $f \in L^1(I;L^2(\Omega)^2) + L^2(I;\blue{\HmOd})$.
  Let further either $q=0$ or $q=1$. 
  Then there exists a solution $\ukh \in \Vkh$ of \eqref{eq:nav_stokes_discrete}.
\end{theorem}
\begin{proof}
  \blue{The existence} can be shown by applying a standard fixpoint argument, using the stability result of 
  \Cref{thm:stability_discrete_navier_stokes}.
\end{proof}
\blue{
\begin{remark}
  Note that uniqueness of solutions $\ukh$ to \eqref{eq:nav_stokes_discrete} is not 
  known at this point. We shall show this later in \Cref{thm:discrete_nav_stokes_unique},
  as we first need to obtain a first convergence result, that allows us to apply
  \Cref{thm:discrete_gronwall_quadlinear}.
\end{remark}
}
\subsection{$L^2(I;H^1(\Omega))$ error estimates}
\blue{With the stability results presented thus far, we can show a first error estimate 
in the norm of $L^2(I;H^1(\Omega))$. For this norm,}
% In this work, we give a positive answer for the $L^2(I;H^1(\Omega))$ (\Cref{thm:nav_stokes_l2h1})
% and $L^\infty(I;L^2(\Omega))$ (\Cref{thm:nav_stokes_bestapprox}) norms.
% For the $L^2(I;L^2(\Omega))$ error, a duality approach would require $L^\infty(I;H^1(\Omega))$ stability
% for either $\ukh$ or the dual state $\zkh$, which to the best of the authors knowledge has not been 
% shown yet.
%For the $L^2(I;H^1(\Omega))$ error, 
\cite[Theorem 5.2]{Chrysafinos2010} shows a result,
estimating the error for the Navier-Stokes equations in terms of 
the error of a Stokes problem. 
The discrete Stokes problem there is defined with right hand side
$\partial_t u - \nu \Delta u$, which corresponds to
$f - \nabla p - (u\cdot \nabla) u$, i.e., the pressure is included on the right.
This means, that when applying the corresponding orthogonality relations, a pressure term remains.
In the following result, we use a different right hand side for the discrete Stokes problem,
yielding the following error estimate for the Navier-Stokes equations in $L^2(I;H^1(\Omega))$.
% in \cite{Chrysafinos2010}, which is not the case in our approach.
% We obtain the following error estimate in the $L^2(I;H^1(\Omega))$ norm.
% We give a slightly different proof, using a different right hand side for the discrete Stokes problem.
% Note that the projection operator $\mathbb P_h$ in 
% \cite{Chrysafinos2010} is chosen slightly differently than $\ukht$ used in the present work.
% In fact, $\mathbb P_h u$  is chosen as the solution to a fully discrete Stokes problem 
% with right hand side $\partial_t u - \nu \Delta u$, which corresponds to
% $f - \nabla p - (u\cdot \nabla) u$, i.e. the pressure is included on the right.
% This means, that when applying the corresponding orthogonality relations, a pressure term remains
% in \cite{Chrysafinos2010}, which is not the case in our approach.
% We obtain the following error estimate in the $L^2(I;H^1(\Omega))$ norm
\begin{theorem}\label{thm:nav_stokes_l2h1}
  Let $f \in L^2(I;L^2(\Omega)^2)$, $u_0 \in V$ and let $(u,p)$, $(\ukh,\pkh)$ \blue{solve the} continuous and
  fully discrete Navier-Stokes equations \eqref{eq:nav_stokes_weak_with_pressure} and
\eqref{eq:nav_stokes_discrete_with_pressure} respectively.
\blue{Further let $k$ be small enough.}
Then for any $\chi_{kh} \in \Vkh$, there holds 
\begin{align*}
  \|\nabla(u-\ukh)\|_{L^2(I\times\Omega)} & \le C \left(\| \nabla(u - \chi_{kh})\|_{L^2(I\times\Omega)} 
   + \|\nabla(u - \pi_\tau u)\|_{L^2(I\times \Omega)} + \|\nabla(u - R_h^S(u,p))\|_{L^2(I\times \Omega)}\right).
\end{align*}
\end{theorem}
\begin{proof}
   The proof uses the same arguments as the one of \cite[Theorem 5.2]{Chrysafinos2010}. We repeat the main
   steps, in order to motivate, why due to the choice of our \blue{instationary} Stokes projection, no error term for 
   the pressures arises on the right hand side.
   \blue{We denote by $e:= u - \ukh$ the error, which we want to estimate,
  and introduce $(\ukht,\pkht) \in \Xkh$ as the instationary Stokes-projection of $(u,p)$, i.e., the solution to
  \begin{equation}\label{eq:discrete_aux_stokes}
    B ((\ukht,\pkht),(\phikh,\psikh)) = \IOprod{f,\vkh} 
    - \cIO{u,u,\vkh} + \Oprod{u_0,v_{kh,0}^+} 
    \quad \text{for all } (\phikh, \psikh) \in \Xkh.
  \end{equation}
  Note that for the definition of this equation, we can equivalently choose $c$ or $\hat c$, as we have
  $u$ with $\nabla \cdot u = 0$ in its first argument.
  There exists a solution $(\ukht,\pkht)$ to this discrete problem,
  since the assumptions on $f$ and $u_0$ yield $(u\cdot \nabla) u \in L^2(I;L^2(\Omega)^2)$,
  see \Cref{thm:nonlinearity_higher_regularity}.
  Furthermore, due to \cite[Theorem 4.10]{Chrysafinos2010}, the solution $\ukht$ to the above discrete Stokes
  problem satisfies a bound
  \begin{equation}\label{eq:discrete_stokes_linfh1bound}
    \|\ukht\|_{L^\infty(I;H^1(\Omega))} \le C,
  \end{equation}
  with $C$ depending on $f,u_0$, but independent of $k,h$.
  We split the velocity error into two components $e = \xi + \eta_{kh}$ where
  $\xi: = u - \ukht$ and $\eta_{kh}: = \ukht - \ukh$.
  The error in the pressure will be denoted by $r := p - \pkh$, and we will split it into the
  same contributions, $r = \omega + \kappa_{kh}$, with $\omega := p - \tilde p_{kh}$ 
  and $\kappa_{kh} := \tilde p_{kh} - \pkh$.
  The error of the Stokes projection $\xi$ is then bounded by \Cref{thm:stokesl2l2_l2h1} since by
  \Cref{thm:nonlinearity_higher_regularity} it holds $f - (u\cdot \nabla) u \in L^2(I;L^2(\Omega)^2)$.
  It remains to estimate $\eta_{kh}$.
  }
%   We denote by $(\ukht,\pkht) \in \Xkh$ the solution to the discrete Stokes problem
%   \begin{equation*}
%     B((\ukht,\pkht),(\phikh,\psikh)) =
%     - \cIO{u,u,\phikh} + \Oprod{u_0,\phi_{kh,0}^+} + \IOprod{f,\phikh} \quad \text{for all }
%     (\phikh,\psikh) \in \Xkh,
%   \end{equation*}
%   and introduce the error splitting
%   $\xi := u-\ukht$, $\eta_{kh} := \ukht - \ukh$, $\omega:=p-\pkht$ and $\kappa_{kh}:=\pkht-\pkh$.
%   Due to the assumptions on $u_0$, $f$ and \Cref{thm:nonlinearity_higher_regularity}, 
%   \blue{the solution } $\ukht$ \blue{to the above discrete Stokes problem} is 
%   bounded in $L^\infty(I;H^1(\Omega))$ independent of $k,h$, \blue{see \cite[Theorem 4.10]{Chrysafinos2010}}.
   By definition, we have the following Galerkin orthogonalities: for all $(\phikh,\psikh) \in \Xkh$ it holds
   \begin{align*}
      B((u-\ukh,p-\pkh),(\phikh,\psikh))+\chatIO{u,u,\phikh}-\chatIO{\ukh,\ukh,\phikh} &= 0, \\
      B((\xi,\omega),(\phikh,\psikh)) &= 0.
   \end{align*}
   Subtracting the two and choosing $(\phikh,\psikh) = (\eta_{kh},\kappa_{kh})$ yields
   \begin{equation}\label{eq:etakhequation}
      B((\eta_{kh},\kappa_{kh}),(\eta_{kh},\kappa_{kh}))
      +\chatIO{u,u,\eta_{kh}}-\chatIO{\ukh,\ukh,\eta_{kh}} = 0.
   \end{equation}
   The two trilinear forms in \eqref{eq:etakhequation} satisfy the relation
   \begin{equation*}
      \chatIO{u,u,\eta_{kh}}-\chatIO{\ukh,\ukh,\eta_{kh}}
      = \chatIO{\xi,u,\eta_{kh}}+\chatIO{\ukht,\xi,\eta_{kh}}
      +\chatIO{\eta_{kh},\ukht,\eta_{kh}}+\chatIO{\ukh,\eta_{kh},\eta_{kh}},
   \end{equation*}
   Due to \Cref{lemm:chat}, the last term on the right hand side vanishes, and for the remaining ones,
   we obtain by \blue{Hölder's and Young's} inequalities
   \begin{align*}
      \chatIO{\xi,u,\eta_{kh}} &\le \sum_{m=1}^M \|\xi\|_{L^2(I_m;H^1(\Omega))}
      \|u\|_{L^\infty(I;H^1(\Omega))}\|\eta_{kh}\|_{L^2(I_m;H^1(\Omega))}\\
      &\le \sum_{m=1}^M C\|\xi\|_{L^2(I_m;H^1(\Omega))}^2
      \|u\|_{L^\infty(I;H^1(\Omega))}^2 + \frac{\nu}{4} \|\eta_{kh}\|_{L^2(I_m;H^1(\Omega))}^2,\\
%    \end{align*}
%    \begin{align*}
      \chatIO{\xi,\ukht,\eta_{kh}} &\le \sum_{m=1}^M \|\xi\|_{L^2(I_m;H^1(\Omega))}
      \|\ukht\|_{L^\infty(I;H^1(\Omega))}\|\eta_{kh}\|_{L^2(I_m;H^1(\Omega))}\\
      &\le \sum_{m=1}^M C\|\xi\|_{L^2(I_m;H^1(\Omega))}^2
      \|\ukht\|_{L^\infty(I;H^1(\Omega))}^2 + \frac{\nu}{4} \|\eta_{kh}\|_{L^2(I_m;H^1(\Omega))},\\
   \end{align*}
   \begin{align*}
      \chatIO{\eta_{kh},\ukht,\eta_{kh}} &\le \sum_{m=1}^M \int_{I_m} \|\eta_{kh}\|_{H^1(\Omega)}^{\frac{3}{2}}
      \|\eta_{kh}\|_{L^2(\Omega)}^{\frac{1}{2}} \|\ukht\|_{H^1(\Omega)}\ dt\\
      & \le \sum_{m=1}^M C\|\ukht\|_{L^\infty(I;H^1(\Omega))}^4 k_m \|\eta_{kh}\|_{L^\infty(I_m;L^2(\Omega))}^2 
      + \frac{\nu}{4} \|\eta_{kh}\|_{L^2(I_m;H^1(\Omega)}^2 .
   \end{align*}
   Due to $\ukh$ and $\ukht$ both being discretely divergence free, all pressure contributions in 
   \eqref{eq:etakhequation} vanish, and with the above estimates, after absorbing all
   $\frac{\nu}{4} \|\eta_{kh}\|_{L^2(I_m;H^1(\Omega))}$ terms to the left, we obtain
   \begin{align*}
      \sum_{m=1}^M \blue{\ImOprod{\partial_t \eta_{kh},\eta_{kh}}} + \frac{\nu}{4}
      \ImOprod{\nabla \eta_{kh},\nabla \eta_{kh}}
      + \sum_{m=1}^M \Oprod{[\eta_{kh}]_{m-1},\eta_{kh,m-1}^+}\\
      \le C\left(\|u\|_{L^\infty(I;H^1(\Omega))}^2 +\|\ukht\|_{L^\infty(I;H^1(\Omega))}^2\right)
      \|\xi\|_{L^2(I;H^1(\Omega))}^2
      + C\|\ukht\|_{L^\infty(I;H^1(\Omega))}^4 \sum_{m=1}^M k_m \|\eta_{kh}\|_{L^\infty(I_m;L^2(\Omega))}^2.
   \end{align*}
   To abbreviate notation we have used here $[\eta_{kh}]_0 := \eta_{kh,0}^+$.
   Choosing $k$ small enough,
   \blue{and using the boundedness of $\|\ukht\|_{L^\infty(I;H^1(\Omega))}$,
     due to \eqref{eq:discrete_stokes_linfh1bound} as well as 
   of $\|u\|_{L^\infty(I;H^1(\Omega))}$, shown in \Cref{thm:h2_regularity},}
   allows us to apply the discrete Gronwall \blue{\Cref{thm:discrete_gronwall_quadlinear}} as in 
   \Cref{thm:stability_discrete_navier_stokes}, which concludes the proof.
\end{proof}

\subsection{Uniqueness and strong stability of discrete solutions}

We next show the \blue{uniqueness and } stability of the discrete solution $\ukh$ in stronger norms,
comparable to the continuous result of \Cref{thm:h2_regularity}.
\blue{These results will be obtained},
by applying the discrete Gronwall \Cref{thm:discrete_gronwall_quadlinear},
to which end we need the two following technical lemmas, guaranteeing, that the coefficients
in the Gronwall lemma become arbitrarily small, uniformly in $m$, as $(k,h) \to 0$.

\begin{lemma}\label{lemm:vanishing_support}
  Let $I \subset \mathbb R$ a bounded interval, $\varepsilon > 0$ and $\zeta \in L^1(I)$.
  Then 
  \begin{equation*}
    \sup_{x \in \bar I} \int_x^{x+\varepsilon} |\zeta(s)| \ ds \xrightarrow{\varepsilon \to 0} 0.
  \end{equation*}
  To keep notation simple, we formally extend $\zeta$ to $[0,T+ \varepsilon]$ by values $0$ such that
  the integration is well defined.
\end{lemma}
\begin{proof}
  The proof relies on dominated convergence and Dini's theorem, see, e.g., \cite[p. 125]{Brokate2015}.
  We first define for fixed $\varepsilon$ the function 
  $\sigma_\varepsilon: \bar I \to \mathbb R, x \mapsto \int_x^{x+\varepsilon} |\zeta(s)| \ ds$.
  Note that for each $\varepsilon > 0$ this function $\sigma_\varepsilon$ is continuous. To see this,
  let $x \in I$ be fixed and let $y \to x$. W.l.o.g. let $y>x$:
  \begin{equation*}
  |\sigma_\varepsilon(x) - \sigma_\varepsilon(y)| 
  \le \left|\int_x^{y} |\zeta(s)| \ ds + \int_{x+\varepsilon}^{y+\varepsilon} |\zeta(s)| \ ds \right|
  = \left| \int_{I} \chi_{x,y}(s)|\zeta(s)| \ ds \right|
  \end{equation*}
  where $\chi_{x,y}(s) = 1$ if $s \in (x,y) \cup (x+\varepsilon, y + \varepsilon)$, and $0$ otherwise.
  For a fixed $s$, the integrand $\chi_{x,y}(s)|\zeta(s)| \to 0$ as $y \to x$, thus the integrand
  converges pointwise to $0$. By the dominated convergence theorem, this means that the integral converges to
  $0$. For fixed $\varepsilon> 0$ this shows the continuity of $\sigma_\varepsilon$.
  Moreover, by the same argument we can show, that for fixed $x$ and $\varepsilon \to 0$ it holds
  $\sigma_\varepsilon(x) \to 0$. Thus the sequence $\sigma_\varepsilon$ converges pointwise to $0$.
  Note moreover, that for $\delta > \varepsilon$, and fixed $x$, it holds $\sigma_\delta(x) \ge \sigma_\varepsilon(x)$,
  as the integration covers a larger interval. We thus have shown, that when 
  $\varepsilon \to 0$ monotonically, also this pointwise convergence of $\sigma_\varepsilon(x)$ is monotone.
  Hence we can apply Dini's theorem, see \cite[p. 125]{Brokate2015}, to obtain
  \begin{equation*}
    \|\sigma_\varepsilon\|_{L^\infty(I)} \rightarrow 0, \quad \text{ as } \varepsilon \rightarrow 0.
  \end{equation*}
\end{proof}
\begin{lemma}\label{lemm:vanishing_support_u}
  Let $u \in L^2(I;H^1_0(\Omega)^2)$, and $u_{kh} \in \Ukh $ such that
  $\|u - u_{kh}\|_{L^2(I;H^1(\Omega))} \to 0$ as $(k,h) \to 0$. Then it holds
  \begin{align*}
    \sup_{1 \le m \le M} \int_{I_{m,k}} \|\nabla u\|_{L^2(\Omega)}^2 \ dt  \xrightarrow{(k,h) \to 0} 0
    \quad \text{and} \quad
    \sup_{1 \le m \le M} \int_{I_{m,k}} \|\nabla u_{kh}\|_{L^2(\Omega)}^2 \ dt \xrightarrow{(k,h) \to 0} 0.
  \end{align*}
\end{lemma}
\begin{proof}
  We first show the statement for $u$. To this end, note that for $k = \max_{1 \le m \le M} |I_{m,k}|$
  \begin{equation*}
      \sup_{1 \le m \le M} \int_{I_{m,k}} \|\nabla u\|_{L^2(\Omega)}^2 \ dt
      \le \sup_{x \in \bar I} \int_{x}^{x+k} \|\nabla u\|_{L^2(\Omega)}^2 \ dt \xrightarrow{k \to 0} 0
  \end{equation*}
  by \Cref{lemm:vanishing_support}.
  In order to show the result for the discrete solution $\ukh$, we cannot directly apply the previous 
  lemma, as $\ukh$ depends on $k$. We thus insert
  insert $\pm u$ and apply the triangle inequality to obtain
  \begin{align*}
      \sup_{1 \le m \le M} \int_{I_{m,k}} \|\nabla u_{kh}\|_{L^2(\Omega)}^2 \ dt
%       & \le 2\sup_{1 \le m \le M} \int_{I_{m,k}} \|\nabla u\|_{L^2(\Omega)}^2 \ dt
%       + 2\sup_{1 \le m \le M} \int_{I_{m,k}} \|\nabla (u - u_{kh})\|_{L^2(\Omega)}^2 \ dt\\
      & \le 2\sup_{1 \le m \le M} \int_{I_{m,k}} \|\nabla u\|_{L^2(\Omega)}^2 \ dt
      + 2\int_{I} \|\nabla (u - u_{kh})\|_{L^2(\Omega)}^2 \ dt.
  \end{align*}
  With the claim for $u$ and $u_{kh} \to u$ in $L^2(I;H^1(\Omega))$, we have shown the claim for $u_{kh}$.
\end{proof}
\blue{With these technical lemmas, we can show two results, the uniqueness of solutions $\ukh$ and 
  their boundedness in stronger norms. We first obtain the following uniqueness result.
  \begin{theorem}\label{thm:discrete_nav_stokes_unique}
    Let $f \in L^2(I;L^2(\Omega)^2)$ and $u_0 \in V$. Then for $(k,h)$ small enough, the solution
    $\ukh$ to \eqref{eq:nav_stokes_discrete} is unique.
  \end{theorem}
  \begin{proof}
    \blue{To show uniqueness, we assume two solutions $\ukh^1, \ukh^2$ and define $\ekh:= \ukh^1 - \ukh^2$.
  Testing the equations for $m=1,...,M$ with $\ekh \chi_m$, where $\chi_m$ denotes the characteristic
  function of the interval $I_m$, and subtracting them leads to
  \begin{equation*}
    \mathfrak B (\ekh,\ekh \chi_m) + \chatIO{\ukh^1,\ukh^1,\ekh \chi_m} - \chatIO{\ukh^2,\ukh^2,\ekh \chi_m} = 0
  \end{equation*}
  Adding and subtracting $\chatIO{\ukh^1,\ukh^2,\ekh \chi_m}$ and applying \Cref{lemm:chat} yields
  \begin{equation*}
    \mathfrak B (\ekh,\ekh \chi_m) + \chatIO{\ekh,\ukh^2,\ekh \chi_m}= 0.
  \end{equation*}
  Applying \Cref{lemm:jump_reorder}, the bilinear form can be written  as 
  \begin{equation*}
    \mathfrak B (\ekh,\ekh \chi_m) = 
    \frac{1}{2} \left(\|e_{kh,m}^-\|_{L^2(\Omega)}^2  
    + \|[\ekh]_{m-1}\|_{L^2(\Omega)}^2 - \|e_{kh,m-1}^-\|_{L^2(\Omega)}^2 \right)
    + \nu\|\nabla \ekh\|_{L^2(I_m;L^2(\Omega))}^2.
  \end{equation*}
  From the definition of $\hat c$, the estimates of \Cref{lemm:c_swap}, and Hölder's inequality in time,
  we obtain
  \begin{align*}
    \chatIO{\ekh,\ukh^2,\ekh \chi_m} &=
    \frac{1}{2} \left[\ImOprod{(\ekh \cdot \nabla)\ukh^2,\ekh} 
        - \ImOprod{(\ekh \cdot \nabla)\ekh,\ukh^2}\right]\\
                                     &\le C \|\ekh\|_{L^\infty(I_m;L^2(\Omega))}\|\nabla \ekh\|_{L^2(I_m;L^2(\Omega))}\| \nabla \ukh^2\|_{L^2(I_m;L^2(\Omega))}\\
    & \quad + C \|\ekh\|_{L^\infty(I_m;L^2(\Omega))}^{\frac{1}{2}}
      \|\nabla \ekh\|_{L^2(I_m;L^2(\Omega))}^{\frac{3}{2}}
        \|\ukh^2\|_{L^\infty(I_m;L^2(\Omega))}^{\frac{1}{2}} 
        \| \nabla \ukh^2\|_{L^2(I_m;L^2(\Omega))}^{\frac{1}{2}}.
  \end{align*}
  We can apply Young's inequality and obtain
  \begin{align*}
    \chatIO{\ekh,\ukh^2,\ekh \chi_m} &\le
    \frac{\nu}{2}\|\nabla \ekh\|_{L^2(I_m;L^2(\Omega))}^2
    + C \|\ekh\|_{L^\infty(I_m;L^2(\Omega))}^2 \left(1 + \|\ukh^2\|_{L^\infty(I_m;L^2(\Omega))}^2\right)
    \| \nabla \ukh^2\|_{L^2(I_m;L^2(\Omega))}^2.
  \end{align*}
  where the first summand can be absorbed into $\mathfrak B(\ekh,\ekh \chi_m)$.
  Thus on each time interval, there holds the bound
  \begin{align*}
    \frac{1}{2} \left(\|e_{kh,m}^-\|_{L^2(\Omega)}^2  
    + \|[\ekh]_{m-1}\|_{L^2(\Omega)}^2 - \|e_{kh,m-1}^-\|_{L^2(\Omega)}^2 
    + \nu\|\nabla \ekh\|_{L^2(I_m;L^2(\Omega))}^2\right)\\
    \le C \|\ekh\|_{L^\infty(I_m;L^2(\Omega))}^2 \left(1 + \|\ukh^2\|_{L^\infty(I_m;L^2(\Omega))}^2\right)
    \| \nabla \ukh^2\|_{L^2(I_m;L^2(\Omega))}^2.
  \end{align*}
  Multiplying the derived inequality by 2 and summing up the inequalities for $m=1,...,n$ yields 
  for any $n \in \{1,...,M\}$
  \begin{align*}
    \|e_{kh,n}^-\|_{L^2(\Omega)}^2 + \sum_{m=1}^n 
    \left(\|[\ekh]_{m-1}\|_{L^2(\Omega)}^2 + \nu \|\nabla \ekh\|_{L^2(I_m;L^2(\Omega))}^2\right)\\
    \le 
    C \sum_{m=1}^n \|\ekh\|_{L^\infty(I_m;L^2(\Omega))}^2 
    \left(1 + \|\ukh^2\|_{L^\infty(I_m;L^2(\Omega))}^2\right) 
    \| \nabla \ukh^2\|_{L^2(I_m;L^2(\Omega))}^2.
  \end{align*}
  Here we have used $e_{kh,0}^- = 0$ as $u_{kh}^1$ and $u_{kh}^2$ satisfy the same initial condition.
  As in the proof of \Cref{thm:stability_discrete_navier_stokes}, 
  $\|\ekh\|_{L^\infty(I_m;L^2(\Omega))}$ can be bounded by evaluations of one-sided limits at the time nodes.
  Due to $f \in L^2(I;L^2(\Omega)^2)$ and $u_0 \in V$, \Cref{thm:nav_stokes_l2h1} holds,
  and thus \Cref{lemm:vanishing_support_u} yields $\| \nabla \ukh^2\|_{L^2(I_m;L^2(\Omega))}^2 \to 0$
  uniformly in $m$ for $(k,h)\to 0$. Together with \Cref{thm:stability_discrete_navier_stokes}, bounding 
  $\|\ukh^2\|_{L^\infty(I_m;L^2(\Omega))}$, we can choose the discretization fine enough, such that
  $C (1 + \|\ukh^2\|_{L^\infty(I_m;L^2(\Omega))}^2)\| \nabla \ukh^2\|_{L^2(I_m;L^2(\Omega))}^2< \frac{1}{2}$
  for all $m$.
  All in all, as in \Cref{thm:stability_discrete_navier_stokes}, an application of 
  the Gronwall \Cref{thm:discrete_gronwall_quadlinear} shows that $\ekh = 0$,
  concluding the proof of uniqueness.
  }
  \end{proof}
}
We now turn toward proving the stability of $\ukh$ in stronger norms.
The proof follows the steps of the continuous result, shown in \cite[Chapter 3, Theorem 3.10]{Temam1977},
using the discrete analogons of the inequalities presented in \Cref{sect:space_discretization}, and 
\blue{the} discrete Gronwall \blue{\Cref{thm:discrete_gronwall_quadlinear}}.
For the application of the lemma, we need coefficients, that 
become small, as $k \to 0$. These coefficients depend on $\ukh$, and thus on $k,h$, 
\blue{hence we apply \Cref{thm:nav_stokes_l2h1} together with \Cref{lemm:vanishing_support_u},
  in order to have coefficients, that converge to 0 uniformly 
in $m$.}
%hence we need to assume a convergence result in $L^2(I;H^1(\Omega))$, in 
%order to have coefficients that converge to 0 uniformly in $m$.
%We shall show later on in \Cref{thm:nav_stokes_l2h1}, that such a convergence result indeed holds.
%Making this assumption here, allows us to present the stability result jointly with similar results.
There holds the following.
\begin{theorem}\label{thm:nav_stokes_discrete_h2}
   Let $f\in L^2(I;L^2(\Omega)^2)$ and $u_0 \in V$, 
   \blue{then for $(k,h)$ small enough, the unique solution $\ukh \in \Vkh$ to \eqref{eq:nav_stokes_discrete}
   satisfies}
%   and let the unique solution $\ukh \in \Vkh$ of
%   \eqref{eq:nav_stokes_discrete} converge to the continuous solution $u$ of \eqref{eq:nav_stokes_weak}
%   in $L^2(I;H^1(\Omega))$ as $(k,h) \to 0$. Then $\ukh$ satisfies
%
%   \begin{equation*}
%      \|\ukh\|_{L^\infty(I;H^1(\Omega))} + \|A_h \ukh\|_{L^2(I;L^2(\Omega))} 
%      \le C\left(\|f\|_{L^2(I;L^2(\Omega))},\|u_0\|_V\right).
%   \end{equation*}
\begin{align*}
   \|\ukh\|_{L^\infty(I;H^1(\Omega))} + \|A_h \ukh\|_{L^2(I;L^2(\Omega))}
   \le & C_1\exp{\left(C_2\|\ukh\|_{L^\infty(I;L^2(\Omega))}^2\|\nabla \ukh\|_{L^2(I;L^2(\Omega))}^2\right)}\\
    & \times \left(\|f\|_{L^2(I;L^2(\Omega))}+\|u_0\|_{H^1(\Omega)}\right),
\end{align*}
   with constants $C_1, C_2$ independent of $k,h$.
   The $L^\infty(I;L^2(\Omega))$ and $L^2(I;H^1(\Omega))$ norms of $\ukh$ can be estimated by
   the results of \Cref{thm:stability_discrete_navier_stokes}.
%   where the constant $C$ is independent on $k,h$ while depending continuously on the norms of $f$ and $u_0$
%   in the same fashion as the continuous result of \Cref{thm:h2_regularity}.
\end{theorem}
\begin{proof}
  \blue{For $m=1,...,M$,} we test \eqref{eq:nav_stokes_discrete} with $A_h \ukh |_{I_m}$, which yields
  \blue{with the convention $[\ukh]_0 := u_{kh,0}^+ - u_0$}
   \begin{align*}
     \blue{\ImOprod{\partial_t \ukh,A_h \ukh}} + \nu \ImOprod{\nabla \ukh,\nabla A_h \ukh}
     &\\
     + \blue{\Oprod{[\ukh]_{m-1},A_h u_{kh,m-1}^+}} + \chatIO{\ukh,\ukh,A_h \ukh \chi_{I_m}} 
     &= \ImOprod{f,A_h \ukh},% + \blue{\delta_{1m}}\Oprod{u_0,A_h u_{kh,0}^+},
   \end{align*}
%   using the Kronecker delta \blue{$\delta_{1m}$} to include the contribution of the initial data,
%   in case \blue{$m=1$}.
   \blue{From the above identity, for $m=2,...,M$,} the definition of $A_h$ gives
   \begin{equation}\label{eq:discrete_Autested}
     \begin{aligned}
       \blue{\ImOprod{\partial_t \nabla \ukh,\nabla \ukh}} + \nu \|A_h \ukh\|_{L^2(I_m;L^2(\Omega))}^2 
     &\\
     +\blue{\Oprod{[\nabla \ukh]_{m-1},\nabla u_{kh,m-1}^+}} + \chatIO{\ukh,\ukh,A_h \ukh \chi_{I_m}}
     &= \ImOprod{f,A_h \ukh}. %+ \blue{\delta_{1m}}\Oprod{u_0,A_h u_{kh,0}^+}.
     \end{aligned}
   \end{equation}
   Here the terms containing time derivatives and jumps can be combined according to \Cref{lemm:jump_reorder}.
   \blue{Some care has to be taken on the first time interval. It holds 
\begin{align*}
  \Oprod{u_0,A_h u_{kh,0}^+} &=- \Oprod{u_0, \Ph \Delta_h u_{kh,0}^+}=
   - \Oprod{\Ph u_0,\Delta_h u_{kh,0}^+} = \Oprod{\nabla \Ph u_0,\nabla u_{kh,0}^+}
\end{align*}
Thus on the first time interval, with the same arguments as in the proof of \Cref{lemm:jump_reorder}, it holds 
\begin{equation*}
  \IjOprod{\partial_t \ukh,A_h \ukh}{1} + \Oprod{u_{kh,0}^+ - u_0,A_h u_{kh,0}^+}
  = \frac{1}{2} \left(\|\nabla u_{kh,1}^-\|_{L^2(\Omega)}^2 + \|\nabla (u_{kh,0}^+-\Ph u_0)\|_{L^2(\Omega)}^2 
  - \|\nabla \Ph u_0\|_{L^2(\Omega)}^2\right).
\end{equation*}
For a more compact notation, we shall write $[\nabla \ukh]_0 := \nabla (u_{kh,0}^+-\Ph u_0)$.
Note the need to include the projection on the right hand side.
With this slight abuse of notation, \eqref{eq:discrete_Autested} also holds on the first time interval.
Since $u_0 \in V$ we can use the stability of $\Ph$ in $H^1$ for continuously divergence free functions,
see \cite[Lemma 5.4]{vexler_l2i_2023}, yielding an estimate 
\begin{equation}\label{eq:inititial_data_projection_stability}
\|\nabla \Ph u_0\|_{L^2(\Omega)} \le C\|\nabla u_0\|_{L^2(\Omega)}.
\end{equation}
}
   %We first estimate the terms containing the data of the right hand side $f$ and initial data $u_0$.
\blue{For the terms of \eqref{eq:discrete_Autested} involving the right hand side $f$}, it holds
\begin{equation*}
   \ImOprod{f,A_h \ukh} \le \|f\|_{L^2(I_m;L^2(\Omega))}\|A_h \ukh\|_{L^2(I_m;L^2(\Omega))}
   \le C \|f\|_{L^2(I_m;L^2(\Omega))}^2 + \frac{\nu}{4} \|A_h \ukh\|_{L^2(I_m;L^2(\Omega))}^2,
\end{equation*}
where we can absorb the $A_h \ukh$ term.
%The contribution of the initial data is estimated via
%\begin{align*}
%  \Oprod{u_0,A_h u_{kh,0}^+} &=- \Oprod{u_0, \Ph \Delta_h u_{kh,0}^+}=
%   - \Oprod{\Ph u_0,\Delta_h u_{kh,0}^+} = \Oprod{\nabla \Ph u_0,\nabla u_{kh,0}^+}\\
%   &\le \|\nabla \Ph u_0\|_{L^2(\Omega)} \|\nabla u_{kh,0}^+\|_{L^2(\Omega)}.
%\end{align*}
%In case $u_0 \in V$, we can use the stability of $\Ph$ in $H^1$ for continuously divergence free functions,
%see \cite[Lemma 5.4]{vexler_l2i_2023}, proving
%\begin{equation*}
%   \Oprod{u_0,A_h u_{kh,0}^+} 
%   \le \|\nabla u_0\|_{L^2(\Omega)} \|\nabla u_{kh,0}^+\|_{L^2(\Omega)}.
%\end{equation*}
We now turn towards estimating the trilinear form remaining in \eqref{eq:discrete_Autested},
which poses the main difficulty of this proof.
Due to \Cref{lemm:c_swap_no_div}, the trilinear term satisfies the expression
\begin{equation*}
   \chatIO{\ukh,\ukh,A_h \ukh \chi_{I_m}} = \ImOprod{(\ukh \cdot \nabla)\ukh,A_h \ukh} 
   + \frac{1}{2} \ImOprod{\nabla \cdot \ukh,\ukh \cdot A_h \ukh},
\end{equation*}
which we can estimate with Hölder's inequality by
\begin{equation*}
   \chatIO{\ukh,\ukh,A_h \ukh \chi_{I_m}} \le C
   \int_{I_m} \|\ukh\|_{L^\infty(\Omega)} \|\nabla \ukh\|_{L^2(\Omega)} \|A_h \ukh\|_{L^2(\Omega)} \ dt.
\end{equation*}
Using \eqref{eq:discrete_stokes_gagliardo_nirenberg} and Young's inequality, we obtain the estimate
\begin{equation*}
  \chatIO{\blue{\ukh,\ukh,A_h \ukh \chi_{I_m}}} \le  
   \int_{I_m} C \|\ukh\|_{L^2(\Omega)}^2 \|\nabla \ukh\|_{L^2(\Omega)}^4 
   + \frac{\nu}{4} \|A_h \ukh\|_{L^2(\Omega)}^2 \ dt.
\end{equation*}
The latter term can be absorbed, hence we proceed by discussing the first one. It holds
\begin{equation*}
   \int_{I_m} C \|\ukh\|_{L^2(\Omega)}^2 \|\nabla \ukh\|_{L^2(\Omega)}^4 \ dt
   \le C\|\ukh\|_{L^\infty(I;L^2(\Omega))}^2\|\nabla \ukh\|_{L^2(I_m;L^2(\Omega))}^2
   \|\nabla \ukh\|_{L^\infty(I_m;L^2(\Omega))}^2.
\end{equation*}
Let us introduce $\gamma_m := C\|\ukh\|_{L^\infty(I;L^2(\Omega))}^2\|\nabla \ukh\|_{L^2(I_m;L^2(\Omega))}^2$.
\blue{All in all, we have derived from \eqref{eq:discrete_Autested} an estimate of the form
  \begin{align*}
    \frac{1}{2} \left(\|\nabla u_{kh,m}^-\|_{L^2(\Omega)}^2  + \|[\nabla \ukh]_{m-1}\|_{L^2(\Omega)}^2 
    - \|\nabla u_{kh,m-1}^-\|_{L^2(\Omega)}^2 \right)
    + \frac{1}{2}\nu \|A_h \ukh\|_{L^2(I_m;L^2(\Omega))}^2\\
    \le
    \|f\|_{L^2(I_m;L^2(\Omega))}^2 
    + \gamma_m \|\nabla \ukh\|_{L^\infty(I_m;L^2(\Omega))}^2,
  \end{align*}
  with the usual modifications on the first time interval.
  Multiplying this by 2 and summing up from $m=1,...,n$ yields for any $n \in \{1,...,M\}$ an 
  estimate of the form
  \begin{align*}
    \|\nabla u_{kh,m}^-\|_{L^2(\Omega)}^2  + \sum_{m=1}^n\left(\|[\nabla \ukh]_{m-1}\|_{L^2(\Omega)}^2 
    + \nu \|A_h \ukh\|_{L^2(I_m;L^2(\Omega))}^2\right)\\
    \le
    C \|\nabla u_0\|_{L^2(\Omega)}^2
    + \sum_{m=1}^n \left(\|f\|_{L^2(I_m;L^2(\Omega))}^2 
    + 2\gamma_m \|\nabla \ukh\|_{L^\infty(I_m;L^2(\Omega))}^2\right),
  \end{align*}
  where due to \eqref{eq:inititial_data_projection_stability} we have dropped the projection $\Ph$ of the 
  initial data on the right hand side.
  Since $f \in L^2(I;L^2(\Omega)^2)$ and $u_0 \in V$, \Cref{thm:nav_stokes_l2h1} shows
  that $\ukh \to u$ in $L^2(I;H^1(\Omega))$ as $(k,h)\to 0$.}
%\blue{Using the assumption of this theorem, that $\ukh \to u$ in $L^2(I;H^1(\Omega))$ as $(k,h)\to 0$,}
%\Red{Using the $L^2(I;H^1(\Omega))$ error estimates of \Cref{thm:nav_stokes_l2h1}}
  \blue{Hence, due to} \Cref{lemm:vanishing_support_u}, \blue{and \Cref{thm:stability_discrete_navier_stokes},}
it holds $\gamma_m \to 0$ uniformly, as $(k,h) \to 0$.
This implies that by following the same steps as the proof of \Cref{thm:stability_discrete_navier_stokes},
we obtain the result as a consequence of the discrete Gronwall \Cref{thm:discrete_gronwall_quadlinear}.
\end{proof}

\section{Error estimates}\label{sec:ErrEst}
\subsection{Duality based best approximation type estimates}
\blue{In this concluding section of our work, we will show the main results, which are the 
  $L^\infty(I;L^2(\Omega))$ and $L^2(I;L^2(\Omega))$ error estimates. Their proofs are based 
  on duality arguments, and hence we begin this section by showing two stability results for 
  discrete dual equations. Let us first motivate the specific dual problem, that we will consider.}
%\textcolor{red}{We conclude this section by showing two stability results for discrete dual equations.}
%Their formulation is motivated by the following considerations.
Due to the nonlinear structure of the Navier-Stokes equations, there does not hold a Galerkin orthogonality
with respect to the bilinear form $B$, i.e.,
for solutions $(u,p)$ of \eqref{eq:nav_stokes_weak_with_pressure} and their discrete counterparts
$(\ukh,\pkh)$ of \eqref{eq:nav_stokes_discrete_with_pressure}, it holds for test functions 
$(\vkh,\qkh) \in \Xkh$:
\begin{equation*}
  B((u-\ukh,p-\pkh),(\vkh,\qkh)) = - \chatIO{u,u,\vkh} + \chatIO{\ukh,\ukh,\vkh},
\end{equation*}
where the right hand side is nonzero in general.
Since we want to use this orthogonality relation after testing the dual equation with $u-\ukh$,
we need to reformulate the trilinear terms, such that $u-\ukh$ occurs linearly.
To this end, we use the identity
\begin{equation}\label{eq:linearize_nav_stokes}
  \chatIO{u,u,\vkh} - \chatIO{\ukh,\ukh,\vkh}
  = \chatIO{\uukh,u-\ukh,\vkh} + \chatIO{u-\ukh,\uukh,\vkh},
\end{equation}
where we linearize around the average of continuous and discrete solutions to the Navier-Stokes 
equations
\begin{equation}\label{eq:def_linearizationpoint}
  \uukh := \frac{1}{2} (u + \ukh).
\end{equation}
With these considerations, we have the following lemma:
\begin{lemma}\label{lemm:Galerkin_orth}
  Let $(u,p)$ be a solution to the Navier-Stokes equations \eqref{eq:nav_stokes_weak_with_pressure},
  and $(\ukh,\pkh)$ their discrete approximation solving \eqref{eq:nav_stokes_discrete_with_pressure}.
  Let further $\uukh$ denote the average of $u$ and $\ukh$ as defined by \eqref{eq:def_linearizationpoint}.
  Then for any $(\vkh,\qkh)\in \Xkh$, it holds
\begin{equation*}
  B((u-\ukh,p-\pkh),(\vkh,\qkh)) + \chatIO{\uukh,u-\ukh,\vkh} + \chatIO{u-\ukh,\uukh,\vkh} = 0.
\end{equation*}
\end{lemma}
\begin{proof}
  This result is an immediate consequence of the definitions of solutions to
  \eqref{eq:nav_stokes_weak_with_pressure} and \eqref{eq:nav_stokes_discrete_with_pressure},
  together with the identity \eqref{eq:linearize_nav_stokes}
\end{proof}
This motivates the choice of $\uukh$ as linearization point for setting up a dual equation:\\
Find $(\zkh,\rhokh) \in \Xkh$ such that for all $(\phi_{kh},\psi_{kh}) \in \Xkh$ it holds
\begin{equation}\label{eq:discrete_dual}
  B((\phi_{kh},\psi_{kh}),(\zkh,\rhokh)) + \chatIO{\uukh,\phi_{kh},\zkh} + \chatIO{\phi_{kh},\uukh,\zkh}
  = \IOpair{ \dualrhs,\phi_{kh}},
\end{equation}
where the right hand side $\dualrhs$ will be chosen appropriately, 
see the proofs of \Cref{thm:nav_stokes_bestapprox,thm:nav_stokes_l2l2}.
Note that the right hand side of \eqref{eq:discrete_dual} implicitly prescribes
the final data $z_{kh,M}^+ = 0$.
This dual equation will help us in deriving the sought error estimates, which we will do in the following
section.
\begin{remark}
  To analyze this dual problem, it will be convenient,
  to have dual represenations of $\mathfrak B$ and $B$ at hand, which are obtained by
  partial integration on each $I_m$ and rearranging the terms. Note that in this representation,
  the time derivative is applied to the second argument. It holds
  \begin{equation}\label{eq:dual_frakB_representation}
    \hspace{14mm} \mathfrak B(u,v) = - \sum_{m=1}^M \blue{\ImOprod{ u, \partial_t v}}
    + \nu\IOprod{\nabla  u,\nabla v}
    - \sum_{m=1}^{M-1} \Oprod{u_m^-,[v]_m}
    + \Oprod{u_M^-,v_M^-},
  \end{equation}
  \begin{equation}\label{eq:dual_B_representation}
    \begin{aligned}
      B((u,p),(v,q)) = & - \sum_{m=1}^M \blue{\ImOprod{ u, \partial_t v}}
     + \nu\IOprod{\nabla  u,\nabla v}
     - \sum_{m=1}^{M-1} \Oprod{u_m^-,[v]_m}
     + \Oprod{u_M^-,v_M^-}\\
     & + \IOprod{\nabla \cdot u, q}
     - \IOprod{\nabla \cdot v, p},
    \end{aligned}
  \end{equation}
  see also \cite{leykekhman_discrete_2017,behringer_fully_2022}.
\end{remark}
\begin{remark}\label{rem:dual_pressure_existence}
  Similar to the discrete Navier-Stokes equations, we can consider an equivalent formulation 
  for the discrete dual equation
  in discretely divergence free spaces: Find $\zkh \in \Vkh$ satisfying
\begin{equation}\label{eq:discrete_dual_pressurefree}
  \mathfrak B(\phi_{kh},\zkh) + \chatIO{\uukh,\phi_{kh},\zkh} + \chatIO{\phi_{kh},\uukh,\zkh}
  = \IOpair{\dualrhs,\phi_{kh}}
  \qquad \text{for all } \phi_{kh} \in \Vkh.
\end{equation}
With the same argument as \cite[Proposition 4.3]{behringer_fully_2022}, there holds:
If $\zkh \in \Vkh$ solves \eqref{eq:discrete_dual_pressurefree}, then there exists $\rhokh \in \Mkh$
such that $(\zkh,\rhokh) \in \Xkh$ solves \eqref{eq:discrete_dual}. If  $(\zkh,\rhokh) \in \Xkh$ solves
\eqref{eq:discrete_dual}, then $\zkh \in \Vkh$ and it solves \eqref{eq:discrete_dual_pressurefree}.
\end{remark}
We first show unique solvability of the discrete dual problem, and the stability in $L^\infty(I;L^2(\Omega)) \cap L^2(I;H^1(\Omega))$. Since both $u$ and $\ukh$ occur in the formulation of the discrete problem,
we need both results of \Cref{lemm:vanishing_support_u} to hold true. 
\begin{theorem}\label{thm:discrete_dual_oseen}
  \blue{Let $f \in L^2(I;L^2(\Omega)^2)$, $u_0 \in V$ }
  and $u$,$u_{kh} \in L^\infty(I;L^2(\Omega)^2) \cap L^2(I;H^1(\Omega)^2)$
  be solutions to the weak and fully discretized Navier-Stokes
  equations \eqref{eq:nav_stokes_weak} and \eqref{eq:nav_stokes_discrete}
  for either $q=0$ or $q=1$.
 % respectively, such that 
 % $\|u-\ukh\|_{L^2(I;H^1(\Omega))} \to 0$ as $(k,h)\to 0$.
 % Let further $q=0$ or $q=1$.
  Then for $(k,h)$ small enough, problem \eqref{eq:discrete_dual_pressurefree} has a unique solution 
  $\zkh \in \Vkh$ for any $\dualrhs \in L^1(I;L^2(\Omega)^2)$, and there holds the bound
  \begin{equation*}
    \|\zkh\|_{L^\infty(I;L^2(\Omega))} + \|\zkh\|_{L^2(I;\blue{\Hone})}
    \le \OseenConstant \|\dualrhs\|_{L^1(I;L^2(\Omega))},
  \end{equation*}
  where $K:[0, + \infty) \to [0, + \infty)$ is a strictly monotonically increasing, continuous
  nonlinear function, independent of $k,h$.
\end{theorem}
% \begin{remark}
%   The main feature of this statement is the fact, that the norm bounds for $\zkh$ may
%   depend nonlinearly on $\uukh$ but are linear w.r.t the norm of \ $\dualrhs$.
%   Instead of $\uukh$ as linearization point we can also show this Lemma for general 
%   $\duallinearizationpoint \in L^2(I;V) \cap L^\infty(I;L^2(\Omega))$. 
%   For a discrete linearization point $\duallinearizationpoint_{kh}$
%   we will need the statement from 
%   \Cref{lemm:vanishing_support_u} to hold true, in order 
%   to achieve a bound independent of $k,h$.
%   We achieve this here by using the convergence $\ukh \to u$. As long as such a statement is true for
%   arbitrary $\duallinearizationpoint_{kh}$, the proof still holds true.
% \end{remark}
\begin{proof}
  On the continuous level and for $\dualrhs \in L^2(I;H^{-1}(\Omega)^2)$,
  the proof of a corresponding estimate
  can be found in \cite[Proposition 2.7]{casas_well_posedness_2021}.
  We adapt it to the discrete setting and to $\dualrhs \in L^1(I;L^2(\Omega)^2)$,
  making use of the previously derived
  discrete Gronwall \Cref{thm:discrete_gronwall_quadlinear}.
  We only have to prove the norm bound, since
  problem \eqref{eq:discrete_dual} is a quadratic system of linear equations, thus is solvable, if 
  it is injective. The norm bound yields, that for right hand side $\dualrhs=0$,
  $\zkh=0$ is the only solution,
  thus the norm bound implies existence and uniqueness.
  For ease of notation, we use the convention $[\zkh]_M = -z_{kh,M}^-$.
  Testing \cref{eq:discrete_dual_pressurefree} with $\zkh \chi_{I_{m}}$, $m=1,...,M$,
  where by $\chi_{I_{m}}$ we denote the indicator function of the subinterval $I_{m}$, yields:
  \begin{equation*}
  \begin{aligned}
    - \blue{\ImOprod{\zkh \chi_{I_{m}},\partial_t \zkh}}
    & + \nu\IOprod{\nabla \zkh \chi_{I_{m}}, \nabla \zkh}
    - \Oprod{z_{kh,m}^-,[\zkh]_{m}} \\
                                        & + \chatIO{\uukh,\zkh \chi_{I_{m}},\zkh}
    + \chatIO{\zkh \chi_{I_{m}},\uukh,\zkh}
                                         \! = \! \IOprod{\dualrhs,\zkh \chi_{I_{m}}}.
  \end{aligned}
 \end{equation*}
 Applying \Cref{lemm:jump_reorder} and writing the inner product as norm yields
 \begin{equation}\label{eq:discrete_dual_tested}
  \begin{aligned}
    \frac{1}{2} \|z_{kh,m-1}^+\|_{L^2(\Omega)}^2 + \frac{1}{2} \|[\zkh]_{m}\|_{L^2(\Omega)}^2
    + \nu\|\nabla \zkh\|_{L^2(I_m;L^2(\Omega))}^2&\\
    + \chatIO{\uukh,\zkh \chi_{I_{m}},\zkh} + \chatIO{\zkh \chi_{I_{m}},\uukh,\zkh}
    \!& = \frac{1}{2} \|z_{kh,m}^+\|_{L^2(\Omega)}^2 + \! \blue{\ImOprod{\dualrhs,\zkh}}.
  \end{aligned}
 \end{equation}
 We proceed by estimating the trilinear forms. According to its definition \eqref{eq:anti_symmetrized_c},
 the first term vanishes, and for the second one, it holds
  \begin{align*}
    \chatIO{\zkh\chi_{I_m},\uukh,\zkh} & = 
    \frac{1}{2} \cIO{\zkh\chi_{I_m},\uukh,\zkh} 
    - \frac{1}{2} \cIO{\zkh\chi_{I_m},\zkh,\uukh}.
  \end{align*}
  After applying \blue{Hölder's} inequality in space, we obtain
  \begin{equation*}
    \chatIO{\zkh\chi_{I_m},\uukh,\zkh} \le
    \frac{1}{2} \int_{I_m} \left(\|\zkh\|_{L^4(\Omega)}^2 \|\nabla \uukh\|_{L^2(\Omega)} 
    + \|\zkh\|_{L^4(\Omega)} \|\uukh\|_{L^4(\Omega)} \|\nabla \zkh\|_{L^2(\Omega)} \right)\ dt.
  \end{equation*}
  After estimating the $L^4$ norms by \Cref{lemm:c_swap}
  and applying \blue{Hölder's} inequality in time, we arrive at
  \begin{equation*}
  \begin{aligned}
    \chatIO{\zkh\chi_{I_m},\uukh,\zkh}
    \le & C %\frac{1}{2\sqrt{2}}
    \Big(\|\nabla \zkh\|_{\LtwomLtwo} \|\zkh\|_{L^\infty(I_m;L^2(\Omega))}
        \|\nabla \uukh\|_{\LtwomLtwo} \\
        & + \|\zkh\|_{L^\infty(I_m;L^2(\Omega))}^{\frac{1}{2}} 
        \|\nabla \zkh\|_{\LtwomLtwo}^{\frac{3}{2}}
        \|\uukh\|_{L^\infty(I_m;L^2(\Omega))}^{\frac{1}{2}}
      \|\nabla \uukh\|_{\LtwomLtwo}^{\frac{1}{2}}\Big).
  \end{aligned}
  \end{equation*}
  An application of Young's inequality yields
  \begin{equation*}\label{eq:estimate_trilinear_form}
    \chatIO{\zkh\chi_{I_m},\uukh,\zkh}
    \le  \frac{\nu}{2} \|\nabla \zkh\|_{\LtwomLtwo}^2 + 
    C\|\zkh\|_{L^\infty(I_m;L^2(\Omega))}^{2}\left(1+\|\uukh\|_{L^\infty(I_m;L^2(\Omega))}^{2}\right)
      \|\nabla \uukh\|_{\LtwomLtwo}^{2}.
  \end{equation*}
  In order to abbreviate the notation, we introduce 
  $%\begin{equation*}
    \gamma_m := C \|\nabla \uukh\|_{\LtwomLtwo}^2 \left(1+ \|\uukh\|_{L^\infty(I_m;L^2(\Omega))}^2\right).
  $%\end{equation*}
  We insert the above estimate into \eqref{eq:discrete_dual_tested}, absorb terms, and multiply by 2,
   which yields
  \begin{equation*}
  \begin{aligned}
    \|z_{kh,m-1}^+\|_{L^2(\Omega)}^2 + \|[\zkh]_{m}\|_{L^2(\Omega)}^2
    + \nu\|\nabla \zkh\|_{L^2(I_m;L^2(\Omega))}^2 
    & \le \gamma_m \|\zkh\|_{L^\infty(I_m;L^2(\Omega))}^2 + 
    \|z_{kh,m}^+\|_{L^2(\Omega)}^2 + 2 \blue{\ImOprod{\dualrhs,\zkh}}.
  \end{aligned}
 \end{equation*}
 After an application of Hölder's inequality, we obtain
 \begin{equation}\label{eq:discrete_dual_timeinterval}
  \begin{aligned}
    \|z_{kh,m-1}^+\|_{L^2(\Omega)}^2 + \|[\zkh]_{m}\|_{L^2(\Omega)}^2
    + \nu\|\nabla \zkh\|_{L^2(I_m;L^2(\Omega))}^2
    & \le \gamma_m \|\zkh\|_{L^\infty(I_m;L^2(\Omega))}^2 + 
    \|z_{kh,m}^+\|_{L^2(\Omega)}^2\\
    & \quad +2\|\dualrhs\|_{L^1(I_m;L^2(\Omega))}\|\zkh\|_{L^\infty(I_m;L^2(\Omega))}.
  \end{aligned}
 \end{equation}
 Since we have assumed $q=0$ or $q=1$, we can estimate the $\|\zkh\|_{L^\infty(I_m;L^2(\Omega))}$ terms
 by evaluations at the right and left endpoints:
    $\|\zkh\|_{L^\infty(I_m;L^2(\Omega))} 
    = \max \lbrace \|z_{kh,m-1}^+\|_{L^2(\Omega)},\|z_{kh,m}^-\|_{L^2(\Omega}\rbrace$.
  With triangle inequality there hold the following estimates:
  \begin{align*}
    \|\zkh\|_{L^\infty(I_m;L^2(\Omega))} 
    &\le \max \lbrace \|z_{kh,m-1}^+\|_{L^2(\Omega)},
    \|z_{kh,m}^+\|_{L^2(\Omega)}+\|[\zkh]_m\|_{L^2(\Omega)}\rbrace,\\
    \|\zkh\|_{L^\infty(I_m;L^2(\Omega))}^2
    &\le \max \lbrace \|z_{kh,m-1}^+\|_{L^2(\Omega)}^2,
    2\|z_{kh,m}^+\|_{L^2(\Omega)}^2+2\|[\zkh]_m\|_{L^2(\Omega)}^2\rbrace.
  \end{align*}
  We introduce 
  $x_n^2 := \|z_{kh,n-1}^+\|_{L^2(\Omega)}^2 
  + \sum_{m=n}^M \|[\zkh]_{m}\|_{L^2(\Omega)} ^2 +\nu\|\nabla \zkh\|_{L^2(I_m;L^2(\Omega))}^2$.
  Note that by the terminal condition for $z_{kh}$, it holds $x_{M+1}=0$.
  Then after summing up \eqref{eq:discrete_dual_timeinterval} from $m=n$ to $m=M$, we have:
  \begin{equation*}
    x_n^2 \le \sum_{m=n}^M 
    2 \gamma_m (x_m^2 + x_{m+1}^2)
    + 2\|\dualrhs\|_{L^1(I_{m};L^2(\Omega))} (x_m + x_{m+1}).
  \end{equation*}
  Shifting indices, we arrive at
  \begin{equation*}
    x_n^2 \le \sum_{m=n}^M 
    2 (\gamma_m + \gamma_{m-1}) x_m^2
    + 2\|\dualrhs\|_{L^1(I_{m} \cup I_{m-1};L^2(\Omega))} x_m,
  \end{equation*}
  where for $n=1$ we use the convention $I_0 = \emptyset$.
  Hence we are in the setting of \Cref{thm:discrete_gronwall_quadlinear},
  where formally we have to introduce an index transformation $\tilde n = M - n$.
  In order to apply the lemma, we need to verify $2(\gamma_m + \gamma_{m-1}) < 1$ for all $m$. 
  \blue{Due to $f\in L^2(I;L^2(\Omega)^2)$ and $u_0 \in V$, \Cref{thm:nav_stokes_l2h1} shows
  $\|u-\ukh\|_{\LtwoHone} \to 0$ as $(k,h) \to 0$ and we can apply \Cref{lemm:vanishing_support_u} and}
%  \blue{Due to the assumed convergence $\|u-\ukh\|_{\LtwoHone} \to 0$ as $(k,h) \to 0$, we can apply}
%  Lemma \ref{lemm:vanishing_support_u}, \blue{and} 
  obtain with triangle inequality, that 
  $\sup_{1\le m\le M} \|\nabla \uukh\|^2_{L^2(I_m;L^2(\Omega))} \to 0$ as $(k,h) \to 0$.
  Together with the bounds from \Cref{prop:navier_stokes_solvability} and 
  \Cref{thm:stability_discrete_navier_stokes} for $\|u\|_{L^\infty(I;L^2(\Omega))}$ and 
  $\|\ukh\|_{L^\infty(I;L^2(\Omega))}$, we obtain that 
  $\gamma_m \to 0$ uniformly in $n$ for $(k,h) \to 0$.
  Thus we can choose the discretization fine enough, such that 
  $\gamma_m < 1/8$, and thus we obtain from \Cref{thm:discrete_gronwall_quadlinear}
  \begin{equation}\label{eq:bound_discrete_dual_dG1}
  \begin{aligned}
    \|z_{kh,n-1}^+\|_{L^2(\Omega)}^2 + \sum_{m=n}^M \|[\zkh]_{m}\|_{L^2(\Omega)}^2 +
      \nu\|\nabla \zkh\|_{L^2(I_m;L^2(\Omega))}^2\\
    \le C_1\exp\left(C_2 \|\nabla \uukh\|^2_{\LtwoLtwo} 
    \left(1 + \|\uukh\|^2_{L^\infty(I,L^2(\Omega))}\right)\right) \|\dualrhs\|_{L^1(I;L^2(\Omega))}^2.
  \end{aligned}
  \end{equation}
\end{proof}
\blue{The next theorem, similar to \Cref{thm:nav_stokes_discrete_h2},
  states a stability result for the discrete dual solution 
in stronger norms, whenever the right hand side posesses more regularity.}
%
%We conclude this section by stating stability of the discrete dual solution in stronger norms,
%whenever the right hand side posesses more regularity.
\begin{theorem}\label{lemm:discrete_dual_h2}
  Let the assumptions of \Cref{thm:discrete_dual_oseen} hold true, i.e.,
  \blue{let $f \in L^2(I;L^2(\Omega)^2)$, $u_0 \in V$ 
  and $u$,$u_{kh}$
  be solutions to the weak and fully discretized Navier-Stokes
  equations \eqref{eq:nav_stokes_weak} and \eqref{eq:nav_stokes_discrete}
  for either $q=0$ or $q=1$.}
 % respectively, such that 
 % $\|u-\ukh\|_{L^2(I;H^1(\Omega))} \to 0$ as $(k,h)\to 0$.
 % Let further $q=0$ or $q=1$.
  Then for $(k,h)$ small enough, for any $\dualrhs \in L^2(I;L^2(\Omega)^2)$,
  the unique solution $\zkh \in \Vkh$ to \eqref{eq:discrete_dual_pressurefree} satisfies the bound
%
%  let $u$,$u_{kh} \in L^\infty(I;L^2(\Omega)^2) \cap L^2(I;H^1(\Omega)^2)$
%  be solutions to the weak and fully discretized Navier-Stokes
%  equations \eqref{eq:nav_stokes_weak} and \eqref{eq:nav_stokes_discrete} respectively, such that 
%  $\|u-\ukh\|_{L^2(I;H^1(\Omega))} \to 0$ as $(k,h)\to 0$.
%  Let additionally $f \in L^2(I;L^2(\Omega)^2)$ and $u_0 \in V$, such that the results of 
%  \Cref{thm:h2_regularity} and \Cref{thm:nav_stokes_discrete_h2} hold, and let 
%  $\dualrhs \in L^2(I;L^2(\Omega)^2)$. 
%  Then for $(k,h)$ small enough \blue{the discrete dual solution $\zkh$ to \eqref{eq:discrete_dual_pressurefree}
%  satisfies the bound}
   \begin{equation*}
      \|\zkh\|_{L^\infty(I;H^1(\Omega))} + \|A_h \zkh\|_{L^2(I;L^2(\Omega))} \le 
      \blue{C_{u,\ukh}} \|\dualrhs\|_{\LtwoLtwo},
   \end{equation*}
%   \blue{with a constant $C_{u,\ukh}$ depending continuously on 
%   $\|A u\|_{L^2(I;L^2(\Omega))},\|u\|_{L^\infty(I;L^2(\Omega))},
%   \|A_h \ukh\|_{L^2(I;L^2(\Omega))}$ and $\|\ukh\|_{L^\infty(I;L^2(\Omega))}$, which are bounded by 
% \Cref{prop:navier_stokes_solvability} and 
% \Cref{thm:navier_stokes_higher_regularity,thm:nav_stokes_discrete_h2,thm:stability_discrete_navier_stokes},
% due to the assumptions on $f$ and $u_0$.}\\
 \blue{ with a constant $C_{u,\ukh}$ depending on $u,\ukh$ in the form
   \begin{equation*}
     C_{u,\ukh} = C_1 \exp\left(C_2 \left(\|A u\|_{L^2(I;L^2(\Omega))}^2\|u\|_{L^\infty(I;L^2(\Omega))}^2
    +\|A_h \ukh\|_{L^2(I;L^2(\Omega))}^2\|\ukh\|_{L^\infty(I;L^2(\Omega))}^2\right)\right),
   \end{equation*}
 where the norms of $u,\ukh$ are bounded by 
 \Cref{prop:navier_stokes_solvability} and 
 \Cref{thm:h2_regularity,thm:nav_stokes_discrete_h2,thm:stability_discrete_navier_stokes},
 due to the assumptions on $f$ and $u_0$.}
\end{theorem}
\begin{proof}
  The proof follows the same steps as the proof the stability of $\ukh$ in stronger norms, presented in
  \Cref{thm:nav_stokes_discrete_h2}.
  We begin by testing the dual equation with $A_h \zkh \chi_{I_{m}}$ for arbitrary $m=1,...,M$,
  where $\chi_{I_{m}}$ denotes the characteristic function of the time interval $I_{m}$.
  The nonlinear terms that occur for the discrete dual equation are of the form
  \begin{equation*}
    \chatIO{\uukh,A_h \zkh \chi_{I_{m}},\zkh} + \chatIO{A_h \zkh \chi_{I_{m}},\uukh,\zkh}.
  \end{equation*}
  By \blue{Hölder's} inequality in space, \eqref{eq:discrete_h1_interpolation} and
  the discrete Gagliardo-Nirenberg inequality \eqref{eq:discrete_stokes_gagliardo_nirenberg}, these
  terms can be estimated by
  \begin{align*}
     C \int_{I_{m}} \|\zkh\|_{L^2(\Omega)}^{\frac{1}{2}}\|A_h \zkh\|_{L^2(\Omega)}^{\frac{3}{2}}
     \left(\|\nabla \uukh\|_{L^2(\Omega)} + \|\uukh\|_{L^\infty(\Omega)} \right)dt.
  \end{align*}
  Applying the continuous and discrete Gagliardo-Nirenberg inequalities
  \eqref{eq:gagliardo_nirenberg} and \eqref{eq:discrete_stokes_gagliardo_nirenberg}, 
  and \eqref{eq:discrete_h1_interpolation} to $\uukh$, it remains
  \begin{align*}
     C \int_{I_{m}} \|\zkh\|_{L^2(\Omega)}^{\frac{1}{2}}\|A_h \zkh\|_{L^2(\Omega)}^{\frac{3}{2}}
     \left(\|u\|_{L^2(\Omega)}^{\frac{1}{2}}\|A u\|_{L^2(\Omega)}^{\frac{1}{2}}
     + \|\ukh\|_{L^2(\Omega)}^{\frac{1}{2}}\|A_h\ukh\|_{L^2(\Omega)}^{\frac{1}{2}}\right) dt.
  \end{align*}
  An application of Young's inequality, absorbing terms and summing over the subintervals allows us to 
  conclude the proof.
  %\blue{Using the assumption $\|u-\ukh\|_{\LtwoHone} \to 0$ as $(k,h)\to 0$,}
  By \Cref{thm:discrete_dual_oseen} it holds $\zkh \in L^\infty(I;L^2(\Omega)^2)$ with a bound 
  independent on $k,h$ and linear in $\|\dualrhs\|_{L^1(I;L^2(\Omega))}$. 
  Further, the terms involving $u, \ukh$ are summable, since the $L^\infty(I;L^2(\Omega))$ 
  norms of $u, \ukh$ remain bounded via \Cref{prop:navier_stokes_solvability} and
  \Cref{thm:stability_discrete_navier_stokes}, and the $Au, A_h \ukh$ terms are summable by
  \Cref{thm:h2_regularity} and \Cref{thm:nav_stokes_discrete_h2}, 
  \blue{where the latter holds true for $(k,h)$ small enough, 
  due to $f \in L^2(I;L^2(\Omega)^2)$ and $u_0 \in V$.}
\end{proof}

%
%\textcolor{red}{
% In order to prove the proposed error estimate by a duality argument, we first need to show
% unique solvability of the discrete dual equation \eqref{eq:discrete_dual}.
%To this end, we need the following proposition and two technical lemmas.
% As mentioned in the introduction, there are numerous error estimates for the 
% instationary Navier-Stokes equations for various discretization schemes. Let us again refer to 
% \cite{casas_discontinuous_2012, heywood_finite_1986, Heywood1990}, with
% \cite[Theorem 4.7]{casas_discontinuous_2012} being the most relevant as it uses the same discretization
% scheme.
% With these results from the literature, we have the following:
% \begin{proposition}\label{prop:bad_convergence}
%   Let $f \in L^2(I;L^2(\Omega)^2)$ and $u_0 \in V$.
%   Then for $\left(k,h) \to 0$ it holds
%   \begin{equation*}
%     \|u- \ukh\|_{\Linftwo} + \|u- \ukh\|_{\LtwoV} \to 0.
%   \end{equation*}
% \end{proposition}
% }
%
%\todo{Übergang}
%\blue{Please note, that the proof of \Cref{thm:nav_stokes_l2h1} only requires 
%  stability in $L^\infty(I;H^1(\Omega))$ for the discrete Stokes solution $\ukht$.
%  It is especially independent of 
%  \Cref{thm:discrete_dual_oseen,lemm:discrete_dual_h2,thm:nav_stokes_discrete_h2}.
%  Conversely, \Cref{thm:nav_stokes_l2h1} allows us now to apply 
%  \Cref{thm:discrete_dual_oseen,lemm:discrete_dual_h2,thm:nav_stokes_discrete_h2}, as it shows, that the 
%  assumption $\|u-\ukh\|_{\LtwoHone}\to 0$ made in the latter theorems holds true.}
We now turn towards showing the main result of our work, i.e., the error estimate for the Navier-Stokes 
equations in the $L^\infty(I;L^2(\Omega))$ norm.
% 
% As mentioned in the introduction, there are numerous related error estimates for the 
% instationary Navier-Stokes equations for various discretization schemes. Let us again refer to 
% \cite{casas_discontinuous_2012, heywood_finite_1986, Heywood1990}, with
% \cite[Theorem 4.7]{casas_discontinuous_2012} being the most relevant as it uses the same discretization
% scheme.
% To this end, we make use of the corresponding error estimate for the Stokes problem stated in 
% \Cref{prop:stokes_linfty}.
\blue{As in the proof of \Cref{thm:nav_stokes_l2h1}, we} 
will split the error $u - \ukh$ into an error for a Stokes problem,
and a remainder term, which we will estimate using the discrete dual equation 
\eqref{eq:discrete_dual_pressurefree},
\blue{to which we can apply the results of \Cref{thm:discrete_dual_oseen}.}
% Note that since we have shown the $L^2(I;H^1(\Omega))$ error estimate in \Cref{thm:nav_stokes_l2h1},
% we can apply the stability results of \Cref{thm:discrete_dual_oseen} for the discrete dual problem.
%
% The main drawback of the results in the literature is the fact, that the errors in 
% $L^\infty(I;L^2(\Omega))$ and $L^2(I;H^1(\Omega))$ are estimated 
% in a combined fashion. Thus the estimate always has to account for the spacial error in the $H^1$ norm,
% yielding an order reduction for the estimate of the error in the spacial $L^2$ norm.
% We are interested in estimating the error $u - \ukh$ only in the $L^\infty(I; L^2(\Omega)^2)$ norm,
% with an improved rate compared to previous results from the literature.
% 
%\textcolor{red}{%
%   These two lemmas allow us, to show the unique solvability of the previously stated dual equation 
% \eqref{eq:discrete_dual_pressurefree}:
% In the proof of the stability of the discrete dual state $\zkh$ \Cref{thm:discrete_dual_oseen}, we have used
% the error estimate of \Cref{thm:nav_stokes_l2h1}, as it allowed the application of 
% the discrete Gronwall Lemma.
% In the same fashion, we can also show, that for smoother data, the discrete solution to the 
% Navier-Stokes equations $\ukh$ satisfies a stability result in stronger norms.
%
%}
% The stability result for the discrete dual equation from the previous theorem allows us 
% to show the main result of this work, estimating the $\LinfLtwo$ error of the Navier-Stokes equations.
Similar to the result for the Stokes equations of \Cref{prop:stokes_linfty},
the error estimate will consist of two terms with the first one being a best approximation
error, and the second one being the error of the stationary Stokes Ritz projection introduced in 
\eqref{eq:stokes_ritzprojection}.
This result estimates the $L^\infty(I;L^2(\Omega))$ norm in an isolated fashion, and thus does not suffer 
from an order reduction, which is observed in results that estimate the error norm combined with the
$L^2(I;H^1(\Omega))$ norm.
%being an analog of a Ritz projection error for the stationary
%Stokes system, see \cite[Section 6]{behringer_fully_2022}.
%For any $(w,r) \in H^1_0(\Omega)^2 \times L^2(\Omega)$, the Ritz projection for the stationary Stokes
%problem
%$\left(R^S_h(w,r),R^{S,p}_h(w,r)\right) \in U_h \times M_h$ is defined by
%\begin{equation*}
%  \begin{aligned}
%    \Oprod{\nabla \left(w - R^S_h(w,r)\right),\nabla \phi_h} 
%    - \Oprod{q - R^{S,p}_h(w,r),\nabla \cdot \phi_h} &= 0
%    \quad \text{for all } \phi_h \in U_h\\
%    \Oprod{\nabla \cdot \left(w - R^S_h(w,r)\right),\psi_h} &= 0
%    \quad \text{for all } \psi_h \in M_h.
%  \end{aligned}
%\end{equation*}
%Here $R^S_h$ denotes the velocity component of the projection, and $R^{S,p}_h$ denotes the 
%pressure component. 
%Note that if $w$ is discretely divergence free, i.e., $\Oprod{\nabla \cdot w,\psi_h}=0$ 
%for all $\psi_h \in M_h$, then it
%holds $R_h^S(w,r) \in V_h$.
%With this notation fixed, we can now state the main theorem of this work:
\begin{theorem}\label{thm:nav_stokes_bestapprox}
  Let $f \in L^2(I;L^2(\Omega)^2)$ and $u_0 \in V$.
  Let $(u,p)$ be the unique solution to the Navier-Stokes equations \eqref{eq:nav_stokes_weak_with_pressure},
  and $(\ukh,\pkh)$ the corresponding solution to the discretized equations 
  \eqref{eq:nav_stokes_discrete_with_pressure} for a discontinuous Galerkin method in time with
  order $q=0$ or $q=1$, \blue{for sufficiently small discretization parameters $(k,h)$.}
  Then there holds
  \begin{equation*}
    \|u - \ukh\|_{\LinfLtwo} \le C \left(\ln \frac{T}{k}\right)
    %\lnh^{\frac{1}{2}}
    \left(
      \inf_{\blue{\chi_{kh}} \in \Vkh} \|u - \blue{\chi_{kh}}\|_{\LinfLtwo} 
    + \|u - R_h^S(u,p)\|_{\LinfLtwo}\right).
  \end{equation*}
\end{theorem}
\begin{proof}
  \blue{We use the same notation as in the proof of \Cref{thm:nav_stokes_l2h1} and denote 
    the velocity error by $e:= u - \ukh$ and the pressure error by $r:=p - \pkh$. We consider the 
    Stokes projection $(\ukht,\pkht)$ of $(u,p)$, solving \eqref{eq:discrete_aux_stokes}, and the 
    induced splitting of the error into $e = \xi + \eta_{kh}$, $r = \omega + \kappa_{kh}$,
    where $\xi = u - \ukht$, $\eta_{kh} = \ukht - \ukh$, $\omega = p - \pkht$ and $\kappa_{kh} = \pkht - \pkh$.
}
%   We denote by $e:= u - \ukh$ the error, which we want to estimate,
%   and introduce $(\ukht,\pkht)$ as the instationary Stokes-projection of $(u,p)$, i.e., the solution to
%   \begin{equation}\label{eq:discrete_aux_stokes}
%     B ((\ukht,\pkht),(\phikh,\psikh)) = \IOprod{f,\vkh} 
%     - \cIO{u,u,\vkh} + \Oprod{u_0,v_{kh,0}^+} 
%     \quad \text{for all } (\phikh, \psikh) \in \Xkh.
%   \end{equation}
%   Note that for the definition of this equation, we can equivalently choose $c$ or $\hat c$, as we have
%   $u$ with $\nabla \cdot u = 0$ in its first argument.
%   There exists a solution $(\ukht,\pkht)$ to this discrete problem,
%   since the assumptions on $f$ and $u_0$ yield $(u\cdot \nabla) u \in L^2(I;L^2(\Omega)^2)$,
%   see \Cref{thm:nonlinearity_higher_regularity}.
%   We split the velocity error $e = \xi + \eta_{kh}$ into two components
%   $\xi = u - \ukht$ and $\eta_{kh} = \ukht - \ukh$.
%   The error in the pressure will be denoted by $r := p - \pkh$, and we will split it into the
%   same contributions, denoted by $\omega := p - \tilde p_{kh}$ and $\kappa_{kh} := \tilde p_{kh} - \pkh$.
  We immediately obtain the estimate
  \begin{equation*}
    \|\xi\|_{\LinfLtwo}
    \le C \left(\ln \frac{T}{k}\right)
    \left( \inf_{\blue{\chi_{kh}} \in \Vkh} \|u - \blue{\chi_{kh}}\|_{\LinfLtwo} 
    + \|u - R_h^S(u,p)\|_{\LinfLtwo}\right),
  \end{equation*}
  by \Cref{prop:stokes_linfty}, since by
  \Cref{thm:nonlinearity_higher_regularity} it holds $f - (u\cdot \nabla) u \in L^2(I;L^2(\Omega)^2)$.
  Thus to finalize the proof, we have to estimate the $L^2(\Omega)$ norm of 
  $\eta_{kh}$ pointwise in time.
  To this end, as in \cite[Theorem 6.2]{behringer_fully_2022},
  we fix $\tilde t \in I$ and construct $\theta \in C^\infty_0(I)$ in such a way, that 
  $\supp \theta \subset I_m$ where $m$ is chosen such that $\tilde t \in I_m$, and
  \begin{equation*}
    \IOprod{\eta_{kh}(\tilde t)\theta,\phi_{kh}} = \Oprod{\eta_{kh}(\tilde t), \phi_{kh}(\tilde t)}
  \ \text{for all } \phi_{kh} \in \Ukh,
  \quad \|\theta \|_{L^1(I)} \le C \text{ independent of } \tilde t \text{ and } k.
  \end{equation*}
  For the construction of such a function $\theta$ serving the purpose of a regularized Dirac measure,
  we refer to \cite[Appendix A.5]{schatz_interior_1995}.
  We then define the dual solution $\zkh \in \Vkh$ such that
  for all $\phi_{kh} \in \Vkh$, it satisfies
  \begin{equation*}
    \mathfrak B(\phi_{kh},\zkh) 
    + \chatIO{\uukh,\phi_{kh},\zkh} 
    + \chatIO{\phi_{kh},\uukh,\zkh}
    = \IOprod{\eta_{kh}(\tilde t) \theta,\phi_{kh}}.
  \end{equation*}
  \blue{By \Cref{thm:discrete_dual_oseen}}
  we have the existence, uniqueness and regularity of
  $\zkh$, \blue{ satisfying the bound
    \begin{equation}\label{eq:discrete_dual_bound}
    \|\zkh\|_{L^\infty(I;L^2(\Omega))} + \|\nabla \zkh\|_{\LtwoLtwo}
    \le \OseenConstant \|\eta_{kh}(\tilde t) \theta \|_{L^1(I;L^2(\Omega))},
\end{equation}}
for $(k,h)$ small enough.
  From Remark \ref{rem:dual_pressure_existence}, we obtain the existence of an 
  associated pressure $\rhokh \in \Mkh$, such that for all $(\phi_{kh},\psi_{kh}) \in \Xkh$, it holds
  \begin{equation*}
    B((\phi_{kh},\psi_{kh}),(\zkh,\rhokh)) 
    + \chatIO{\uukh,\phi_{kh},\zkh} 
    + \chatIO{\phi_{kh},\uukh,\zkh}
    = \IOprod{\eta_{kh}(\tilde t) \theta,\phi_{kh}}.
  \end{equation*}
  Choosing the specific test functions $(\phi_{kh},\psi_{kh}) = (\eta_{kh},\kappa_{kh}) \in \Xkh$, we have
  \begin{align*}
    \|\eta_{kh}(\tilde t)\|^2_{L^2(\Omega)} 
    & = \IOprod{\eta_{kh}(\tilde t)\theta,\eta_{kh}}\\
    & = B((\eta_{kh},\kappa_{kh}),(\zkh,\rhokh)) 
        + \chatIO{\uukh,\eta_{kh},\zkh} 
        + \chatIO{\eta_{kh},\uukh,\zkh}\\
    & = B((e-\xi,r - \omega),(\zkh,\rhokh)) 
        + \chatIO{\uukh,e-\xi,\zkh} 
        + \chatIO{e-\xi,\uukh,\zkh}\\
    & = B((e,r),(\zkh,\rhokh)) 
        + \chatIO{\uukh,e,\zkh} 
        + \chatIO{e,\uukh,\zkh}\\
    & \phantom{==} - B((\xi,\omega),(\zkh,\rhokh)) 
    - \chatIO{\uukh,\xi,\zkh} 
    - \chatIO{\xi,\uukh,\zkh}.
  \end{align*}
  For $\xi$ we can use the Galerkin orthogonality with respect to $B$, i.e., from \eqref{eq:discrete_aux_stokes},
  we obtain
  \begin{equation*}
    B((\xi,\omega),(\phikh,\psikh)) = 0 \qquad \text{for all } (\phikh, \psikh) \in \Xkh.
  \end{equation*}
%  Moreover, by definition of $\uukh$, it holds
%  \begin{equation*}
%    \chatIO{\uukh,e,\zkh} + \chatIO{e,\uukh,\zkh} = \chatIO{u,u,\vkh} - \chatIO{\ukh,\ukh,\vkh}.
%  \end{equation*}
  \blue{Furthermore,} due to \Cref{lemm:Galerkin_orth} it holds
  \begin{equation*}
    B((e,r),(\zkh,\rhokh)) + \chatIO{\uukh,e,\zkh} + \chatIO{e,\uukh,\zkh} = 0.
  \end{equation*}
  Thus we see directly 
  \begin{equation*}
    \|\eta_{kh}(\tilde t)\|^2_{L^2(\Omega)} 
     = - \chatIO{\uukh,\xi,\zkh} - \chatIO{\xi,\uukh,\zkh}.
  \end{equation*}
  We want to make use of the $L^\infty(I;L^2(\Omega))$ estimate of $\xi$ and thus have to move it
  to an argument of $\hat c$ that has no spatial gradient applied to it.
  Since $\hat c$ was obtained by anti-symmetrizing $c$, there are gradients in the second and third
  argument of $\hat c$, which is why we revert to the original trilinear form $c$.
  The first argument has no gradient, and thus 
  \begin{equation}\label{eq:error_split_remainder}
    \|\eta_{kh}(\tilde t)\|^2_{L^2(\Omega)} 
    = - \frac{1}{2} \cIO{\uukh,\xi,\zkh} + \frac{1}{2} \cIO{\uukh,\zkh,\xi} - \chatIO{\xi,\uukh,\zkh}.
  \end{equation}
  Lemma \ref{lemm:c_swap_no_div} allows us to switch $\xi$ to the third argument, and we thus have
  \begin{equation*}
    \|\eta_{kh}(\tilde t)\|^2_{L^2(\Omega)} 
    = \frac{1}{2} \IOprod{\nabla \cdot \uukh, \zkh \cdot \xi} 
    + \cIO{\uukh,\zkh,\xi} - \chatIO{\xi,\uukh,\zkh}.
  \end{equation*}
  After application of Hölder's inequality, we obtain
  \begin{equation}\label{eq:pointwise_split_error}
  \begin{aligned}
    \|\eta_{kh}(\tilde t)\|^2_{L^2(\Omega)} 
    \le & C\|\xi\|_{L^\infty(I;L^2(\Omega))}  \|1\|_{L^4(I;L^\infty(\Omega))}\\
        & \times \left(\|\zkh\|_{L^4(I;L^4(\Omega))}\|\nabla \uukh\|_{L^2(I;L^4(\Omega))}
                   +  \|\nabla \zkh\|_{\LtwoLtwo}\|\uukh\|_{L^4(I;L^\infty(\Omega))} \right).
  \end{aligned}
  \end{equation}
  By \Cref{thm:h2_regularity} and \Cref{thm:nav_stokes_discrete_h2}, 
  \blue{the solutions $u$,$\ukh$ satisfy the bound} 
    \begin{equation*}
      \|\uukh\|_{L^\infty(I;H^1(\Omega))} +
      \|u\|_{L^2(I;H^2(\Omega))} +
      \|A_h \ukh\|_{L^2(I;L^2(\Omega))} \le C,
    \end{equation*}
    with a constant depending on the data.
    \blue{Hölder's} inequality and the continuous and discrete Gagliardo-Nirenberg inequalities 
    \eqref{eq:gagliardo_nirenberg} and \eqref{eq:discrete_stokes_gagliardo_nirenberg} yield moreover 
    \begin{equation*}
      \|\uukh\|_{L^4(I;L^\infty(\Omega))} 
      \le C \left(\|u\|_{L^\infty(I;L^2(\Omega))}^{\frac{1}{2}}\|A u\|_{L^2(I;L^2(\Omega))}^{\frac{1}{2}}
      +\|\ukh\|_{L^\infty(I;L^2(\Omega))}^{\frac{1}{2}} \|A_h \ukh\|_{L^2(I;L^2(\Omega))}^{\frac{1}{2}}\right).
    \end{equation*}
    From \eqref{eq:w14_interpolation} and \eqref{eq:discrete_w14_interpolation}, we further obtain
    \begin{equation*}
      \|\nabla \uukh\|_{L^2(I;L^4(\Omega))} 
      \le C \left(\|\nabla u\|_{L^\infty(I;L^2(\Omega))}^{\frac{1}{2}}\|A u\|_{L^2(I;L^2(\Omega))}^{\frac{1}{2}}
      + \|\nabla \ukh\|_{L^\infty(I;L^2(\Omega))}^{\frac{1}{2}} \|A_h \ukh\|_{L^2(I;L^2(\Omega))}^{\frac{1}{2}}
    \right).
    \end{equation*}
    With the above estimates, we thus obtain from \eqref{eq:pointwise_split_error}
    \begin{equation*}
    \|\eta_{kh}(\tilde t)\|^2_{L^2(\Omega)} 
    \le C   T^{\frac{1}{4}}\|\xi\|_{L^\infty(I;L^2(\Omega))}
    \left(\|\zkh\|_{L^4(I;L^4(\Omega))} +  \|\nabla \zkh\|_{\LtwoLtwo}\right).
    \end{equation*}
%   By \Cref{thm:h2_regularity}, it holds $\|u\|_{L^2(I;L^\infty(\Omega))}<\infty$,
%   with a bound depending on the data.
%   This bound however is not preserved on the discrete level.
%   Thus, for the $L^\infty(\Omega)$ norms of $\ukh$ and $\zkh$ we make use of the discrete Sobolev inequality,
%   see \cite[Lemma 4.9.2]{Brenner2008}:
%   \begin{equation}\label{eq:discrete_sobolev}
%     \|v_{h}\|_{L^\infty(\Omega)} \le C (1 + \lnh)^{1/2}\|v_{h}\|_{H^1(\Omega)}
%     \qquad \text{for all } v_h \in V_h.
%   \end{equation}
%   After splitting the terms in $\uukh$ in the $L^\infty(\Omega)$ norm
%   by triangle inequality and treating them individually,  we hence have the estimate
%   \begin{align*}
%     \|\eta_{kh}(\tilde t)\|^2_{L^2(\Omega)} 
%     & \le C(1 + \lnh)^{1/2}\|\xi\|_{L^\infty(I;L^2(\Omega))}
%     \|\nabla \zkh\|_{\LtwoLtwo}\\
%     & \hspace{5mm} \times \left(\|\nabla \uukh\|_{L^2(I;L^{2}\Omega))} 
%     + \|\nabla u_{kh} \|_{L^2(I;L^{2}\Omega))}
%     + \|u\|_{L^2(I;L^\infty(\Omega))}\right).
%   \end{align*}
  With \Cref{lemm:c_swap}, we further obtain the estimate
%   With the interpolation inequality 
%   $\|v\|_{L^4(\Omega)}\|\le C\|v\|_{L^2(\Omega)}^{\frac{1}{2}}\|v\|_{H^1(\Omega)}^{\frac{1}{2}}$,
%   see \cite[Theorem 3]{adams_cone_1977}, we further obtain
  \begin{equation*}
    \|\zkh\|_{L^4(I;L^4(\Omega))} 
    \le \|\zkh\|_{L^\infty(I;L^2(\Omega))}^{\frac{1}{2}}\blue{\|\nabla\zkh\|_{\LtwoLtwo}^{\frac{1}{2}}}.
  \end{equation*}
  %\blue{where the norms of $\zkh$ can be estimated using \eqref{eq:discrete_dual_bound}}
  %\Cref{thm:discrete_dual_oseen}, yielding
  %\begin{equation*}
  %  \|\zkh\|_{L^\infty(I;L^2(\Omega))} + \|\nabla \zkh\|_{\LtwoLtwo}
  %  \le \OseenConstant \|\eta_{kh}(\tilde t) \theta \|_{L^1(I;L^2(\Omega))}.
  %\end{equation*}
  \blue{Hence, applying \eqref{eq:discrete_dual_bound} yields
    \begin{equation}\label{eq:pointwise_split_error_2}
    \|\eta_{kh}(\tilde t)\|_{L^2(\Omega)}^2
    \le C T^{\frac{1}{4}} \|\xi\|_{L^\infty(I;L^2(\Omega))}
    \|\eta_{kh}(\tilde t) \theta \|_{L^1(I;L^2(\Omega))}.
    \end{equation}
  }
  By definition of $\theta$ we can estimate $\|\eta_{kh}(\tilde t) \theta \|_{L^1(I;L^2(\Omega))}$
  by $C \|\eta_{kh}(\tilde t)\|_{L^2(\Omega)}$.
  This allows us to divide \eqref{eq:pointwise_split_error_2}
  by $\|\eta_{kh}(\tilde t)\|_{L^2(\Omega)}$ which shows the bound
  \begin{align*}
    \|\eta_{kh}(\tilde t)\|_{L^2(\Omega)} 
    \le C T^{\frac{1}{4}}\|\xi\|_{L^\infty(I;L^2(\Omega))}.
  \end{align*}
  As $\tilde t \in I$ was arbitrary,  this shows
  \begin{align*}
    \|\eta_{kh}\|_{L^\infty(I;L^2(\Omega))} 
    & \le C T^{\frac{1}{4}} \|\xi\|_{L^\infty(I;L^2(\Omega))}.
  \end{align*}
  Making use of the previously derived bound for $\|\xi\|_{L^\infty(I;L^2(\Omega))}$ and triangle
  inequality for $e = \xi + \eta_{kh}$ concludes the proof.
\end{proof}
The above theorem is the main result of this work, and the development of the discrete Gronwall 
lemma and analysis of the dual problem were the key ingredients, in order to prove it.
With these techniques established, it is now straightforward, to also prove an 
error estimate in the $L^2(I;L^2(\Omega))$ norm.
If we follow the steps of the proof of \Cref{thm:nav_stokes_bestapprox} up to \eqref{eq:error_split_remainder},
we then have to estimate the occuring trilinear form in terms of $\|\xi\|_{L^2(I;L^2(\Omega))}$, 
i.e., the $L^2(I;L^2(\Omega))$ norm of the Stokes error.
\blue{This implies that we need to estimate the occuring trilinear terms by stronger norms of
the dual state $\zkh$, which are bounded by the results presented in \Cref{lemm:discrete_dual_h2}}
%in order to estimate the occuring trilinear terms.
%To this end we need the discrete stability result of \Cref{lemm:discrete_dual_h2},
%which holds due to the $L^2(I;H^1(\Omega))$ error estimate presented in \Cref{thm:nav_stokes_l2h1}.
With these considerations, we can show the following theorem.
\begin{theorem}\label{thm:nav_stokes_l2l2}
  Let $f \in L^2(I;L^2(\Omega)^2)$ and $u_0 \in V$. Further let $(u,p)$ and $(\ukh,\pkh)$ be the continuous
  and fully discrete solutions to the Navier-Stokes equations \eqref{eq:nav_stokes_weak_with_pressure}
  and \eqref{eq:nav_stokes_discrete_with_pressure},
  \blue{for sufficiently small discretization parameters $(k,h)$.}
  Then for any $\chi_{kh} \in \Vkh$, there holds
   \begin{align*}
     \|u-\ukh\|_{L^2(I\times \Omega)} & \le C \left(\| u - \chi_{kh}\|_{L^2(I\times\Omega)} 
     + \|u - \pi_\tau u\|_{L^2(I\times \Omega)} + \|u - R_h^S(u,p)\|_{L^2(I\times \Omega)}\right).
   \end{align*}
\end{theorem}
\begin{proof}
  To deduce the error estimate for the Navier-Stokes equations from the corresponding Stokes result of 
  \Cref{thm:stokesl2l2_l2h1},  we follow a duality argument similar to
  the proof of \Cref{thm:nav_stokes_bestapprox}.
  We split the error in the same fashion $u - \ukh = u - \ukht + \ukht - \ukh$ where $\xi:=u - \ukht$ is 
  estimated by the Stokes result, and for $\eta_{kh} = \ukht - \ukh$ we consider $\zkh$ a solution to the dual 
  problem
\begin{equation*}
   B((\phi_{kh},\psi_{kh}),(z_{kh},\rho_{kh})) 
   + \chatIO{\uukh,\phi_{kh},\zkh} + \chatIO{\phi_{kh},\uukh,\zkh} = 
   \IOprod{\eta_{kh},\phi_{kh}}.
\end{equation*}
As before, we test the dual equation with $\eta_{kh}$ and obtain after elimination of terms by 
applying Galerkin orthogonalities and \Cref{lemm:c_swap_no_div}:
\begin{equation*}
   \|\eta_{kh}\|_{\LtwoLtwo}^2
    = \frac{1}{2} \IOprod{\nabla \cdot \uukh, \zkh \cdot \xi} 
    + \cIO{\uukh,\zkh,\xi} - \chatIO{\xi,\uukh,\zkh}.
\end{equation*}
Applying \blue{Hölder's} inequality yields
\begin{align*}
   \|\eta_{kh}\|_{\LtwoLtwo}^2
   \le & C \|\xi\|_{\LtwoLtwo} \|1\|_{L^4(I;L^\infty(\Omega))} \\
       & \times 
   \left(\|\uukh\|_{L^4(I;L^\infty(\Omega))}\|\nabla \zkh\|_{\LinfLtwo} 
      + \|\zkh\|_{L^4(I;L^\infty(\Omega))}\|\nabla \uukh\|_{\LinfLtwo}\right).
\end{align*}
By \Cref{prop:navier_stokes_solvability}, \Cref{thm:h2_regularity},
\Cref{thm:stability_discrete_navier_stokes}
and \Cref{thm:nav_stokes_discrete_h2} all terms containing $u$, $\ukh$ are bounded.
Together with 
the discrete Gagliardo-Nirenberg inequality \eqref{eq:discrete_stokes_gagliardo_nirenberg} and 
Young's inequality, we obtain
\begin{align*}
   \|\eta_{kh}\|_{\LtwoLtwo}^2
   \le  C(u,\ukh) \|\xi\|_{\LtwoLtwo} T^{\frac{1}{4}} 
 \left( \|A_h \zkh\|_{L^2(I;L^2(\Omega))} + \|\nabla \zkh\|_{L^\infty(I;L^2(\Omega))}   +\|\zkh\|_{L^2(I;L^2(\Omega))}\right).
\end{align*}
Using \Cref{lemm:discrete_dual_h2}, we can bound the norms of $\zkh$ and obtain
\begin{align*}
   \|\eta_{kh}\|_{\LtwoLtwo}^2
   \le  C(u,\ukh) T^{\frac{1}{4}}\|\xi\|_{\LtwoLtwo} \|\eta_{kh}\|_{\LtwoLtwo}.
\end{align*}
Canceling terms concludes the proof.
\end{proof}
\subsection{Explicit orders of convergence}
Using the same arguments from before, instead of the best approximation type estimate, we can also
directly use the error estimate for the Stokes projection, shown in 
\cite[Theorem 7.4]{behringer_fully_2022} and \cite[Corollaries 6.2\& 6.4]{vexler_l2i_2023}
to obtain the following corollaries, yielding explicit orders of convergence.
\begin{corollary}\label{corr:err_est_nav_stokes_spec_orders}
  Let \Cref{ass:interpolation_operators} be fulfilled.
  \blue{Further let $f \in L^\infty(I;L^2(\Omega)^2)$, $u_0 \in V \cap H^2(\Omega)^2$ 
    and let $u$, $\ukh$ be the continuous and fully discrete solutions to the Navier-Stokes 
    equations \eqref{eq:nav_stokes_weak} and \eqref{eq:nav_stokes_discrete} respectively.
  }
%  Further let $u$ and $\ukh$ be the solutions to the weak and fully discretized Navier-Stokes equations,
%  with $f \in L^\infty(I;L^2(\Omega))$ and $u_0 \in V \cap H^2(\Omega)^2$. 
  Then there holds
  \begin{equation*}
    \|u - \ukh\|_{L^\infty(I;L^2(\Omega)))} \le C \left(\ln \frac{T}{k}\right)^2 
    %\lnh^{ \frac{1}{2}} 
    (k + h^2)
    \left( \|f\|_{L^\infty(I;L^2(\Omega))} + \|u_0\|_{V \cap H^2(\Omega)^2} 
    + \|(u \cdot \nabla)u\|_{L^\infty(I;L^2(\Omega))} \right),
  \end{equation*}
  where the constants $C$ depend continuously on $\|f\|_{L^2(I;L^2(\Omega))}$ and $\|u_0\|_V$.
  The last term $\|(u \cdot \nabla)u\|_{L^\infty(I;L^2(\Omega))}$ can be bounded 
  in terms of $\|f\|_{L^\infty(I;L^2(\Omega))}$ and $\|u_0\|_{V \cap H^2(\Omega)^2}$ by
  \Cref{corr:nonlinearity_linftyl2}.
\end{corollary}
\begin{proof}
  From \Cref{corr:nonlinearity_linftyl2} and \Cref{rem:initial_data_spaces},
  we obtain $(u\cdot \nabla)u \in L^\infty(I;L^2(\Omega)^2)$.
  Hence this result is a direct consequence of \Cref{thm:nav_stokes_bestapprox} and 
  \cite[Theorem 7.4]{behringer_fully_2022}.
\end{proof}
\begin{corollary}\label{corr:l2h1_orders}
  Let \Cref{ass:interpolation_operators} hold true. Further 
  let $f \in L^2(I;L^2(\Omega)^2)$, $u_0 \in V$ and let $u$, $\ukh$ be the continuous and fully discrete 
solutions to the Navier-Stokes equations \eqref{eq:nav_stokes_weak} and \eqref{eq:nav_stokes_discrete}
respectively. Then there hold the estimates
\begin{align*}
  \|\nabla(u-\ukh)\|_{L^2(I;L^2(\Omega))} & \le C \left( k^{\frac{1}{2}} + h\right), \text{ and}\\
   \|(u-\ukh)\|_{L^2(I;L^2(\Omega))} & \le C \left( k + h^2\right),
\end{align*}
where the constants $C$ depend continuously on $\|f\|_{L^2(I;L^2(\Omega))}$ and $\|u_0\|_V$.
\end{corollary}
\begin{proof}
  This result is a direct consequence of \Cref{thm:nav_stokes_l2h1}, \Cref{thm:nav_stokes_l2l2} 
  and \cite[Corollaries 6.2 \& 6.4]{vexler_l2i_2023}.
\end{proof}

\bibliographystyle{siam}
\bibliography{Quellen.bib}
%\printbibliography

%\newpage
%\makeatletter
%\providecommand\@dotsep{5}
%\makeatother
%\listoftodos\relax
%%\listoftodos

\end{document}